\documentclass[11pt,reqno]{amsart}
\usepackage[T1]{fontenc}
\usepackage{lmodern}
\usepackage{amssymb}
\usepackage{microtype}
\usepackage{needspace}
\usepackage{booktabs,longtable,array}
\usepackage{tikz}
\usepackage[hidelinks]{hyperref}
\newtheorem{theorem}{Theorem}[section]
\newtheorem{lemma}[theorem]{Lemma}
\newtheorem{proposition}[theorem]{Proposition}
\newtheorem{corollary}[theorem]{Corollary}
\theoremstyle{definition}

\newtheorem{example}[theorem]{Example}
\theoremstyle{remark}
\newtheorem{remark}[theorem]{Remark}
\numberwithin{equation}{section}
\DeclareMathOperator{\Isom}{Isom}
\DeclareMathOperator{\Sym}{Sym}
\DeclareMathOperator{\Aut}{Aut}

\DeclareMathOperator{\Dic}{Dic}

\newcommand{\Q}{\mathbb Q}
\newcommand{\Z}{\mathbb Z}

\newcommand{\F}{\mathbb F}

\newcommand{\id}{\mathrm{id}}
\newcommand{\Cyc}[1]{\mathrm C_{#1}}
\newcommand{\Ctwo}{\Cyc{2}}
\DeclareMathOperator{\rk}{rk}

\DeclareMathOperator{\Affpm}{Aff_{\pm}}

\title[Finite-color reconstruction]{Finite-color reconstruction of groups and coherent configurations}
\author{Alex J Sutherland}
\address{Oregon State University, Corvallis, Oregon, USA}
\email{suthalex@oregonstate.edu}
\thanks{ORCID: \href{https://orcid.org/0000-0002-1364-5663}{0000-0002-1364-5663}.}
\date{October 3, 2026}
\hypersetup{
 pdftitle={Finite-color reconstruction of groups and coherent configurations},
 pdfauthor={Alex J Sutherland},
 pdfsubject={Exact five-valued invariant metrics and the classification of natural groups},
 pdfkeywords={Cayley coloring, invariant metric, coherent configuration, Weisfeiler-Leman refinement, regular permutation group, natural group},
 pdfcreator={LaTeX with hyperref}}

\begin{document}
\raggedbottom

\begin{abstract}
Every group has a canonical coloring of its unordered pairs by inverse
pairs of differences. We prove that this coloring can be coarsened to four
colors without changing its color-preserving permutation group, for groups
of arbitrary cardinality. Four is the least uniform bound. A related
construction applies to finite homogeneous coherent configurations: if
every basic relation has valency at most $d$, where $d\geq2$, a coarsening
with at most $2d+1$ off-diagonal colors in the symmetric case, or $2d+3$
off-diagonal arc colors in general, has the original configuration as its
coherent closure. Two-dimensional Weisfeiler--Leman refinement recovers
that partition within $2\lceil\log_2 n\rceil$ rounds on $n$ points, with
the specified atomic initialization. Affine-plane schemes show that linear
dependence on valency is necessary. We apply the group realization to
classify groups admitting naturalizing invariant metrics and show that,
for a suitable such metric, naming the identity determines the compatible
multiplication uniquely. Finally, using additional Cayley-realization and
finite-classification inputs and explicit certificates, we determine the
least number of metric values for every group. The direct universal
construction and the finite coherent-reconstruction theorem do not depend
on those computations.
\end{abstract}
\maketitle

\section{Introduction}
Throughout, groups are abstract and we work in ZFC. An isometry is a
surjective distance-preserving map; no topology is prescribed. A metric
$d$ on $G$ is \emph{right-invariant} if
\[
 d(xg,yg)=d(x,y)\qquad\text{for all }x,y,g\in G.
\]
Writing $L(g)=d(g,e)$ gives $d(x,y)=d(xy^{-1},e)=L(xy^{-1})$;
symmetry and right invariance give $L(g^{-1})=L(g)$. For $g\in G$, put
\[
 [g]=\{g,g^{-1}\},\qquad R_g(x)=xg,\qquad R(G)=\{R_g:g\in G\}.
\]
The \emph{inverse pair} $[g]$ may be a singleton. Define
\begin{equation}\label{eq:Xi}
 \Xi(G)=\{F\in\Sym(G):[F(x)F(y)^{-1}]=[xy^{-1}]
                         \text{ for all }x,y\in G\}.
\end{equation}
Thus $R(G)\leq\Xi(G)$, and every element of $\Xi(G)$ preserves every
right-invariant metric, so a metric with isometry group exactly $\Xi(G)$
is as rigid as a right-invariant metric on $G$ can be. The results below
show that every group, of any cardinality, has such a metric with at most
five values and determine the least number of values for each group;
this extends the graphical regular representation problem from finite
Cayley graphs to all groups and addresses the cardinality-free naturality
questions of \cite[Sections~7.1--7.6]{KnillNatural}.
The canonical inverse-pair coloring assigns
$[xy^{-1}]$ to the unordered pair $\{x,y\}$ for distinct $x,y\in G$.
A map $C:G\setminus\{e\}\to\mathcal C$, with $\mathcal C$ an arbitrary
set of colors, is \emph{inverse-invariant} if $C(g)=C(g^{-1})$ for all
$g\in G\setminus\{e\}$. Its \emph{occurring colors} are
$C(G\setminus\{e\})$. A permutation $f\in\Sym(G)$ is
\emph{color-preserving} if
\[
 C(f(x)f(y)^{-1})=C(xy^{-1})
 \qquad\text{for all distinct }x,y\in G.
\]
The color-preserving group consists of all such permutations.
Every inverse-invariant coloring coarsens the canonical one, and its
color-preserving group contains $\Xi(G)$; the question is how few colors
such a coarsening can use while its color-preserving group remains
exactly $\Xi(G)$.
A metric is \emph{uniformly discrete} if its nonzero distances have a
positive lower bound.

\begin{theorem}[Sharp four-color realization]\label{thm:finite}
Every group $G$ has an inverse-invariant coloring
\[
 C:G\setminus\{e\}\longrightarrow\{1,2,3,4\}
\]
whose color-preserving group is $\Xi(G)$. Four is the least uniform
bound: already for $\Cyc{2}^3$ and for $\Cyc{3}^2$, no inverse-invariant
coloring with at most three colors has color-preserving group
$\Xi(\Cyc{2}^3)$, respectively $\Xi(\Cyc{3}^2)$. Equivalently, $G$ admits
a right-invariant metric $d$ with $\Isom(G,d)=\Xi(G)$ and at most five
values, including zero. It can be chosen bounded, uniformly discrete,
and complete, by assigning any four distinct lengths in $[1,2]$ to the
four colors.
\end{theorem}

\begin{example}[The four-color construction on $\mathbb Z/11\mathbb Z$]\label{ex:fourcolor}
Let $G=\mathbb Z/11\mathbb Z$, written additively, and color $\{\pm1\}$
differently from the other nonzero elements; this is the coloring that
the proof of Theorem~\ref{thm:finite} gives for $G$. Its
color-preserving group is the automorphism group of the $11$-cycle
$0,1,\dots,10$. Composing with a translation and, if needed, with
$x\mapsto-x$, normalize such an automorphism to fix $0$ and $1$; its
successive neighbors are then fixed, so it is the identity. Thus the
color-preserving group is the affine-sign group $\Affpm(G)$ of the maps
$x\mapsto c\pm x$, which is $\Xi(G)$ by Lemma~\ref{lem:affine}.
\end{example}

A subset is \emph{symmetric}, or \emph{inverse-closed}, if it equals its
inverse. A \emph{connection set} is a symmetric subset
$\Sigma\subseteq G\setminus\{e\}$; it need not generate $G$. The Cayley
graph $\operatorname{Cay}(G,\Sigma)$ has vertex set $G$ and edges
$\{x,sx\}$ for $x\in G$ and $s\in\Sigma$; it is a \emph{graphical regular
representation} (GRR) if its full automorphism group is $R(G)$. The finite groups admitting a GRR were determined by Watkins, Nowitz and Watkins, Imrich, Hetzel and Godsil \cite{Watkins1971,NowitzWatkins1972a,NowitzWatkins1972b,Imrich1970,Hetzel1976,Godsil1981} (see \cite[Theorem~2.1]{MorrisTymburski}). For a
finite group $G$, its \emph{Cayley index} is
\[
 \operatorname{ci}(G)=\min_{\Sigma=\Sigma^{-1}\subseteq G\setminus\{e\}}
 [\Aut(\operatorname{Cay}(G,\Sigma)):R(G)]
\]
\cite[Definitions~1.1--1.2]{MorrisTymburski}, with the translation
convention converted as in Lemma~\ref{lem:conventions}. Since $\Xi(G)$
preserves every Cayley graph of $G$,
$\operatorname{ci}(G)\geq[\Xi(G):R(G)]$. The \emph{exceptional set}
$\mathcal E$ is the set of the 24 finite isomorphism types listed in
\eqref{eq:exceptional-set}; they are exactly the finite groups $G$ with
$\operatorname{ci}(G)>[\Xi(G):R(G)]$ (Section~\ref{sec:exact}).

Let $N(G)$ be the least number of values, counting zero, of a
right-invariant metric $d$ with $\Isom(G,d)=\Xi(G)$, and let $k(G)$ be the
least number of occurring colors of an inverse-invariant coloring whose
color-preserving group is $\Xi(G)$. Lemma~\ref{lem:encoding}, and
conversely the coloring of each $g\neq e$ by $d(g,e)$, give
$N(G)=k(G)+1$. The following theorem gives the exact value of $N(G)$,
and hence of $k(G)$.

\begin{theorem}[Exact metric-value classification]\label{thm:exact}
Every group $G$ admits a right-invariant metric whose isometry group is
$\Xi(G)$ and which takes at most five values, counting zero. More precisely,
\[
 N(G)=
 \begin{cases}
 1,&G=\{e\},\\
 2,&G\cong\Cyc{2}\text{ or }\Cyc{3},\\
 3,&G\notin\mathcal E\text{ and }|G|>3,\\
 4,&G\in\mathcal E\setminus\{\Cyc{2}^3,\Cyc{3}^2\},\\
 5,&G\cong\Cyc{2}^3\text{ or }\Cyc{3}^2.
 \end{cases}
\]
In particular, every infinite group has $N(G)=3$, and the universal
constant five is sharp.
\end{theorem}

\subsection*{Naturality}
In Knill's right-translation formulation~\cite{KnillNatural}, a
right-invariant metric on $G$ is \emph{naturalizing} if every group law
on its underlying set whose right translations are isometries is
abstractly isomorphic to $G$.
A group is \emph{natural} if it admits such a metric.
Competing laws need not be abelian and may have a different identity.
An action is \emph{regular} if it is transitive and all point stabilizers
are trivial. Compatible group laws correspond to regular subgroups of
the isometry group, as detailed in Section~\ref{sec:natural}.
The earlier formulation in \cite{KnillGraphs} requires both translation
actions and inversion to be isometries; we use the right-translation
formulation above.

An \emph{involution} is an element of order exactly two. A group is
\emph{Boolean} if every element has order at most two. A
\emph{generalized dicyclic group} is a nonabelian group
$G=\langle A,t\rangle$ with $A$ abelian of index two and
\begin{equation}\label{eq:dic}
 t^2=z\in A,\qquad z\ne e,\qquad z^2=e,\qquad
 tat^{-1}=a^{-1}\quad\text{for all }a\in A.
\end{equation}
No finiteness is assumed. In particular, $t$ has order exactly four,
and $A$ is not Boolean: otherwise $t$ centralizes $A$ and $G$ is
abelian. Here the generalized dihedral group $A\rtimes_{-1}\Ctwo$ is the split
extension in which the nonidentity element of $\Ctwo$ acts on $A$ by
inversion; it is excluded by the condition on $t$.\footnote{We use order exactly four, as in
\cite[Introduction]{LeemannSalleFew}; see also the proof of
\cite[Theorem~8]{LeemannSalle}.}
The \emph{quaternion--Boolean family} consists of $Q_8\times B_0$ with
$B_0$ Boolean.
\needspace{10\baselineskip}

\begin{theorem}[Classification of natural groups]\label{thm:classification}
Let $G$ be any group.
\begin{enumerate}
\item If $G$ is abelian, written additively, then $G$ is natural if and only if
\[
 2G=G\quad\text{or}\quad 2G=\{0\}.
\]
\item If $G$ is nonabelian, then $G$ is natural if and only if it is not
 generalized dicyclic.
\end{enumerate}
Every group satisfying the corresponding condition has a naturalizing metric
taking at most five values, and five is sharp, already for $\Cyc{2}^3$.
\end{theorem}

The five-value bound in Theorem~\ref{thm:classification} comes from
Theorem~\ref{thm:finite}: on a natural group, every right-invariant metric
with isometry group $\Xi(G)$ is naturalizing
(Corollary~\ref{cor:naturalXi}). For a natural group $G$, the least
number $M(G)$ of values, counting zero, of a naturalizing right-invariant
metric therefore satisfies $M(G)\leq N(G)$, and the inequality can be
strict: $M(\Cyc{3}^2)=3<5=N(\Cyc{3}^2)$ (Proposition~\ref{prop:ternarynat}
and Remark~\ref{rem:twominima}).

The recovery idea behind Theorem~\ref{thm:finite} also yields a theorem
on finite coherent configurations. Let $\mathcal X=(X,\mathcal R)$ be a
finite homogeneous coherent configuration on $n=|X|\geq2$ points in which
every basic relation has valency at most $d\geq2$. A coarsening gives each
basic relation a single color; the diagonal keeps its own color, which is
not counted below. Theorem~\ref{thm:coherent} gives a coarsening $c$ with
at most $2d+1$ off-diagonal colors when every basic relation is symmetric,
and at most $2d+3$ off-diagonal arc colors in general, whose coherent
closure $\operatorname{WL}(c)$ is $\mathcal X$ itself; in particular
$\Aut(c)=\Aut(\mathcal X)$. Two-dimensional Weisfeiler--Leman refinement,
with transpose-aware initialization in the directed case, recovers
$\mathcal X$ from $c$ within $2\lceil\log_2 n\rceil$ rounds.

\subsection*{Idea of the proofs}
The proof of Theorem~\ref{thm:finite} starts from a retained well-order of
generators; for elements in different generator inverse pairs, the entry
rank of their difference is the larger rank (Lemma~\ref{lem:rank}).
Alternating rank bits therefore recover prescribed rank blocks
(Corollary~\ref{cor:fullcomparison}). Finite color patterns on the cosets
of a recognized anchor (Lemma~\ref{lem:cyclicpatterns}) or on signed
generators control the remaining local actions, and normalization at
every base point (Lemma~\ref{lem:normalize}) extends them to all of $G$.
Table~\ref{tab:branchmap} gives the exhaustive case division.

For Theorem~\ref{thm:coherent}, bounded valency bounds the support of
relational products, and saturation lets one selector color forget rank
without destroying incoming comparisons. The source recovery of
Lemma~\ref{lem:sources} then recovers inside $\operatorname{WL}(c)$ the
retained basic relations, which carry alternating rank labels in $c$;
every remaining basic relation is a selected relational product, the
composition of two recovered relations intersected with a distance
relation and with its own color class. Since $\operatorname{WL}(c)$ is
closed under these operations and $\mathcal X$ refines $c$, this gives
$\operatorname{WL}(c)=\mathcal X$; counting them gives the round bound of
Proposition~\ref{prop:rounds}. This is a separate output: neither theorem
is deduced from the other.

Naturality is then a question about the regular subgroups of $\Xi(G)$
(Corollary~\ref{cor:naturalXi}). Theorem~\ref{thm:XiStructure} describes
$\Xi(G)$, and the regular-subgroup theorem for involution extensions
(Theorem~\ref{thm:involext}) treats the abelian and generalized dicyclic
cases; in the latter, $\Xi(G)$ contains a regular generalized dihedral
group not isomorphic to $G$ (Proposition~\ref{prop:dic-negative}).

For Theorem~\ref{thm:exact}, two colors amount to a Cayley graph and its
complement; the infinite realization theorem
\cite[Theorem~C]{SutherlandWatkins} and the Cayley-index classification
\cite[Table~1 and Sections~2--5]{MorrisTymburski} reduce it to the 24
groups of $\mathcal E$, where two colors fail. The 22 three-color
certificates of Lemma~\ref{lem:certificates} cover all of them except
$\Cyc{2}^3$ and $\Cyc{3}^2$, for which Theorem~\ref{thm:finite} gives four
colors and hand arguments exclude three.

\subsection*{Related work}
\paragraph{Group-realization predecessors.}
The retained well-order, the collision of entry ranks at the larger
rank, alternating ordinal bits and the threshold encoding of rank
comparisons come from the involution-generated construction of Sutherland
\cite[Theorem~1.1, Sections~2--5]{SutherlandReflection}, which gives seven
metric values; its main result, on reflection groups, follows from
Theorem~\ref{thm:classification}. We use this predecessor rather than
reproduce its full proof; here these ingredients are extended to
arbitrary groups, and the additional recovery statement treats
inverse-pair and general rank blocks rather than individual involutions. The normalization and local-to-global propagation
method also occurs in \cite[Lemma~7]{LeemannSalle}. Four-color graphical
representations in the finite abelian permutation-group setting,
including Boolean-cube graphical sharpness, were established by Grech and
Kisielewicz \cite[Corollaries~1--2]{GrechKisielewicz}. The direct
construction here treats every abstract group, without a cardinality
bound, and realizes its full inverse-pair symmetry group on its original
set.

\paragraph{Known structure and naturality.}
The finite inverse-pair symmetry classification was obtained by Byrne,
Donner and Sibley \cite{ByrneDonnerSibley}; Leemann and de la Salle's
\cite[Theorem~2]{LeemannSalle} applies to arbitrary groups. We cite this
classification, transported to our right-translation convention, rather
than reprove it. The generalized-dihedral
competing law already occurs in \cite[proof of Theorem~3]{KnillNatural};
our regular-subgroup calculation makes the required action and
nonisomorphism explicit. The underlying regular-action correspondence is
classical \cite{Sabidussi}.

\paragraph{Coherent generation.}
Generation from merged basic relations is an established problem: Ziv-Av
discusses coherent closures of such partitions and coherently cyclic
schemes \cite[Section~2]{ZivAv2014}. For restricted low-class schemes,
Monzillo--Penji\'c and Yang give one-graph or one-digraph generation
results \cite[Theorem~2]{MonzilloPenjic2025}\cite[Theorem~1.4]{Yang2025}.
These results bear directly on the present question. If an adjacency
matrix $A$ generates the original Bose--Mesner algebra $\mathcal M$ as an
ordinary algebra, then the coherent algebra $\mathcal W$ generated by its
adjacency coloring satisfies
$\mathbb C[A]\subseteq\mathcal W\subseteq\mathcal M$, and hence equals
$\mathcal M$. Our bound instead depends only on maximum basic valency and
permits unbounded rank; the nonsymmetric theorem assumes neither
commutativity nor Schurity.

\paragraph{Refinement and lower bounds.}
Weisfeiler--Leman refinement goes back to Weisfeiler and Leman \cite{WeisfeilerLeman1968}, and coherent configurations to Higman \cite{Higman1975}; coherent closure and its computation by Weisfeiler--Leman refinement are
standard \cite{Skresanov2021,MuzychukPonomarenko2012}. The logarithmic
simulation of longer walks by two-dimensional refinement is established
\cite[Lemma~3]{LichterEtAl2019}; our round count concerns the specially
constructed bounded-color input, not arbitrary inputs. The affine-plane
schemes and amorphic fusion phenomenon are established
\cite{VanDam2000,VanDamKoolenXiong2025}; we include the multiplication
calculation needed for their use as a $d+2$ lower-bound family for full
coherent reconstruction.

\needspace{5\baselineskip}
\subsection*{Consequences and organization}
\begin{corollary}[Additive groups of fields]\label{cor:fields}
The additive group of every field is natural, without a cardinality
restriction. More generally, the additive group of every vector space
over $\Q$ or a prime field is natural.
\end{corollary}

Section~\ref{sec:prelim} proves the shared lemmas;
Sections~\ref{sec:coherent}, \ref{sec:finite} (with
Appendix~\ref{app:branches}), \ref{sec:natural} and \ref{sec:exact} prove
Theorems~\ref{thm:coherent}, \ref{thm:finite}, \ref{thm:classification}
and \ref{thm:exact}, Section~\ref{sec:consequences} proves
Corollary~\ref{cor:fields}, and Appendix~\ref{app:certificates} gives the
conventions of the finite certificates.

\subsection*{Notation, conventions and dependencies}
For $m\geq1$, $\Cyc{m}$ is the cyclic group of order $m$, and
$\Cyc{p^\infty}$ is the Pr\"ufer group; $p$ denotes a prime and $\F_p$
the field with $p$ elements. We reserve \emph{color} for the values of an
inverse-invariant coloring and, in Section~\ref{sec:coherent}, of a
coarsening of a coherent configuration. We write $J$ for inversion and
$\mathcal L_g$ for left translation, reserving $L$ for metric lengths.

The trivial group has the zero metric; it satisfies all metric and
naturality assertions below. Construction proofs henceforth treat nontrivial groups.

The proof of Theorem~\ref{thm:finite} uses neither the companion
realization theorem \cite[Theorem~C]{SutherlandWatkins} nor the computer
certificates. The proof of Theorem~\ref{thm:exact} uses both, together
with the finite Cayley-index classification of Morris and Tymburski
\cite{MorrisTymburski}. Section~\ref{sec:coherent} is independent of
Sections~\ref{sec:finite}--\ref{sec:exact}.

\section{Selection, ranks and normalization}\label{sec:prelim}

\subsection{Inverse-pair colorings and metric encodings}\label{sec:unavoidable}

For an inverse-invariant coloring $C:G\setminus\{e\}\to\mathcal C$ of a group $G$, the \emph{complete color-preserving group} is
\[
 \operatorname{Aut}_C(G)=\{F\in\operatorname{Sym}(G):C(F(x)F(y)^{-1})=C(xy^{-1})\text{ whenever }x\neq y\}.
\]
The right translations $R_g\in R(G)$ preserve inverse-pair differences, since $R_g(x)R_g(y)^{-1}=xy^{-1}$, so they lie in the group $\Xi(G)$ of~\eqref{eq:Xi}; every member of $\Xi(G)$ preserves the coarsenings of inverse-pair differences, which are the inverse-invariant colorings. Thus
\begin{equation}\label{eq:colorgroup}
 R(G)\leq\Xi(G)\leq\operatorname{Aut}_C(G).
\end{equation}

\begin{lemma}[Unavoidable symmetries]\label{lem:unavoidable}
For every group $G$,
\[
 \Xi(G)=\bigcap_{d\text{ right-invariant metric on }G}\Isom(G,d).
\]
\end{lemma}
\begin{proof}
A right-invariant metric has the form $d(x,y)=L(xy^{-1})$ with $L(g)=d(g,e)$, and symmetry gives $L(g)=L(g^{-1})$; hence every permutation satisfying~\eqref{eq:Xi} preserves $d$.

Conversely, if $F$ fails~\eqref{eq:Xi} at $x,y$, then $x\neq y$ and $[F(x)F(y)^{-1}]\neq[xy^{-1}]$. Color $[xy^{-1}]$ by $1$ and all other nonidentity elements by $2$; Lemma~\ref{lem:encoding} with $\mu_1=1$ and $\mu_2=2$ gives a right-invariant metric $d$ whose isometries preserve these colors, so $F\notin\Isom(G,d)$.
\end{proof}

\begin{lemma}[Encoding colors by distances]\label{lem:encoding}
Let $G$ be a group, let $\mathcal C$ be an arbitrary set of colors, and
let $C:G\setminus\{e\}\to\mathcal C$ satisfy $C(g)=C(g^{-1})$ for all
$g\in G\setminus\{e\}$.
Suppose that at most continuum many colors occur, that is,
$|C(G\setminus\{e\})|\leq\mathfrak c=|\mathbb R|$.
Choose distinct numbers $\mu_\gamma\in[1,2]$ for
$\gamma\in C(G\setminus\{e\})$ and set
\[
 L(e)=0,\qquad L(g)=\mu_{C(g)}\quad(g\neq e),\qquad
 d(x,y)=L(xy^{-1}).
\]
Then $d$ is a right-invariant metric; since its nonzero values lie in $[1,2]$, it is bounded, uniformly discrete and complete. Its isometries are exactly the
permutations preserving the pair colors $C(xy^{-1})$ for $x\neq y$.
\end{lemma}
\begin{proof}
This encoding is standard. Since $L(g)\geq1$ for $g\neq e$ and $C(g)=C(g^{-1})$, $d$ separates points and is symmetric, and $(xg)(yg)^{-1}=xy^{-1}$ gives right invariance. For distinct $x,y,z\in G$, $d(x,z)\leq2\leq d(x,y)+d(y,z)$; coincident cases are immediate. Every Cauchy sequence is eventually constant because nonzero distances are at least $1$, so $d$ is complete. As $\gamma\mapsto\mu_\gamma$ is injective, a permutation preserves $d$ exactly when it preserves the pair colors $C(xy^{-1})$ $(x\neq y)$.
\end{proof}

\begin{remark}
For infinite $G$, the canonical coloring has exactly $|G|$ occurring colors, because each nonidentity inverse pair has one or two elements. Assigning a separate distance to each pair therefore fails the hypothesis when $|G|>\mathfrak c$. The constructions below instead use finitely many colors; this is why they satisfy the hypothesis of Lemma~\ref{lem:encoding} regardless of $|G|$.
\end{remark}

\Needspace{8\baselineskip}
\begin{lemma}[Affine inverse-pair symmetries]\label{lem:affine}
For an abelian group $A$, written additively, set
\[
 \begin{aligned}
 T(A)&=\{x\mapsto c+x:c\in A\},\\
 \Affpm(A)&=\{x\mapsto c+\sigma x:c\in A,\ \sigma\in\{1,-1\}\}.
 \end{aligned}
\]
Then $\Xi(A)=\Affpm(A)$. If $2A=\{0\}$, this is just $T(A)$.
\end{lemma}
\begin{proof}
Translations and translated inversions preserve every inverse-pair color. Conversely, normalize $F\in\Xi(A)$ by a translation so that $F(0)=0$; comparison with $0$ gives $[F(x)]=[x]$, so $F(x)\in\{x,-x\}$ for every $x$, and $F=\id$ if $2A=\{0\}$. Otherwise fix $a$ with $2a\ne0$ and compose with inversion so that $F(a)=a$. If $F(x)=-x\ne x$, the condition at $a,x$ gives
\[
 a+x\in\{a-x,-a+x\},
\]
forcing $2x=0$ or $2a=0$, a contradiction. Hence the normalized map is the identity.
\end{proof}
By Lemma~\ref{lem:unavoidable}, affine sign maps therefore preserve every translation-invariant metric on $A$.

An inverse-invariant coloring $C$ of $G$ \emph{recognizes} a symmetric generating set $S$ of $G$ with $e\notin S$ when $S$ is a union of color classes of $C$.

\begin{lemma}[Normalization and word distance]\label{lem:normalize}
Let $C$ be an inverse-invariant coloring of a group $G$. For any $F\in\operatorname{Aut}_C(G)$ and any $x\in G$, the map $F_x(g)=F(gx)F(x)^{-1}$ belongs to $\operatorname{Aut}_C(G)$ and fixes $e$. If $C$ recognizes a symmetric generating set $S$, then, for word length $\ell$ with respect to $S$,
\[
 \ell(F(y)F(x)^{-1})=\ell(yx^{-1})\qquad(x,y\in G).
\]
If every $F\in\operatorname{Aut}_C(G)$ with $F(e)=e$ fixes $S$ pointwise, then $\operatorname{Aut}_C(G)=R(G)=\Xi(G)$.
\end{lemma}

\begin{proof}
Root normalization is standard \cite[Lemma~7]{LeemannSalle}, \cite[Proposition~5.1]{SutherlandReflection}. The map $F_x=R_{F(x)^{-1}}\circ F\circ R_x$ is a permutation fixing $e$, and $F_x(g)F_x(h)^{-1}=F(gx)F(hx)^{-1}$ has the color of $gh^{-1}$. Recognition of $S$ makes $F$ an automorphism of its Cayley graph, in which $x$ and $y$ are adjacent when $yx^{-1}\in S$; graph distance is $\ell(yx^{-1})$. In the pointwise case, $F_x(s)=s$ gives $F(sx)=sF(x)$ for all $s\in S$ and $x\in G$; writing $g\in G$ as a word $s_1\cdots s_n$ in $S$, which is symmetric and generates $G$, and applying this $n$ times gives $F(g)=s_1\cdots s_nF(e)=gF(e)$, so $F=R_{F(e)}\in R(G)$, and~\eqref{eq:colorgroup} gives the equalities.
\end{proof}

\subsection{Retained generators, entry ranks and source recovery}\label{sec:ranks}

Retained inverse pairs are recovered by deleting successive sources. The retained sequence, the maximum-rank comparison, the alternating ordinal bits and source deletion come from the involution-generated case \cite{SutherlandReflection}; here they are extended to arbitrary generating sets and to rank blocks of arbitrary size.

Assume that $G\neq\{e\}$, and well-order a generating set $T\subseteq G\setminus\{e\}$. Retain generators by the recursion
\[
 G_\alpha=\langle x_\beta:\beta<\alpha\rangle,\qquad x_\alpha=\min(T\setminus G_\alpha)\quad\text{if }G_\alpha\neq G.
\]
Since $T$ generates $G$, a choice exists while $G_\alpha\neq G$; the choices are distinct elements of $T$, so the recursion stops at an ordinal $\tau<|T|^+$, and
\[
 G_\tau=G,\qquad x_\alpha\notin G_\alpha\quad(\alpha<\tau).
\]
Set
\[
 \mathcal X=\{x_\alpha:\alpha<\tau\},\qquad \mathcal S=\mathcal X\cup\mathcal X^{-1},\qquad \mathcal B_\alpha=[x_\alpha].
\]
The sets $\mathcal B_\alpha$ are pairwise disjoint by the retention condition at the larger index, and they partition $\mathcal S$.

At every nonzero limit $\lambda\leq\tau$, finite support of group words gives $G_\lambda=\bigcup_{\alpha<\lambda}G_\alpha$. As $G_0=\{e\}$, the least stage containing $g\neq e$ is therefore a successor, which defines the entry rank $\rk(g)<\tau$ by
\[
 g\in G_{\rk(g)+1}\setminus G_{\rk(g)}.
\]
Subgroup closure gives the following identities.

\begin{lemma}[Rank identities]\label{lem:rank}
For $g\in G\setminus\{e\}$, and for $u,v\in G\setminus\{e\}$ with $\rk(u)\neq\rk(v)$,
\begin{equation}\label{eq:rankrule}
 \rk(g^{-1})=\rk(g),\qquad \rk(uv^{-1})=\max\{\rk(u),\rk(v)\}.
\end{equation}
In particular, $\rk(u)=\alpha$ for $u\in\mathcal B_\alpha$.
\end{lemma}
\begin{proof}
Each $G_\alpha$ is a subgroup, which gives the inverse identity and $\rk(x_\alpha^{\pm1})=\alpha$. If $\rk(u)<\rk(v)=\beta$, then $uv^{-1}\in G_{\beta+1}$, and $uv^{-1}\in G_\beta$ would give $v^{-1}=u^{-1}(uv^{-1})\in G_\beta$; the case $\rk(u)>\rk(v)$ follows by inverting $uv^{-1}$.
\end{proof}

Define ordinal bits by
\[
 \epsilon_0=0,\qquad \epsilon_{\alpha+1}=1-\epsilon_\alpha,\qquad \epsilon_\lambda=0\quad\text{at nonzero limits},
\]
and put $\pi(g)=\epsilon_{\rk(g)}$ for $g\neq e$. Lemma~\ref{lem:rank} implies that $\pi$ is inverse-invariant. Let $\ell$ denote word length for $\mathcal S$; then $\ell(g^{-1})=\ell(g)$, and $\ell(g)=1$ exactly when $g\in\mathcal S$.

\begin{lemma}[Recovery by successive sources]\label{lem:sources}
Let $D=\bigsqcup_{\alpha<\tau}D_\alpha$ be partitioned into nonempty blocks. Give every vertex of $D_\alpha$ a bit $\epsilon_\alpha$, where consecutive bits differ. Among pairs with opposite vertex bits, some are labeled by the bit of their larger-index endpoint, and the remaining pairs are unlabeled. Suppose that for $\alpha<\beta$ with opposite bits, every vertex of $D_\beta$ has a labeled edge to $D_\alpha$. Every permutation of $D$ preserving the vertex bits, the edge labels, and the absence of labels fixes each $D_\alpha$ setwise.
\end{lemma}

\begin{proof}
Use labeled edges only between opposite vertex bits, and direct $u$ to $v$ if the edge label is the bit of $v$ (Figure~\ref{fig:sources}). Every arrow increases the block index. In the tail $D_{\geq\gamma}=\bigcup_{\alpha\geq\gamma}D_\alpha$, the source set is exactly $D_\gamma$. Its vertices have no incoming arrows. If $v\in D_\beta$ with $\beta>\gamma$, then $D_\gamma$ supplies an incoming arrow when its bit differs from that of $v$. When the bits agree, $\gamma+1<\beta$ and the opposite-bit block $D_{\gamma+1}$ supplies an incoming arrow. Thus no later vertex is a source. A preserving permutation fixes these source sets successively by transfinite induction. At limit stages the union of the previously fixed blocks is invariant, so the remaining tail is invariant as well.
\end{proof}

\begin{figure}[ht]
\centering
\begin{tikzpicture}[x=2.6cm,y=1cm]
 \node[draw,rounded corners,fill=gray!20,minimum width=1.5cm,minimum height=0.9cm] (d0) at (0,0) {$D_0$};
 \node[draw,rounded corners,minimum width=1.5cm,minimum height=0.9cm] (d1) at (1,0) {$D_1$};
 \node[draw,rounded corners,minimum width=1.5cm,minimum height=0.9cm] (d2) at (2,0) {$D_2$};
 \node[draw,rounded corners,minimum width=1.5cm,minimum height=0.9cm] (d3) at (3,0) {$D_3$};
 \node at (0,-0.9) {$\epsilon_0=0$};
 \node at (1,-0.9) {$\epsilon_1=1$};
 \node at (2,-0.9) {$\epsilon_2=0$};
 \node at (3,-0.9) {$\epsilon_3=1$};
 \draw[->,thick] (d0) -- (d1);
 \draw[->,thick] (d1) -- (d2);
 \draw[->,thick] (d2) -- (d3);
 \draw[->,thick] (d0.north) to[bend left=25] (d3.north);
\end{tikzpicture}
\caption{Recovery by successive sources (Lemma~\ref{lem:sources}) for four blocks with alternating bits. An arrow from $D_\alpha$ to $D_\beta$ means that every vertex of $D_\beta$ receives an arrow from $D_\alpha$; blocks of equal bit are not joined. The shaded first block is the source set. Deleting it leaves $D_1$ as the source set of the remaining tail.}
\label{fig:sources}
\end{figure}

\begin{corollary}[Full comparison labels]\label{cor:fullcomparison}
For each $\alpha<\tau$, let $D_\alpha$ be a nonempty inverse-invariant subset of $G_{\alpha+1}\setminus G_\alpha$, and put $D=\bigcup_{\alpha<\tau}D_\alpha$. The four-color map
\[
 C(g)=1+\pi(g)+2\cdot\mathbf 1_D(g)\qquad(g\neq e)
\]
has the property that every rooted complete color preserver fixes each $D_\alpha$ setwise. More generally, any coloring recognizing $D$ and its vertex bits and decoding the true rank bit on every difference between its vertices has the same property.
\end{corollary}

\begin{proof}
The colors recognize $D$ and decode its vertex bits and all difference bits. For distinct rank blocks, the difference rank is the larger rank by Lemma~\ref{lem:rank}. Apply Lemma~\ref{lem:sources}, with all comparison edges labeled.
\end{proof}

\begin{lemma}[Rank-block protocol]\label{lem:protocol}
For a permutation $F$ of $G$, the set $Y_F$ of $g\in G$ with $F(gy)=gF(y)$ for all $y\in G$ is a subgroup, and $F=R_{F(e)}$ when $Y_F$ contains a generating set of $G$. Now let $D_\alpha$ $(\alpha<\tau)$ be nonempty inverse-invariant subsets of $G_{\alpha+1}\setminus G_\alpha$, which may be infinite, put $D=\bigcup_{\alpha<\tau}D_\alpha$, let $C$ be an inverse-invariant coloring of $G$ whose rooted complete color preservers fix each $D_\alpha$ setwise, such as the map $C_D(g)=1+\pi(g)+2\cdot\mathbf 1_D(g)$ of Corollary~\ref{cor:fullcomparison}, and let $F\in\operatorname{Aut}_C(G)$. Then:
(i) every normalization $F_x(u)=F(ux)F(x)^{-1}$ $(x\in G)$ fixes each $D_\alpha$ setwise;
(ii) if $D_\alpha=[a]$, then $F$ is an automorphism of the graph on $G$ with edges $\{y,ay\}$ $(y\in G)$;
(iii) if $D_\alpha=\{h\}$, then $h\in Y_F$;
(iv) if $F_x(u)\in[u]$ for all $u,x\in G$, then $F\in\Xi(G)$.
Consequently $\operatorname{Aut}_C(G)=R(G)=\Xi(G)$ if $Y_F=G$ for every $F\in\operatorname{Aut}_C(G)$, and $\operatorname{Aut}_C(G)=\Xi(G)$ if every $F\in\operatorname{Aut}_C(G)$ satisfies the hypothesis of~(iv).
\end{lemma}

\begin{proof}
If $g,k\in Y_F$, then $F(gky)=gF(ky)=gkF(y)$, and $F(y)=gF(g^{-1}y)$ gives $g^{-1}\in Y_F$; as $e\in Y_F$, this is a subgroup, and if it is $G$ then $F(g)=gF(e)$ for all $g$. By Lemma~\ref{lem:normalize}, each $F_x$ is a complete color preserver fixing $e$, which gives~(i). If $D_\alpha=[a]$, then (i) gives $F(ax)\in\{aF(x),a^{-1}F(x)\}$, so $F$, and likewise $F^{-1}\in\operatorname{Aut}_C(G)$, maps edges to edges: this is~(ii). If $D_\alpha=\{h\}$, then $F_x(h)=h$ gives $F(hx)=hF(x)$ for every $x$: this is~(iii). For~(iv), $u=yx^{-1}$ gives $[F(y)F(x)^{-1}]=[yx^{-1}]$ for all $x,y\in G$, which is~\eqref{eq:Xi}; the two equalities follow from~\eqref{eq:colorgroup}.
\end{proof}

The signed-generator graph constructions of Theorem~\ref{thm:abelianmin}, Lemma~\ref{lem:commonsquare} and Theorem~\ref{thm:residual} use their own colorings, not the rank-block coloring $C_D$.

\begin{lemma}[Oriented-coset transfer]\label{lem:cosettransfer}
Let $C$ be an inverse-invariant coloring of a group $G$ generated by an element $a$ and a nonempty family $(b_\beta)$, and let $F\in\operatorname{Aut}_C(G)$. If every normalization $F_x$ fixes each coset $\langle a\rangle b_\beta$ pointwise, then $F=R_{F(e)}$.
\end{lemma}

\begin{proof}
For every $x\in G$, $F_x(b_\beta)=b_\beta$ and $F_x(ab_\beta)=ab_\beta$ say that $F(b_\beta x)=b_\beta F(x)$ and $F(ab_\beta x)=ab_\beta F(x)$, so $b_\beta$ and $ab_\beta$ lie in the subgroup $Y_F$ of Lemma~\ref{lem:protocol}. For one $\beta$, so does $a=(ab_\beta)b_\beta^{-1}$. As $a$ and the $b_\beta$ generate $G$, Lemma~\ref{lem:protocol} gives $F=R_{F(e)}$.
\end{proof}

\subsection{Two selector lemmas}

\begin{lemma}[Selected witnesses in a bounded cover]\label{lem:cover}
Let $V$ be a set, let $q$ be a positive integer, and let $\mathcal B$ be a family of nonempty subsets of $V$, each of size at most $q$, whose union is $V$. There are a map $\chi:V\to\{1,\ldots,q\}$ and, for each $v\in V$, a set $B_v\in\mathcal B$ with $v\in B_v$, such that
\[
 w\in B_v\setminus\{v\}\quad\Longrightarrow\quad\chi(w)\neq\chi(v).
\]
\end{lemma}

\begin{proof}
Well-order $\mathcal B$ as $(B_\alpha)_{\alpha<\sigma}$ and assign each point to its first containing set; the points assigned at stage $\alpha$ are
\[
 A_\alpha=B_\alpha\setminus\bigcup_{\beta<\alpha}B_\beta.
\]
By transfinite recursion, label the points of $A_\alpha$ distinctly, avoiding the labels of the $r$ earlier points of $B_\alpha$; with $s=|A_\alpha|$ we have $r+s\leq q$, so $s$ labels remain. For $v\in A_\alpha$ put $B_v=B_\alpha$; every other point of $B_v$ is new at stage $\alpha$ or earlier, and in both cases its label differs from $\chi(v)$. Only the selected set $B_v$ is controlled: a member processed later may contain earlier points with equal labels, so the members of $\mathcal B$ need not be rainbow.
\end{proof}

\begin{lemma}[One sparse reserved label]\label{lem:star}
Let $V$, $q$ and $\mathcal B$ be as in Lemma~\ref{lem:cover}, with $q\geq2$, and let $*$ be a label outside $\{1,\ldots,q-1\}$. The conclusion of Lemma~\ref{lem:cover} holds with $\chi$ taking values in $\{1,\ldots,q-1,*\}$ and such that every member of $\mathcal B$ contains at most one point with label $*$, and every member of $\mathcal B$ of size less than $q$ contains none.
\end{lemma}

\begin{proof}
Call $v\in V$ \emph{hard} if every member of $\mathcal B$ containing $v$ has size $q$. Join two distinct hard points when some member of $\mathcal B$ contains both, and choose a maximal independent set $I$ of hard points in this graph (by Zorn's lemma). Then every member of $\mathcal B$ meets $I$ at most once, every member of size less than $q$ misses $I$, and a member $B$ meeting $I$ has $|B\setminus I|=q-1\geq1$.

The members of size less than $q$ and the sets $B\setminus I$ for members $B$ meeting $I$ form a family $\mathcal B'$ of nonempty subsets of $V\setminus I$ of size at most $q-1$. It covers $V\setminus I$: a point that is not hard lies in a member of size less than $q$, and by maximality a hard point $v\notin I$ shares a member $B$ with a point of $I$, so $v\in B\setminus I$. Apply Lemma~\ref{lem:cover} to $\mathcal B'$ with $q-1$ labels and put $\chi=*$ on $I$. For $v\notin I$ take $B_v\in\mathcal B$ with $B'_v\in\{B_v,B_v\setminus I\}$; the points of $B_v\setminus B'_v$ lie in $I$ and have label $*\neq\chi(v)$. For $v\in I$, any member of $\mathcal B$ containing $v$ serves as $B_v$, since it contains no other point of $I$.
\end{proof}

\section{Finite coherent reconstruction}\label{sec:coherent}

Throughout this section $\mathcal{X}=(X,\mathcal{R})$ is a finite coherent configuration~\cite{Higman1975} with $n=|X|\ge2$: the set $\mathcal{R}$ of \emph{basic relations} partitions $X^2$ and is closed under the transpose $r\mapsto r^*$, and for $r,s,t\in\mathcal{R}$ the number $p^t_{rs}$ of points $z$ with $(x,z)\in r$ and $(z,y)\in s$ is the same for every $(x,y)\in t$. We assume throughout that $\mathcal{X}$ is \emph{homogeneous}: the diagonal $\Delta_X$ is a basic relation. For $r\in\mathcal{R}$ and $x\in X$ put $xr=\{y:(x,y)\in r\}$; its size $k_r=p^{\Delta_X}_{rr^*}$ does not depend on $x$ and is the \emph{valency} of $r$. The configuration is \emph{symmetric} if every basic relation is symmetric, and \emph{thin} if every valency is $1$. The letter $d$ always denotes a bound on the maximum basic valency.

A \emph{coarsening} of $\mathcal{X}$ is a coloring $c$ of $X^2$ that is constant on every basic relation and gives $\Delta_X$ a color of its own; it splits no basic relation and adds no points. Palette sizes count only the off-diagonal colors; the diagonal color is one more. A coarsening need not satisfy $c(x,y)=c(y,x)$ (it is then an \emph{arc coloring}), and it need not be coherent. Its \emph{coherent closure} $\operatorname{WL}(c)$ is the coarsest coherent configuration on $X$ each of whose basic relations lies in one color class of $c$; it is computed by two-dimensional Weisfeiler--Leman refinement \cite{Skresanov2021,MuzychukPonomarenko2012}. The equation $\operatorname{WL}(c)=\mathcal{X}$ means that the two partitions of $X^2$ coincide. We write $\operatorname{Aut}(\mathcal{X})$ for the group of permutations of $X$ fixing every basic relation, and $\operatorname{Aut}(c)$ for the group of permutations $g$ of $X$ with $c(g(x),g(y))=c(x,y)$ for all $x,y\in X$. Generating a coherent configuration from merged basic relations is an established problem \cite[Section~2]{ZivAv2014}, and Monzillo--Penji\'c and Yang show that a single graph or digraph generates certain three-class and four-class schemes \cite[Theorem~2]{MonzilloPenjic2025}, \cite[Theorem~1.4]{Yang2025}. The bounds below depend only on $d$, not on the rank.

\begin{theorem}[Finite coherent reconstruction]\label{thm:coherent}
Let $\mathcal{X}$ be a finite homogeneous coherent configuration in which every basic relation has valency at most $d$.
\begin{enumerate}
\item[(i)] If $d\ge2$ and $\mathcal{X}$ is symmetric, there is a symmetric coarsening $c$ of $\mathcal{X}$ with at most $2d+1$ off-diagonal colors and $\operatorname{WL}(c)=\mathcal{X}$.
\item[(ii)] If $d\ge2$, there is a coarsening $c$ of $\mathcal{X}$ with at most $2d+3$ off-diagonal arc colors and $\operatorname{WL}(c)=\mathcal{X}$.
\item[(iii)] If $d=1$, so that $\mathcal{X}$ is thin, the conclusion of (i) holds when $\mathcal{X}$ is symmetric and the conclusion of (ii) holds, with four colors in place of $2d+1$ and six in place of $2d+3$.
\end{enumerate}
In each case $\operatorname{Aut}(c)=\operatorname{Aut}(\mathcal{X})$.
\end{theorem}

For relations $r,s\subseteq X^2$ write $r\circ s=\{(x,y):(x,z)\in r\text{ and }(z,y)\in s\text{ for some }z\in X\}$, and for $r,s\in\mathcal{R}$ let $rs=\{t\in\mathcal{R}:t\subseteq r\circ s\}$ be the \emph{support} of the product.

Let $\mathcal{Y}$ be a coherent configuration on $X$. If $(x,y)\in r\circ s$ for basic relations $r,s$ of $\mathcal{Y}$ and $(x,y)\in t$, then $p^t_{rs}>0$ and $t\subseteq r\circ s$. So the unions of basic relations of $\mathcal{Y}$, which we call \emph{$\mathcal{Y}$-relations}, are closed under composition, union, intersection, difference and transpose.

\begin{lemma}[Product support and saturation]\label{lem:support}
Let $\mathcal{X}$ be a finite homogeneous coherent configuration and $r,s\in\mathcal{R}$.
\begin{enumerate}
\item[(i)] $r\circ s$ is the union of the relations in $rs$, and $|rs|\le\min(k_r,k_s)$.
\item[(ii)] For every $(x,y)\in s$, the sets $Z_t=\{z\in xr^*:(z,y)\in t\}$ with $t\in rs$ are nonempty and partition $xr^*$.
\item[(iii)] If every valency is at most $d$ and $|rs|=d$, then $k_r=d$ and every set $Z_t$ in (ii) is a singleton.
\end{enumerate}
\end{lemma}

\begin{proof}
If $(x,y)\in r\circ s$ and $(x,y)\in t\in\mathcal{R}$, then $p^t_{rs}>0$, so every pair of $t$ lies in $r\circ s$ and $t\in rs$; hence $r\circ s$ is the union of the relations in $rs$. Fix $(x,y)\in s$. A point $z\in xr^*$ satisfies $(z,x)\in r$, so $(z,y)\in r\circ s$, and the basic relation containing $(z,y)$ lies in $rs$; thus the sets $Z_t$ partition $xr^*$. Given $t\in rs$, choose $(a,b)\in t$ and $w$ with $(a,w)\in r$ and $(w,b)\in s$. The point $a$ satisfies $(w,a)\in r^*$ and $(a,b)\in t$, so $p^s_{r^*t}>0$ at the pair $(w,b)\in s$; by coherence $|Z_t|=p^s_{r^*t}>0$ at $(x,y)$ as well. This proves (ii) and gives $|rs|\le|xr^*|=k_{r^*}$. Counting the pairs of $r$ by their first and by their second coordinates gives $nk_r=|r|=nk_{r^*}$, so $k_{r^*}=k_r$. Since $t\mapsto t^*$ maps $rs$ bijectively onto $s^*r^*$, also $|rs|=|s^*r^*|\le k_s$, which completes (i). Under the hypothesis of (iii), the $d$ nonempty sets $Z_t$ partition $xr^*$, which has at most $d$ points; hence $k_r=k_{r^*}=d$ and every $Z_t$ is a singleton.
\end{proof}

A saturation criterion closely related to Lemma~\ref{lem:support} appears in Hirasaka, Kim and Ponomarenko~\cite{HirasakaKimPonomarenko2020}, building on Chen and Ponomarenko~\cite{ChenPonomarenko2017}.

\emph{Retained relations, distances and boxes.} Put $E_0=\Delta_X$. While $E_i\ne X^2$, choose a \emph{retained} relation $s_i\in\mathcal{R}$ with $s_i\not\subseteq E_i$, put $P_i=\{s_i,s_i^*\}$ (a single relation when $s_i$ is symmetric), and let $E_{i+1}$ be the equivalence relation generated by $E_i\cup s_i\cup s_i^*$; the process ends with $E_t=X^2$, and any such sequence can be used. The equivalence relation generated by a reflexive symmetric $\mathcal{X}$-relation $U$ is $\bigcup_{m\ge0}U^{\circ m}$, so every $E_i$ is an $\mathcal{X}$-relation and $s_i\cap E_i=\emptyset$. Hence each $r\in\mathcal{R}\setminus\{\Delta_X\}$ lies in $E_{i+1}\setminus E_i$ for exactly one $i$, its \emph{rank} $\rho(r)=i$; write $\rho(x,y)=\rho(r)$ for $(x,y)\in r$, and $\epsilon_i=i\bmod 2$. If $\rho(x,y)=i<j=\rho(x,z)$, then $(x,y)\in E_j$ and $(x,z)\in E_{j+1}\setminus E_j$, so $(y,z)\in E_{j+1}\setminus E_j$. Hence
\begin{equation}\label{eq:cohmax}
\rho(x,y)\ne\rho(x,z)\quad\Longrightarrow\quad\rho(y,z)=\max\{\rho(x,y),\rho(x,z)\}.
\end{equation}
For $x\in X$ the \emph{block} $xP_i=xs_i\cup xs_i^*$ is nonempty. For $0\le j<t$ and $\epsilon\in\{0,1\}$ let $T^\epsilon$ be the union of the relations $s_i\cup s_i^*$ with $i\ge j$ and $\epsilon_i=\epsilon$ (these relations depend on $j$).

Let $S=\bigcup_i(s_i\cup s_i^*)$. It is symmetric and generates the equivalence relation $E_t=X^2$, so the graph $(X,S)$ is connected; let $D$ be its diameter. The relation of distance at most $m$ is $(\Delta_X\cup S)^{\circ m}$, an $\mathcal{X}$-relation, so the distance $\ell(r)$ is constant on each $r\in\mathcal{R}$. Let $\Delta_m$ be the relation of distance exactly $m$ (so $\Delta_0=\Delta_X$) and $\mathcal{R}_m=\{r\in\mathcal{R}:\ell(r)=m\}$; thus $\mathcal{R}_1=\bigcup_iP_i$. For $2\le m\le D$ call each nonempty set
\begin{equation}\label{eq:cohboxes}
B(u,v)=uv\cap\mathcal{R}_m\qquad(u\in\mathcal{R}_{\lfloor m/2\rfloor},\ v\in\mathcal{R}_{\lceil m/2\rceil})
\end{equation}
a \emph{box} of distance $m$; here $u$ and $v$ are individual basic relations, not transpose pairs. By Lemma~\ref{lem:support}(i), every box has at most $d$ elements. The boxes of distance $m$ cover $\mathcal{R}_m$: for $(x,y)\in r\in\mathcal{R}_m$, let $z$ be the point at distance $\lfloor m/2\rfloor$ from $x$ on a shortest path from $x$ to $y$; both subpaths are shortest, so the relations $u\ni(x,z)$ and $v\ni(z,y)$ satisfy $\ell(u)=\lfloor m/2\rfloor$ and $\ell(v)=\lceil m/2\rceil$, and $r\in B(u,v)$ by Lemma~\ref{lem:support}(i).

\begin{lemma}[The designed coarsening]\label{lem:cohcoloring}
Let $\mathcal{X}$ and $d$ be as in Theorem~\ref{thm:coherent}, with retained relations $s_0,\ldots,s_{t-1}$ as above, and put $\bar d=\max(d,2)$. Put $h=1$ in the \emph{symmetric case}, where $\mathcal{X}$ is symmetric and a symmetric coloring is sought (Theorem~\ref{thm:coherent}(i), and (iii) with four colors), and $h=2$ in the \emph{arc case} (Theorem~\ref{thm:coherent}(ii), and (iii) with six colors). The \emph{orientation tag} $o$ on $\bigcup_iP_i$ is $0$ if $h=1$; if $h=2$, then $o(s_i)=0$ and, when $s_i\ne s_i^*$, $o(s_i^*)=1$. There is a coarsening $c$ of $\mathcal{X}$, given by~\eqref{eq:cohsym} below and symmetric in the symmetric case, with at most $2h+2d-1$ off-diagonal colors (that is, $2d+1$, respectively $2d+3$) when $d\ge2$ and at most $2h+2$ (four, respectively six) when $d=1$, such that the following hold, where for $\epsilon\in\{0,1\}$
\[
L^\epsilon=c^{-1}\bigl(\{1+h\epsilon,\ldots,h+h\epsilon\}\cup\bigl(2h+(\bar d-1)\epsilon+\{1,\ldots,\bar d-1\}\bigr)\bigr).
\]
\begin{enumerate}
\item[(a)] Every $r\in P_i$ has $c(r)=1+h\epsilon_i+o(r)$, and every other off-diagonal color exceeds $2h$. Hence $S=c^{-1}(\{1,\ldots,2h\})$ and, for $j=0$, $T^0=c^{-1}(\{1,\ldots,h\})$ and $T^1=c^{-1}(\{h+1,\ldots,2h\})$.
\item[(b)] For every $0\le j<t$,
\begin{equation}\label{eq:cohincoming}
I=\bigl((T^0\circ L^1)\cap T^1\bigr)\cup\bigl((T^1\circ L^0)\cap T^0\bigr),\qquad s_j\cup s_j^*=(T^0\cup T^1)\setminus I,
\end{equation}
where in fact $I=\bigcup_{i>j}(s_i\cup s_i^*)$. In the arc case, $s_j=(s_j\cup s_j^*)\cap c^{-1}(1+2\epsilon_j)$ and, when $s_j\ne s_j^*$, $s_j^*=(s_j\cup s_j^*)\cap c^{-1}(2+2\epsilon_j)$.
\item[(c)] Every $r\in\mathcal{R}$ with $\ell(r)=m\ge2$ has basic relations $u_r,v_r$ with $\ell(u_r)=\lfloor m/2\rfloor$ and $\ell(v_r)=\lceil m/2\rceil$ such that
\begin{equation}\label{eq:cohrecover}
r=(u_r\circ v_r)\cap\Delta_{\ell(r)}\cap c^{-1}(c(r)).
\end{equation}
\end{enumerate}
\end{lemma}

\begin{proof}[Proof of Lemma~\ref{lem:cohcoloring}]

\emph{Selection.} If $d\ge2$, Lemma~\ref{lem:star} with $q=d$, applied to the cover of $\mathcal{R}_2$ by the boxes of distance $2$, gives a set $\Sigma\subseteq\mathcal{R}_2$ of \emph{starred} relations, meeting every box at most once and no box with fewer than $d$ elements, and labels $\chi_2(r)\in\{1,\ldots,d-1\}$ for $r\in\mathcal{R}_2\setminus\Sigma$; every $r\in\mathcal{R}_2$ lies in a selected box $B(u_r,v_r)$ in which no other relation carries the label of $r$ (the star counting as one label). For $3\le m\le D$, Lemma~\ref{lem:cover} with $q=d$ gives labels $\chi_m(r)\in\{1,\ldots,d\}$ and, for each $r\in\mathcal{R}_m$, a selected box $B(u_r,v_r)$ in which $r$ alone has label $\chi_m(r)$. If $d=1$, every box is a singleton by Lemma~\ref{lem:support}(i); put $\Sigma=\emptyset$ and $\chi_m(r)=1$ for all $m\ge2$, and let the selected box of $r\in\mathcal{R}_m$ be a box $B(u_r,v_r)=\{r\}$, which exists because the boxes cover $\mathcal{R}_m$.

\emph{The coloring.} With $\bar d$ and $o$ as in the statement, define
\begin{equation}\label{eq:cohsym}
c(r)=\begin{cases}
1+h\epsilon_i+o(r), & r\in P_i,\\
2h+2\bar d-1, & r\in\Sigma,\\
2h+(\bar d-1)\epsilon_{\rho(r)}+\chi_2(r), & \ell(r)=2,\ r\notin\Sigma,\\
2h+\chi_{\ell(r)}(r), & \ell(r)\ge3.
\end{cases}
\end{equation}
and give $\Delta_X$ the color $0$, which lies outside every palette used. As $\mathcal{R}_1=\bigcup_iP_i$ and $\Sigma\subseteq\mathcal{R}_2$, each off-diagonal basic relation falls under exactly one clause, so $c$ is a coarsening of $\mathcal{X}$; in the symmetric case $o=0$ and every basic relation is symmetric, so $c$ is symmetric. As $0\le o(r)\le h-1$, the seed colors lie in $\{1,\ldots,h\}$ for even $i$ and in $\{h+1,\ldots,2h\}$ for odd $i$, and all other values exceed $2h$, which proves (a). If $d\ge2$, then $\bar d=d$ and all values lie in $\{1,\ldots,2h+2d-1\}$; if $d=1$, then $\Sigma=\emptyset$ and all labels are $1$, so they lie in $\{1,\ldots,2h+2\}$. The two color sets defining $L^0$ and $L^1$ are disjoint and omit the star color $2h+2\bar d-1$.

\emph{Decoding.} Let $x\in X$, $y\in xP_i$ and $z\in xP_k$ with $i\ne k$, and $j=\max\{i,k\}$. Then $y\ne z$, $\rho(y,z)=j$ by~\eqref{eq:cohmax}, and $\ell(y,z)\le2$ as $(y,x),(x,z)\in S$. As $\mathcal{R}_1=\bigcup_iP_i$, \eqref{eq:cohsym} puts the color of an unstarred pair $(y,z)$ in $\{1+h\epsilon_j,\ldots,h+h\epsilon_j\}$ or in $2h+(\bar d-1)\epsilon_j+\{1,\ldots,\bar d-1\}$, so $(y,z)$ lies in $L^{\epsilon_j}$ and not in $L^{1-\epsilon_j}$.

\emph{Incoming comparisons.} Let $x\in X$, $i<k$ and $z\in xP_k$; we show that some $y\in xP_i$ has $(y,z)$ unstarred. Choose $u\in P_i$ and let $v\in P_k$ contain $(x,z)$. By Lemma~\ref{lem:support}(ii) with $r=u^*$ and $s=v$, the sets $Z_t=\{y\in xu:(y,z)\in t\}$ with $t\in u^*v$ are nonempty and partition $xu\ne\emptyset$. If no relation of $u^*v$ is starred, we are done. Otherwise a starred $t$ lies in $u^*v\cap\mathcal{R}_2=B(u^*,v)$, a box of distance $2$; so this box has exactly $d$ elements and contains no other starred relation. Then $d\le|u^*v|\le d$, and Lemma~\ref{lem:support}(iii) gives $|xu|=d$ and $|Z_t|=1$ for every $t\in u^*v$. Exactly one $y\in xu$ has $(y,z)$ starred, and since $d\ge2$ another $y\in xu$ has not.

\emph{Source deletion.} Let $0\le j<t$, with $T^0,T^1$ for this $j$ and $I$ as in~\eqref{eq:cohincoming}. 
Take $(x,z)\in T^\epsilon$ with $z\in xP_k$, $k\ge j$, $\epsilon_k=\epsilon$. If $(x,z)\in I$, some $y\in xP_i$ with $i\ge j$ and $\epsilon_i\ne\epsilon$ has $(y,z)\in L^\epsilon$; then $i\ne k$, the pair $(y,z)$ is unstarred as $L^\epsilon$ omits the star color, and decoding gives $\epsilon_{\max\{i,k\}}=\epsilon\ne\epsilon_i$, so $k>i\ge j$. Conversely, if $k>j$, let $i=j$ if $\epsilon_j\ne\epsilon_k$ and $i=j+1$ otherwise (then $k\ge j+2$), so that $j\le i<k$ and $\epsilon_i\ne\epsilon_k$; the incoming comparisons and decoding give $y\in xP_i$ with $(y,z)\in L^\epsilon$, and $(x,y)\in T^{1-\epsilon}$, so $(x,z)\in I$. Hence $I=\bigcup_{i>j}(s_i\cup s_i^*)$, and $s_j\cup s_j^*=(T^0\cup T^1)\setminus I$ as $T^0\cup T^1=\bigcup_{i\ge j}(s_i\cup s_i^*)$. In the arc case the orientation tags give $c(s_j)=1+2\epsilon_j$ and, when $s_j\ne s_j^*$, $c(s_j^*)=2+2\epsilon_j$. This proves (b).

\emph{Selected products.} Let $r\in\mathcal{R}_m$, $m\ge2$, with selected box $B(u_r,v_r)$; by~\eqref{eq:cohboxes} and Lemma~\ref{lem:support}(i), $\ell(u_r)=\lfloor m/2\rfloor$, $\ell(v_r)=\lceil m/2\rceil$, and the right side of~\eqref{eq:cohrecover} is the union of the relations $t$ of this box with $c(t)=c(r)$. At one distance equal colors give equal labels: for $m\ge3$ the color is $2h+\chi_m(t)$, and for $m=2$ the star has its own color, while an unstarred color $2h+(\bar d-1)\epsilon_{\rho(t)}+\chi_2(t)$ determines $\chi_2(t)\in\{1,\ldots,\bar d-1\}$, whatever the ranks in the box. So $t=r$ by the choice of the selected box, which proves (c).

\end{proof}

\begin{proof}[Proof of Theorem~\ref{thm:coherent}]

Let $c$ be the coarsening of Lemma~\ref{lem:cohcoloring}, which has the symmetry and the number of colors required in (i), (ii) and (iii). The $\operatorname{WL}(c)$-relations are closed under composition, union, intersection and difference, and they include every color class of $c$ (each basic relation of $\operatorname{WL}(c)$ lies in one), hence $L^0$, $L^1$ and, by Lemma~\ref{lem:cohcoloring}(a), $S$ and $T^0,T^1$ for $j=0$. Passing from $j$ to $j+1$ removes $s_j\cup s_j^*$ from $T^0\cup T^1$, so by Lemma~\ref{lem:cohcoloring}(b) and induction on $j$ every $s_j\cup s_j^*$ is a $\operatorname{WL}(c)$-relation, and in the arc case so are $s_j$ and $s_j^*$, cut out by color classes. Thus every relation of $\mathcal{R}_1=\bigcup_iP_i$ and every $\Delta_m=(\Delta_X\cup S)^{\circ m}\setminus(\Delta_X\cup S)^{\circ(m-1)}$ is a $\operatorname{WL}(c)$-relation, and by Lemma~\ref{lem:cohcoloring}(c) and induction on $m\ge2$ so is every $r\in\mathcal{R}$ with $\ell(r)=m$, as $1\le\ell(u_r),\ell(v_r)<m$. Since $\Delta_X$ is a color class and $(X,S)$ is connected, every basic relation of $\mathcal{X}$ is a $\operatorname{WL}(c)$-relation. Conversely, $\mathcal{X}$ is coherent and each of its basic relations lies in one color class of $c$, so every basic relation of $\operatorname{WL}(c)$ is an $\mathcal{X}$-relation; hence $\operatorname{WL}(c)=\mathcal{X}$. A permutation preserving $c$ preserves every relation built from the color classes by these operations, hence every basic relation, and $c$ is constant on basic relations; so $\operatorname{Aut}(c)=\operatorname{Aut}(\mathcal{X})$.
\end{proof}

\begin{proposition}[Refinement depth]\label{prop:rounds}
Let $c$ be the coarsening of Lemma~\ref{lem:cohcoloring}, $t$ the number of its retained blocks $P_i$, $D$ the diameter of the graph $(X,S)$, and $L=\lceil\log_2\max(1,D)\rceil$. Run two-dimensional Weisfeiler--Leman refinement from the initial color $c(x,y)$ of each pair if $c$ is symmetric, and from $(c(x,y),c(y,x))$ in general, the diagonal keeping its own color. After at most
\[
\max(t,L)+L\le2\lceil\log_2 n\rceil
\]
rounds, its color classes are the basic relations of $\mathcal{X}$.
\end{proposition}

\begin{proof}
Let $\chi^0$ be the initial coloring and $\chi^{j+1}(x,y)$ the pair formed by $\chi^j(x,y)$ and the multiset of the pairs $(\chi^j(x,z),\chi^j(z,y))$, $z\in X$. Each $\chi^j$ is constant on every basic relation $t$ of $\mathcal{X}$: this holds for $\chi^0$, and if it holds for $\chi^j$, then for $(x,y)\in t$ the number of $z$ with $\chi^j(x,z)=\alpha$ and $\chi^j(z,y)=\beta$ is a sum of intersection numbers $p^t_{rs}$. If $A$ and $B$ are unions of $\chi^j$-classes, then their Boolean combinations and $A\circ B$ are unions of $\chi^{j+1}$-classes, since $\chi^{j+1}(x,y)$ records whether some $z$ has $(x,z)\in A$ and $(z,y)\in B$.

By Lemma~\ref{lem:cohcoloring}(a), $L^0,L^1$ and, for $j=0$, $T^0,T^1$ are unions of color classes of $c$, hence of $\chi^0$-classes. By~\eqref{eq:cohincoming} and induction on $j$, $s_j\cup s_j^*$ and the next $T^0,T^1$ are unions of $\chi^{j+1}$-classes, so all retained blocks are recovered by round $t$; in the arc case the halves $s_j,s_j^*$ are cut out by color classes (Lemma~\ref{lem:cohcoloring}(b)), at no further round.

In parallel, let $U_m$ be the relation of distance at most $m$ in $(X,S)$, so $U_0=\Delta_X$, $U_1=\Delta_X\cup S$, and $U_m=U_{\lfloor m/2\rfloor}\circ U_{\lceil m/2\rceil}$. By Lemma~\ref{lem:cohcoloring}(a), $U_0$ and $U_1$ are unions of $\chi^0$-classes, so by induction on $j$ every $U_m$ with $m\le2^j$ is a union of $\chi^j$-classes, as in \cite[Lemma~3]{LichterEtAl2019}. Since $2^L\ge D$, every $\Delta_m=U_m\setminus U_{m-1}$ is a union of $\chi^L$-classes.

Let $j_0=\max(t,L)$. Since each $\chi^{j+1}$ refines $\chi^j$, at round $j_0$ the diagonal, the relations of $\mathcal{R}_1=\bigcup_iP_i$, every $\Delta_m$ and every color class $c^{-1}(c(r))$ are unions of classes. If every basic relation of distance at most $2^{j-1}$ is a union of $\chi^{j_0+j-1}$-classes and $2^{j-1}<m=\ell(r)\le2^j$, then the factors $u_r,v_r$ of Lemma~\ref{lem:cohcoloring}(c) have distances $\lfloor m/2\rfloor,\lceil m/2\rceil\le2^{j-1}$, and the one composition in~\eqref{eq:cohrecover} makes $r$ a union of $\chi^{j_0+j}$-classes. Since $(X,S)$ is connected, $\ell(r)\le D\le2^L$ for every basic relation $r$ (if $D=1$, then $L=0$ and no relation has distance $2$ or more), so by induction on $j\le L$ every basic relation is a union of $\chi^{j_0+L}$-classes; as $\chi^{j_0+L}$ is constant on basic relations, its classes are the basic relations of $\mathcal{X}$.       

Each $E_i$ of the retained chain is an equivalence relation and an $\mathcal{X}$-relation, so the $E_i$-class of $x$ is the union of the sets $xr$ with $r\subseteq E_i$ and has size $e_i=\sum_{r\subseteq E_i}k_r$. Each $E_{i+1}$-class is a union of $E_i$-classes, and $E_i\ne E_{i+1}$ gives $e_{i+1}\ge2e_i$; hence $n=e_t\ge2^t$ and $t\le\lfloor\log_2n\rfloor$. Also $D\le n-1$ because $(X,S)$ is connected, so $L\le\lceil\log_2n\rceil$, and $\max(t,L)+L\le2\lceil\log_2n\rceil$.
\end{proof}

Proposition~\ref{prop:rounds} bounds the number of refinement rounds on the coloring of Lemma~\ref{lem:cohcoloring}; it is not a bound on running time. In case (ii) the initial pair coloring may have more classes than the $2d+3$ arc colors counted in Theorem~\ref{thm:coherent}.

The affine-plane schemes and their amorphic fusions are established \cite{VanDam2000,VanDamKoolenXiong2025}; the calculation below shows that they need $d+2$ colors for coherent reconstruction.

\begin{proposition}[Affine-plane obstruction]\label{prop:affinelower}
Let $q$ be a prime power and $X=\mathbb{F}_q^2$. For each of the $q+1$ one-dimensional subspaces $V$ of $X$ let $R_V=\{(x,y):y-x\in V\setminus\{0\}\}$. Then $\mathcal{X}_q=(X,\{\Delta_X\}\cup\{R_V\})$ is a symmetric homogeneous coherent configuration in which every relation $R_V$ has valency $d=q-1$, and every coarsening $c$ of $\mathcal{X}_q$ with $\operatorname{WL}(c)=\mathcal{X}_q$ gives the $q+1$ relations $R_V$ distinct colors; so it uses $q+1=d+2$ off-diagonal colors.
\end{proposition}

\begin{proof}
Every nonzero vector lies in exactly one line $V$, and $V=-V$, so the relations partition $X^2$ and are symmetric; each $R_V$ has valency $q-1$. Write $A_V$ for the adjacency matrix of $R_V$, $\mathrm{Id}$ for the identity matrix and $J$ for the all-ones matrix. For a line $V$, the number of $z$ with $z-x,\,y-z\in V\setminus\{0\}$ is $q-1$ if $y=x$, is $q-2$ if $y-x\in V\setminus\{0\}$, and is $0$ otherwise. For distinct lines $V,W$ we have $X=V\oplus W$, so $y-x$ has exactly one decomposition $a+b$ with $a\in V$, $b\in W$, and both summands are nonzero exactly when $y-x\notin V\cup W$. Hence
\[
A_V^2=(q-1)\mathrm{Id}+(q-2)A_V,\qquad A_VA_W=J-\mathrm{Id}-A_V-A_W\quad(V\ne W).
\]
Let $c$ be a coarsening of $\mathcal{X}_q$. Its off-diagonal color classes are the relations $R_B=\bigcup_{V\in B}R_V$ for the blocks $B$ of a partition of the set of lines; let $M_B=\sum_{V\in B}A_V$. Summing the identities above, for distinct blocks $B,C$ of sizes $b$ and $c'$,
\[
M_BM_C=bc'(J-\mathrm{Id})-c'M_B-bM_C,\qquad M_B^2=b(q-1)\mathrm{Id}+b(b-1)(J-\mathrm{Id})+(q-2b)M_B.
\]
Since $J-\mathrm{Id}=\sum_BM_B$, every product of two matrices among $\mathrm{Id}$ and the $M_B$ is a linear combination of them, so the number of $z$ with $(x,z)\in R_B$ and $(z,y)\in R_C$ depends only on the class of $(x,y)$. The classes $\Delta_X$ and $R_B$ are symmetric, so they form a coherent configuration, and $\operatorname{WL}(c)$ is this partition. If two lines lie in one block, it differs from $\mathcal{X}_q$. Hence $\operatorname{WL}(c)=\mathcal{X}_q$ forces all blocks to be singletons, that is, $q+1$ off-diagonal colors.
\end{proof}

\begin{remark}
Proposition~\ref{prop:affinelower} gives a family, for infinitely many $d$, in which coherent reconstruction needs $d+2$ colors; it does not show that $2d+1$ is optimal. This lower bound concerns coherent reconstruction, not coarsenings that only preserve $\operatorname{Aut}(\mathcal{X}_q)$.
\end{remark}

\begin{example}[A non-Schurian family]\label{ex:shrikhande}
Let $\Gamma$ be the Cayley graph of $(\mathbb{Z}/4\mathbb{Z})^2$ with connection set $C=\{\pm(1,0),\pm(0,1),\pm(1,1)\}$, the Shrikhande graph. Each of the fifteen nonzero elements is a sum $a+b$ with $a,b\in C$ in exactly two ordered ways, so $\Gamma$ is strongly regular with parameters $(16,6,2,2)$, and the diagonal, the adjacency relation $A$ and the nonadjacency relation $N$ form a symmetric homogeneous coherent configuration with valencies $1,6,9$. It is not Schurian: the nonadjacent pairs $((0,0),(0,2))$ and $((0,0),(1,2))$ have the common neighbors $(0,1),(0,3)$, which are nonadjacent, and $(0,1),(1,1)$, which are adjacent, so a group whose orbitals are the basic relations, preserving $A$, maps neither pair to the other, and $N$ is not one of its orbitals.

Let $\mathcal{X}_m$ be the tensor product of this configuration with the thin configuration of $C_2^m$, on the points $(g,a)$, $g\in(\mathbb{Z}/4\mathbb{Z})^2$, $a\in C_2^m$, with basic relations $r_w=\{((g,a),(g',a')):(g,g')\in r,\ a'-a=w\}$ for $r\in\{\Delta,A,N\}$ and $w\in C_2^m$; the number of intermediate points for $r_w$ and $s_{w'}$ at a pair of $t_{w''}$ is $p^t_{rs}$ if $w''=w+w'$ and $0$ otherwise. So $\mathcal{X}_m$ is a symmetric homogeneous coherent configuration with
\[
n=16\cdot2^m,\qquad|\mathcal{R}|=3\cdot2^m,\qquad d=9.
\]
It is not Schurian: the fiber $\{(g,0)\}$ contains the common $A_0$-neighbors of any two of its points, and $A_0,N_0$ induce $A,N$ on it, so the two pairs above show that $N_0$ is not one orbital. By Theorem~\ref{thm:coherent}(i), a coarsening with at most $19$ off-diagonal colors has coherent closure $\mathcal{X}_m$, while $\mathcal{X}_m$ has $3\cdot2^m-1$ off-diagonal basic relations, more than $19$ exactly when $m\ge3$. The Shrikhande graph and the thin configuration of $C_2^m$ are classical.
\end{example}

\begin{corollary}[Orbital configurations]\label{cor:orbital}
Let $H$ be a transitive permutation group on a finite set $X$ with $|X|\ge2$, all of whose subdegrees are at most $d$, where $d\ge2$. Its orbital configuration $\mathcal{O}(H)$ has a coarsening $c$ with at most $2d+3$ off-diagonal arc colors, and at most $2d+1$ off-diagonal colors if every orbital is symmetric, such that $\operatorname{WL}(c)=\mathcal{O}(H)$. The group $\operatorname{Aut}(c)$ is the $2$-closure of $H$ in the sense of Wielandt~\cite{Wielandt1969}, and it equals $H$ exactly when $H$ is $2$-closed.
\end{corollary}

\begin{proof}Each $h\in H$ maps the intermediate points of a path of two given orbitals from $x$ to $y$ bijectively onto those from $h(x)$ to $h(y)$, so $\mathcal{O}(H)$ is a coherent configuration, homogeneous as $H$ is transitive, with the subdegrees as valencies. Theorem~\ref{thm:coherent}(ii), or (i) when every orbital is symmetric, gives $c$ with $\operatorname{Aut}(c)=\operatorname{Aut}(\mathcal{O}(H))$, which is by definition the $2$-closure of $H$; $H$ is $2$-closed when it equals its $2$-closure.
\end{proof}

\begin{corollary}[Metric realization]\label{cor:coherentmetric}
Let $\mathcal{X}$ be a symmetric finite homogeneous coherent configuration in which every basic relation has valency at most $d$, where $d\ge2$. There is a metric on $X$ with at most $2d+2$ values, its nonzero values lying in $[1,2]$, whose isometry group is $\operatorname{Aut}(\mathcal{X})$.
\end{corollary}

\begin{proof}
Let $c$ be the symmetric coarsening of Theorem~\ref{thm:coherent}(i), with off-diagonal colors $1,\ldots,k$, $k\le2d+1$, and put $\delta(x,x)=0$ and $\delta(x,y)=\lambda_{c(x,y)}$ for $x\ne y$, with distinct $\lambda_1,\ldots,\lambda_k\in[1,2]$. As $c$ is symmetric, so is $\delta$; it is positive off the diagonal and $\delta(x,z)\le2\le\delta(x,y)+\delta(y,z)$ for distinct $x,y,z$, so $\delta$ is a metric with at most $k+1\le2d+2$ values. As the $\lambda_i$ are distinct, its isometries are the permutations preserving $c$, so its isometry group is $\operatorname{Aut}(c)=\operatorname{Aut}(\mathcal{X})$.
\end{proof}

\begin{remark}
The symmetry hypothesis cannot be dropped. The translations of $\mathbb{Z}/3$ form the automorphism group of the configuration of their orbitals, but a metric $\delta$ invariant under them has $\delta(0,1)=\delta(1,2)=\delta(2,0)=\delta(1,0)$, so it is constant off the diagonal and its isometry group is the full symmetric group on three points.
\end{remark}

\section{Sharp four-color realization}\label{sec:finite}

The proof of Theorem~\ref{thm:finite} forces every color-preserving
permutation fixing $e$ to preserve every inverse pair of differences.
Each branch does this in three steps: it \emph{recognizes} distinguished
generators or rank blocks from finite color patterns, it
\emph{synchronizes} the local orientations of the generator steps across
$G$, and it \emph{propagates} the recovered steps to all of $G$ by
normalization at every base point. Table~\ref{tab:branchmap} lists the
branches. The central cyclic
anchors (Lemma~\ref{lem:cyclicpatterns} and
Theorem~\ref{thm:centralanchor}) are proved here as the model branch; the
other branches are proved in Appendix~\ref{app:branches}.

\begin{table}[!htbp]
\centering
\caption{The branches of the proof of Theorem~\ref{thm:finite}. Here $a$
is an element of infinite order or of order at least six, and $\Omega(G)$
is the subgroup generated by all involutions.}
\label{tab:branchmap}
\begin{tabular}{p{0.58\textwidth}p{0.34\textwidth}}
\hline
Case & Construction\\
\hline
$G=\{e\}$ & immediate\\
Some $a$; $\langle a\rangle$ not normal & Theorem~\ref{thm:nonnormal}\\
Some $a$; $\langle a\rangle$ normal, a conjugate $a^r$ with $r\not\equiv\pm1$ & Theorem~\ref{thm:nonbinary}\\
Some $a$; $\langle a\rangle$ normal, every conjugate of $a$ in $\{a,a^{-1}\}$ & Theorem~\ref{thm:signed}, with Theorem~\ref{thm:centralanchor} or Theorem~\ref{thm:orderfour}\\
No such $a$; $\Omega(G)$ noncentral & Theorem~\ref{thm:noncentral}\\
No such $a$; $\Omega(G)\neq\{e\}$ central, $|\Omega(G)|\geq8$ & Theorem~\ref{thm:centrallarge}\\
No such $a$; $\Omega(G)\neq\{e\}$ central, $|\Omega(G)|\leq4$ & Theorem~\ref{thm:centralsmall}\\
No such $a$; $\Omega(G)=\{e\}$, all orders in $\{1,3,5\}$ & Theorem~\ref{thm:residual}\\
\hline
\end{tabular}
\end{table}

\begin{lemma}[Patterns on a cyclic coset]\label{lem:cyclicpatterns}
Put $A=\langle a\rangle$ and assume that the order of $a$ is infinite or
is a finite integer $N\geq6$. Suppose $a$ is central, and let $g\notin A$.
The components of the graph with edges $\{x,ax\}$ are the $A$-cosets:
two-sided infinite paths if $a$ has infinite order, and $N$-cycles
otherwise. In the coordinate $k\mapsto a^kg$ on $Ag$, with $k$ in
$\mathbb Z$ or in $\mathbb Z/N\mathbb Z$, the graph isometries of $Ag$
are the maps $k\mapsto\pm k+s$. If $Ag=Ag^{-1}$, write $g^{-1}=a^tg$;
inversion acts on $Ag$ as the reflection $r_t(k)=t-k$. The patterns are:
\begin{enumerate}
\item[(i)] if $Ag\neq Ag^{-1}$, with $P=\{0,1,3\}$,
\begin{equation}\label{eq:ca-D}
 D(g)=\{a^kg:k\in P\}\cup\{a^{-k}g^{-1}:k\in P\};
\end{equation}
\item[(ii)] if $Ag=Ag^{-1}$, $D(g)=\{a^kg:k\in K_t\}$, where $K_t=\{0,t\}$
for an infinite anchor and, for a finite anchor,
\begin{equation}\label{eq:ca-K}
 K_t=
 \begin{cases}
  \{0\},&t=0,\\
  \{0,t\},&t\neq0\text{ and }2t\neq0,\\
  \{1,t-1\},&N\text{ is even and }t=N/2;
 \end{cases}
\end{equation}
\item[(iii)] for an infinite anchor, $E(g)$ is $D(g)$ with $P$ replaced by
$Q=\{0,1,\dots,8,10\}$ and $K_t$ by $Q\cup(t-Q)$; for a finite anchor,
$E(g)=(Ag\cup Ag^{-1})\setminus D(g)$.
\end{enumerate}
Then $D(g)$ is inverse-invariant, lies in $Ag\cup Ag^{-1}$, and has the
following properties:
\begin{enumerate}
\item $D(g)$ meets each coset in $Ag\cup Ag^{-1}$, and every $A$-coset
meets it in at most three points; moreover,
$\langle A,D(g)\rangle=\langle A,g\rangle$;
\item every automorphism of the graph with edges $\{x,ax\}$ that
preserves $D(g)$ setwise acts on the entire coset union $Ag\cup Ag^{-1}$
either as the identity or as inversion.
\end{enumerate}
The graph automorphism in property~2 is allowed to be defined on all of
$G$; only its restriction to the displayed coset union is being
described.

If the order of $a$ is infinite or is at least seven, or if $N=6$ and
$Ag=Ag^{-1}$, then $E(g)$ is inverse-invariant, satisfies property~2,
and every coset that it meets contains at least four of its points. The set $E(g)$ need not contain $g$.
\end{lemma}

\begin{proof}
First suppose $Ag\neq Ag^{-1}$. The set $P$ has trivial stabilizer among isometries of the infinite path.
In a cycle of length $N\geq6$, its successive cyclic gaps are
$1,2,N-3$, which are distinct, so it again has trivial dihedral
stabilizer. Therefore a graph isometry fixing both displayed cosets and
preserving $D(g)$ must fix them pointwise. A graph isometry exchanging
the two cosets must be inversion on their union, since inversion is one
such exchange and composing with it reduces to the first case.

Now suppose $Ag=Ag^{-1}$. For an infinite anchor, $K_t$ has exactly
$\{\id,r_t\}$ as its isometry stabilizer. For a finite anchor, the one-point set $K_0=\{0\}$ has no nonidentity rotational stabilizer, since a nonidentity rotation of $\mathbb Z/N\mathbb Z$ fixes no point; the
two-point set in the middle case of~\eqref{eq:ca-K} is not an antipodal pair, so it has
no nonidentity rotational stabilizer. In the last case, write $N=2m$,
where $m\geq3$. The two points $1,m-1$ are distinct and nonantipodal, so
this case also has no nonidentity rotational stabilizer. Each displayed
set is invariant under $r_t$. Two distinct reflections would produce a
nonidentity rotation, so its full stabilizer is exactly $\{\id,r_t\}$.
This proves the rigidity assertion and the bound of three points per
coset. Although the last finite pattern need not contain $g$, any
selected point $a^kg$, together with $A$, recovers $g$. Thus the
generation assertion also holds.

For $E(g)$ with an infinite anchor and distinct cosets, the finite set
$Q$ has trivial isometry stabilizer: its only possible reflection
would send $1$ to $9$, which is absent. In a self-inverse coset,
$Q\cup(t-Q)$ is a finite set invariant under $r_t$. A finite nonempty
subset of the integer line has no nonidentity translation and at most one
reflection, so its stabilizer is exactly $\{\id,r_t\}$. Each resulting
coset contains at least ten selected points.

For a finite anchor, $E(g)$ is the complement of $D(g)$ in each coset
that $D(g)$ meets. A set and its complement have exactly the same dihedral
stabilizer, and inversion exchanges or preserves the complements as
required. A distinct coset contains $N-3$ selected points, and a
self-inverse coset contains at least $N-2$ selected points. If $N\geq7$,
these numbers are at least four; if $N=6$, the second one is. The same stabilizer argument proves
property~2 for $E(g)$.
\end{proof}

\begin{remark}
The bounds on the order of $a$ are sharp when $Ag\neq Ag^{-1}$. Every
subset of a $5$-cycle, and every subset of a $6$-cycle whose complement has
at most two points, is invariant under a reflection of the cycle, so for
$N=5$ no choice of $D(g)$, and for $N=6$ no choice of $E(g)$, excludes the
graph automorphism that reflects $Ag$ and fixes $Ag^{-1}$ pointwise.
\end{remark}

\begin{theorem}[A central-anchor four-color construction]\label{thm:centralanchor}
Let $G$ contain a central element $a$ whose order is infinite or is a
finite integer $N\geq6$. Then there is an inverse-invariant coloring
\[
 C:G\setminus\{e\}\longrightarrow\{1,2,3,4\}
\]
with $\operatorname{Aut}_C(G)=\Xi(G)$. Consequently $G$ has a complete,
uniformly discrete, right-invariant metric with at most five values,
including zero, and isometry group $\Xi(G)$.

If $G$ is nonabelian, the coloring can be chosen so that
$\operatorname{Aut}_C(G)=R(G)$. In particular, this construction also
proves $\Xi(G)=R(G)$ and gives a naturalizing metric in that case.
\end{theorem}

\begin{proof}
Choose a retained generating sequence with $a$ first. Let $G_\alpha$ and $\pi$ be its subgroup chain and rank label, as in Corollary~\ref{cor:fullcomparison}; thus $G_1=A=\langle a\rangle$. For each later retained
generator $b_\beta$, let
\begin{equation}\label{eq:ca-blocks}
 S_0=\{a,a^{-1}\},\qquad S_\beta=D(b_\beta).
\end{equation}
All members of $S_\beta$ have entry rank $\beta$: multiplication by
powers of $a$, and inversion, do not change a later entry rank. In
particular these blocks are disjoint. Their union is symmetric and
generates $G$, because $a$ and any point of $D(b_\beta)$ in $Ab_\beta$
recover $b_\beta$.

For the moment assume $G$ is nonabelian. Two retained generators do not
commute, since pairwise commuting generators would generate an abelian
group. Choose such generators $b,c$, with $b$ of higher rank $\beta$. Let
$w=bc$. When the anchor has infinite order or has finite order at least
seven, or has order six and $Aw=Aw^{-1}$, put $W=E(w)$ from
Lemma~\ref{lem:cyclicpatterns}; in the last case $W$ is the complement in
$Aw$ of the self-inverse pattern~\eqref{eq:ca-K}.

If the anchor has order six and $Aw\neq Aw^{-1}$, choose $t\in\{1,3\}$
such that
\begin{equation}\label{eq:ca-t}
 b^{-1}c^{-1}\neq a^tw.
\end{equation}
In $\mathbb Z/6\mathbb Z$, let
$Q$ be the complement of $\{0,1\}$ if $t=1$, and the complement of
$\{1,2\}$ if $t=3$. It has four points and stabilizer $\{\id,r_t\}$: its
complementary adjacent pair is nonantipodal, so it has no nonidentity
rotational symmetry. Set
\begin{equation}\label{eq:ca-W}
 W=\{a^kw:k\in Q\}\cup\{a^{-k}w^{-1}:k\in Q\}.
\end{equation}
In all cases enlarge just one rank block by setting
\begin{equation}\label{eq:ca-enlarge}
 S_\beta=D(b)\cup W.
\end{equation}
The element $w$ has rank $\beta$, because $c$ belongs to the subgroup
before $b$ enters. The coset unions $Ab\cup Ab^{-1}$ and $Aw\cup Aw^{-1}$ are disjoint.
Indeed, $bc\in Ab$ would imply $c\in A$, whereas $bc\in Ab^{-1}$ would
imply $c\in\langle A,b\rangle$. Either possibility makes $b$ and $c$
commute. Taking inverses gives the other two possible intersections and
the same contradictions. By Lemma~\ref{lem:cyclicpatterns}, a coset
meeting $D(b)$ contains at most three points of $S_\beta$, and a coset
meeting $W$ contains at least four. This numerical distinction will make
the two parts of the enlarged block separately recognizable.

In the abelian case make no modification~\eqref{eq:ca-enlarge}. In
either case put $S=\bigcup_\beta S_\beta$ and define the four-color map
\begin{equation}\label{eq:ca-color}
 C(g)=1+\pi(g)+2\mathbf 1_S(g)\qquad(g\neq e).
\end{equation}
This is the inverse-invariant map $C_D$ of Lemma~\ref{lem:protocol} for the entry-layer blocks $D_\beta=S_\beta$, which need not be inverse pairs.

Let $F\in\operatorname{Aut}_C(G)$. Its normalizations $F_x(z)=F(zx)F(x)^{-1}$ lie in $\operatorname{Aut}_C(G)$ and fix $e$ (Lemma~\ref{lem:normalize}). By Lemma~\ref{lem:protocol}(i), each $F_x$ fixes every $S_\beta$ setwise; since $S_0=[a]$, Lemma~\ref{lem:protocol}(ii), applied to $F$ and to each $F_x$, shows that these maps preserve the graph with edges $\{x,ax\}$ and are isometries on each of its path or cycle components. Let
$\sigma(x)\in\{1,-1\}$ be the orientation of $F$ on the component $Ax$, so that
\begin{equation}\label{eq:ca-orient}
 F(a^kx)=a^{k\sigma(x)}F(x),\qquad\sigma(a^kx)=\sigma(x).
\end{equation}
For an ordinary block, Lemma~\ref{lem:cyclicpatterns} shows that $F_x$
acts on the entire union $Ab_\beta\cup Ab_\beta^{-1}$ either as identity
or as inversion. In the enlarged block, graph component intersection
cardinalities distinguish $D(b)$ from $W$, so the same conclusion still
holds for the original generator coset union. It also holds for the
witness coset union except in the special order-six
case~\eqref{eq:ca-W}, which will be handled separately. Thus for every
retained generator $b\neq a$ there is a sign $\tau_b(x)$ with
\begin{equation}\label{eq:ca-steps}
 F(a^kbx)=(a^kb)^{\tau_b(x)}F(x)\qquad\text{for every integer }k.
\end{equation}
In the nonabelian case the identical statement also
holds with $b$ replaced by the witness $w$, except
for~\eqref{eq:ca-W}.

By~\eqref{eq:ca-steps}, $\sigma(bx)=\tau_b(x)$. Since $a$ is central and $\sigma$ is constant on each $A$-coset, $\tau_b(ax)=\sigma(bax)=\sigma(abx)=\sigma(bx)=\tau_b(x)$. Compute $F(abx)=F(bax)$ in two ways, by~\eqref{eq:ca-steps} at $x$ and at $ax$ and by~\eqref{eq:ca-orient}:
\[
\begin{aligned}
 F(abx)&=a^{\tau_b(x)}b^{\tau_b(x)}F(x),\\
 F(bax)&=b^{\tau_b(ax)}a^{\sigma(x)}F(x)=a^{\sigma(x)}b^{\tau_b(x)}F(x).
\end{aligned}
\]
Cancellation gives $a^{\tau_b(x)}=a^{\sigma(x)}$. The order of $a$ is
greater than two, so
\begin{equation}\label{eq:ca-sync}
 \tau_b(x)=\sigma(x),\qquad\sigma(bx)=\sigma(x).
\end{equation}
Together with constancy along $a$, and generation, this says that
$\sigma$ is one global sign $\varepsilon$. Outside~\eqref{eq:ca-W}, the same calculation applies to the witness
coset.

If $G$ is abelian and $F$ is rooted, the resulting generator-step
identities give $F(g)=g^\varepsilon$ for every $g\in G$. Applied to every normalization $F_x$, this gives $F_x(u)\in[u]$ for all $u,x\in G$, so $\operatorname{Aut}_C(G)=\Xi(G)$ by Lemma~\ref{lem:protocol}(iv).

If $G$ is nonabelian, suppose $F$ is rooted and $\varepsilon=-1$. The
generator steps give $F(w)=F(bc)=b^{-1}c^{-1}$. Outside~\eqref{eq:ca-W},
the witness step for $w$ with its synchronized sign gives $F(w)=w^{-1}$.
In the exceptional order-six case, preservation of $W$ sends $Aw$ to
$Aw$ or to $Aw^{-1}$, with negative anchor orientation. If
$F(w)=a^dw$, the selected coordinates satisfy $d-Q=Q$, forcing $d=t$
because $r_t$ is the unique reflection stabilizing $Q$; then
$b^{-1}c^{-1}=a^tw$, contrary to~\eqref{eq:ca-t}. Hence
$F(w)=a^dw^{-1}$ with $d-Q=-Q$; the set $Q$ is aperiodic, so $d=0$ and
$F(w)=w^{-1}$. In every case $b^{-1}c^{-1}=w^{-1}=c^{-1}b^{-1}$, so $b$
and $c$ commute, contrary to their choice.

Thus the global sign of every rooted $F$ is positive. Then~\eqref{eq:ca-orient} and~\eqref{eq:ca-steps} put $a$ and every retained generator in the subgroup $Y_F$ of Lemma~\ref{lem:protocol}, so $F=R_{F(e)}$ is the identity, and Lemma~\ref{lem:normalize} gives $\operatorname{Aut}_C(G)=R(G)=\Xi(G)$ without an external classification.

By Lemma~\ref{lem:encoding}, four distinct lengths in $[1,2]$ on the
colors in~\eqref{eq:ca-color} give the claimed complete uniformly
discrete metric. In the nonabelian case a compatible
group law has a regular translation subgroup of $R(G)$. A transitive
subgroup of a regular group is the whole group, so the competing law is
isomorphic to the original one. This proves naturality.
\end{proof}

\begin{proof}[Proof of Theorem~\ref{thm:finite}]If $G$ has an element of infinite order
or finite order at least six, apply Corollary~\ref{cor:largeorder}.
Otherwise all element orders lie in $\{1,2,3,4,5\}$. Put $H=\Omega(G)$,
the subgroup generated by all involutions of $G$.

If $H$ is noncentral, Theorem~\ref{thm:noncentral} applies. If $H$ is
nontrivial and central, any element of order three or five would commute
with a nonidentity involution and produce an element of order six or
ten. Thus $G$ has exponent dividing four. Use
Theorem~\ref{thm:centrallarge} when $|H|\geq8$ and
Theorem~\ref{thm:centralsmall} otherwise; these cases exhaust the
possible sizes of a Boolean group. Finally, if $H=\{e\}$, no element has
order two or four, so Theorem~\ref{thm:residual} applies.

By Lemma~\ref{lem:encoding}, any four distinct lengths in $[1,2]$
assigned to the four colors give a bounded, uniformly discrete, complete
right-invariant metric $d$ with at most five values, including zero, and
$\Isom(G,d)=\operatorname{Aut}_C(G)=\Xi(G)$.

Four colors are needed: by Lemmas~\ref{lem:affine} and~\ref{lem:encoding}, a coloring of $\Cyc{2}^3$ with at most three colors and color-preserving group $\Xi(\Cyc{2}^3)=T(\Cyc{2}^3)$ would give a right-invariant metric with at most four values, including zero, and isometry group $T(\Cyc{2}^3)$, which Proposition~\ref{prop:sharp} excludes, and Lemma~\ref{lem:ternary-plane} excludes at most three colors for $\Cyc{3}^2$.
\end{proof}

The four-color bound is the upper bound used for $\Cyc{2}^3$ and $\Cyc{3}^2$ in the proof of Theorem~\ref{thm:exact}, and, through Corollary~\ref{cor:naturalXi}, it gives the five-value bound of Theorem~\ref{thm:classification}.

\section{Regular subgroups and naturality}\label{sec:natural}

This section identifies compatible group laws with regular subgroups of the isometry group, determines the regular subgroups of an involution extension, and from these proves the classification of natural groups in Theorem~\ref{thm:classification}. We use the structure of $\Xi(G)$ in Theorem~\ref{thm:XiStructure}, due to Leemann--de la Salle~\cite{LeemannSalle}.

\subsection{Compatible group laws as regular actions}

We use the regular-action correspondence underlying Sabidussi's characterization of Cayley graphs~\cite{Sabidussi}. Lemma~\ref{lem:regular} gives its set-theoretic form, which fixes the opposite-group convention.

An action on a set $X$ is \emph{regular} if for every $x,y\in X$ there is exactly one group element taking $x$ to $y$.

\begin{lemma}[Regular-subgroup criterion]\label{lem:regular}
The right translations of a group law on a set $X$ form a regular permutation group $K$, and the abstract group of the law is isomorphic to $K^{\mathrm{op}}$, hence, by inversion, to $K$. Conversely, if $K$ is a regular permutation group on $X$, $o\in X$, and $k_y$ is the element of $K$ taking $o$ to $y$, then $x*y=k_y(x)$ is a group law on $X$ with identity $o$ whose right translations are precisely $K$. Consequently, a right-invariant metric $d$ on $G$ is naturalizing if and only if every regular subgroup of $\Isom(G,d)$ is abstractly isomorphic to $G$.
\end{lemma}

\begin{proof}
The right translations of a group law form a regular permutation group isomorphic to its opposite group, and every group is isomorphic to its opposite by inversion.

For the converse, $k_{x*y}=k_yk_x$, since both permutations take $o$ to $x*y$. This identity proves associativity, while $k_o=\id$ gives the identity axioms. The inverse of $x$ is $k_x^{-1}(o)$. Thus $*$ is a group law with identity $o$, its right translations are precisely $K$, and its abstract group is isomorphic to $K^{\mathrm{op}}$, hence to $K$. Applying this construction inside the isometry group proves the criterion.
\end{proof}

\begin{corollary}[Naturality as a property of the unavoidable group]\label{cor:naturalXi}
For any group $G$, the following conditions are equivalent:
\begin{enumerate}
\item $G$ is natural.
\item Every regular subgroup of $\Xi(G)$ is isomorphic to $G$.
\item Every right-invariant metric on $G$ with full isometry group $\Xi(G)$ is naturalizing.
\end{enumerate}
At least one metric in the third condition exists.
\end{corollary}

\begin{proof}
Every right-invariant metric has $\Xi(G)$ inside its isometry group, by Lemma~\ref{lem:unavoidable}. Thus a nonisomorphic regular subgroup of $\Xi(G)$ would obstruct every such metric, proving the first implication. The second condition implies the third by Lemma~\ref{lem:regular}. Theorem~\ref{thm:finite} supplies a metric with isometry group exactly $\Xi(G)$, so the third condition implies the first and is not vacuous.
\end{proof}

\subsection{Regular subgroups of an involution extension}

Write $R_g(x)=xg$, and compose permutations in the convention $fg=f\circ g$. Thus $R_gR_h=R_{hg}$.

\begin{theorem}[Classification in an involution extension]\label{thm:involext}
Let $G$ be any group, and let $\alpha\in\operatorname{Aut}(G)$ have order exactly two. Set
\[
 H=\langle R(G),\alpha\rangle=\{R_g\alpha^\varepsilon:g\in G,\ \varepsilon\in\F_2\}.
\]
The regular subgroups of $H$ are exactly
\[
 K_\chi=\{R_g\alpha^{\chi(g)}:g\in G\},\qquad \chi\in\operatorname{Hom}(G,\F_2),\quad \chi\circ\alpha=\chi.
\]
The parameter $\chi$ is unique. Every $K_\chi$ is normal of index two in $H$, and $K_0=R(G)$.
\end{theorem}

\begin{proof}
Each element of $H$ has a unique expression $R_g\alpha^\varepsilon$, since $R_g\alpha^\varepsilon(e)=g$ and $\alpha\ne\id$. Multiplication is
\[
 (R_g\alpha^\varepsilon)(R_h\alpha^\delta)=R_{\alpha^\varepsilon(h)g}\,\alpha^{\varepsilon+\delta}.
\]
For an $\alpha$-invariant homomorphism $\chi:G\to\F_2$, it follows that $\theta_\chi(R_g\alpha^\varepsilon)=\chi(g)+\varepsilon$ is a homomorphism $H\to\F_2$, surjective as $\theta_\chi(\alpha)=1$, with kernel $K_\chi$; evaluation at $e$ maps $K_\chi$ bijectively onto $G$, so $K_\chi$ is regular.

Conversely, let $K\leq H$ be regular. For $h\in H$, the element $k\in K$ with $k(e)=h(e)$ satisfies $k^{-1}h\in\{\id,\alpha\}$, the stabilizer of $e$ in $H$; hence $H=K\sqcup K\alpha$, disjoint as $\alpha\ne\id$ fixes $e$, and $K$ is normal of index two, with quotient homomorphism $\theta:H\to\F_2$ satisfying $\theta(\alpha)=1$. Since $\F_2$ is abelian and $\alpha R_g\alpha^{-1}=R_{\alpha(g)}$, the map $\chi(g)=\theta(R_g)$ is an $\alpha$-invariant homomorphism, and $\theta(R_g\alpha^\varepsilon)=\chi(g)+\varepsilon$ gives $K=K_\chi$. The element of $K$ taking $e$ to $g$ determines $\chi(g)$, which proves uniqueness.
\end{proof}

\begin{remark}[The compatible law and the identity automorphism]\label{rem:compatiblelaw}
The group law furnished by $K_\chi$, with identity $e$, is explicitly
\[
 x*_\chi y=\alpha^{\chi(y)}(x)\,y.
\]
Its right translations are $R_y\alpha^{\chi(y)}$. The requirement in Theorem~\ref{thm:involext} that $\alpha\ne\id$ is essential for the parameterization: if $\alpha=\id$, then $H=R(G)$ and all the displayed sets $K_\chi$ collapse to $R(G)$.
\end{remark}

\subsection{The abelian and dicyclic arguments}
For an abelian group $A$, written additively, write $\tau_a(x)=x+a$ and $T(A)=\{\tau_a:a\in A\}$, and recall from Lemma~\ref{lem:affine} that $\Xi(A)=\Affpm(A)$.

\begin{corollary}[All regular subgroups of the affine-sign group]\label{cor:affineregular}
Let $A$ be a non-Boolean abelian group. The regular subgroups of $\Affpm(A)$ are exactly the groups
\[
 K_\chi=\{x\mapsto a+(-1)^{\chi(a)}x:a\in A\},\qquad \chi\in\operatorname{Hom}(A,\F_2),
\]
and distinct $\chi$ give distinct $K_\chi$. For $\chi\ne0$, this subgroup is isomorphic to $\operatorname{Dih}(\ker\chi)$, the semidirect product of $\ker\chi$ by an involution acting by inversion.
\end{corollary}

\begin{proof}
Inversion $\alpha(x)=-x$ is an automorphism of order exactly two of the non-Boolean abelian group $A$, every homomorphism $A\to\F_2$ is $\alpha$-invariant, and $R_a\alpha^\varepsilon$ is $x\mapsto a+(-1)^\varepsilon x$, so $\langle R(A),\alpha\rangle=\Affpm(A)$. Theorem~\ref{thm:involext} therefore gives exactly the displayed regular subgroups, with $\chi$ unique.

For nonzero $\chi$, put $N=\ker\chi$ and choose $a\notin N$. The subgroup $K_\chi$ consists of the translations by $N$ and the coset containing $s(x)=a-x$. Since
\[
 s^2=\id,\qquad s\tau_ns^{-1}=\tau_{-n}\quad(n\in N),
\]
it is $\operatorname{Dih}(N)$.
\end{proof}

\begin{corollary}[The two abelian naturality directions]\label{cor:abeliannat}
For an abelian group $A$, the affine-sign group $\Affpm(A)$ has only the translation group $T(A)$ as a regular subgroup if and only if $A$ is Boolean or $2A=A$. If $A$ is non-Boolean and $2A\ne A$, every nontranslation regular subgroup in Corollary~\ref{cor:affineregular} is nonisomorphic to $A$.

Consequently, using Theorem~\ref{thm:finite}, an abelian group $A$ is natural if and only if it is Boolean or $2A=A$.
\end{corollary}

\begin{proof}
If $A$ is Boolean, inversion is the identity, so $\Affpm(A)=T(A)$, whose only transitive subgroup is $T(A)$. Otherwise, by Corollary~\ref{cor:affineregular}, the regular subgroups of $\Affpm(A)$ are the groups $K_\chi$, distinct for distinct $\chi\in\operatorname{Hom}(A,\F_2)$, with $K_0=T(A)$. If $2A=A$, every such $\chi$ vanishes, as $\chi(2a)=0$; if $2A\ne A$, a nonzero linear functional on the nonzero $\F_2$-vector space $A/2A$, composed with the quotient, is a nonzero $\chi$. For $\chi\ne0$, the group $K_\chi\cong\operatorname{Dih}(\ker\chi)$ is nonabelian if $2\ker\chi\ne\{0\}$ and Boolean otherwise, so it is not isomorphic to the non-Boolean abelian group $A$. Since $\Xi(A)=\Affpm(A)$ by Lemma~\ref{lem:affine} and $T(A)\cong A$, Corollary~\ref{cor:naturalXi} gives the naturality equivalence.
\end{proof}

For a generalized dicyclic group, the standard inversion on the nontrivial coset is an unavoidable isometry; see~\cite{LeemannSalle}. Knill~\cite[proof of Theorem~3]{KnillNatural} uses the associated generalized dihedral group as an obstruction. We establish that obstruction as a permutation action, checking regularity and nonisomorphism for arbitrary cardinalities.

\begin{proposition}[The generalized dicyclic obstruction]\label{prop:dic-negative}
Every generalized dicyclic group is nonnatural.
\end{proposition}

\begin{proof}
Let $G=\langle A,t\rangle$ satisfy~\eqref{eq:dic}, so that $A$ is abelian of index two, $t\notin A$, the element $z=t^2$ is a nonidentity involution in $A$, and $tat^{-1}=a^{-1}$ for $a\in A$. The assignments
\[
 \psi(a)=a\quad(a\in A),\qquad \psi(t)=t^{-1}
\]
preserve these relations and define an automorphism of order two. Since $\psi(at)=at^{-1}=(at)^{-1}$ for $a\in A$, every $g\in G$ satisfies $\psi(g)\in\{g,g^{-1}\}$. Thus $\psi(x)\psi(y)^{-1}=\psi(xy^{-1})\in[xy^{-1}]$ for $x,y\in G$, so $\psi\in\Xi(G)$.

Let $\chi:G\to\F_2$ be the quotient character with kernel $A$. It is $\psi$-invariant, so Theorem~\ref{thm:involext} supplies a regular subgroup
\[
 K_\chi=\{R_g\psi^{\chi(g)}:g\in G\}\leq\langle R(G),\psi\rangle\leq\Xi(G).
\]
Set $s=R_t\psi$. Then
\[
 s^2=\id,\qquad sR_as^{-1}=R_{a^{-1}}\quad(a\in A),
\]
and $K_\chi=\langle R(A),s\rangle\cong\operatorname{Dih}(A)$. This group is generated by involutions: both $s$ and $sR_a$ are involutions, and $R_a=s(sR_a)$.

By contrast, every element $at$ outside $A$ satisfies $(at)^2=z\ne e$, so every involution of $G$ lies in the proper subgroup $A$. Hence $G$ is not generated by involutions, so $G\not\cong K_\chi$, and $G$ is not natural by Corollary~\ref{cor:naturalXi}.
\end{proof}

\subsection{The structure of $\Xi(G)$}

\begin{theorem}[{Leemann--de la Salle~\cite[Theorem~2]{LeemannSalle}}]\label{thm:XiStructure}
Write $\Xi(G)_e$ for the stabilizer of $e$ in $\Xi(G)$. For every
group $G$, its unavoidable inverse-pair symmetry group $\Xi(G)$ has the following description.
\begin{center}
\begin{tabular}{lll}
\hline
 & $G$ & $\Xi(G)$\\
\hline
(1) & Boolean & $R(G)$\\
(2) & abelian, not Boolean & $\Affpm(G)$\\
(3) & $Q_8\times E$, with $E$ Boolean & $\Xi(G)_e\cong\Cyc{2}^3$\\
(4) & generalized dicyclic, not of the form in (3) & $R(G)\rtimes\langle\psi\rangle$\\
(5) & every other group & $R(G)$\\
\hline
\end{tabular}
\end{center}
In case~(3), the three signs of $\Xi(G)_e$ independently invert the three noncentral quaternion inverse pairs,
uniformly over $E$, and every element of $\Xi(G)$ has a unique expression $R_gF$ with $F\in\Xi(G)_e$; no
semidirect-product assertion is made there. In case~(4), $\psi$ fixes an abelian subgroup of index two pointwise and inverts
its complement. In case~(5), $G$ is nonabelian.
\end{theorem}

The proof is due to Leemann and de la Salle \cite{LeemannSalle}; the statement above is \cite[Theorem~2]{LeemannSalle} transported to the right-translation convention by the inversion $J(g)=g^{-1}$, which satisfies $JR_gJ=L_{g^{-1}}$ with $L_h(x)=hx$; this changes neither the hypotheses nor the classification.

\subsection{The naturality classification}

\begin{proof}[Proof of Theorem~\ref{thm:classification}]
For abelian $G$, the first assertion is Corollary~\ref{cor:abeliannat}, since $G$ is Boolean exactly when $2G=\{0\}$. Generalized dicyclic groups are nonnatural by Proposition~\ref{prop:dic-negative}. A nonabelian group that is not generalized dicyclic lies in case~(5) of Theorem~\ref{thm:XiStructure}, because the groups of cases~(3) and~(4) are generalized dicyclic. Thus $\Xi(G)=R(G)$, whose only transitive subgroup is $R(G)\cong G$, and $G$ is natural by Corollary~\ref{cor:naturalXi}. In every positive case, Corollary~\ref{cor:naturalXi} makes the metric of Theorem~\ref{thm:finite}, which has isometry group $\Xi(G)$ and at most five values, naturalizing. Proposition~\ref{prop:sharp} proves that five is sharp for $\Cyc{2}^3$.
\end{proof}

This settles naturality for every group; Section~\ref{sec:exact} turns to the exact invariant $N(G)$, and comparisons with $M(G)$.

\begin{corollary}[Compatible laws with a prescribed identity]\label{cor:rootedlaw}
Let $G$ be natural, and let $d$ be a right-invariant metric on $G$ with $\Isom(G,d)=\Xi(G)$. Call a group law on the
set $G$ \emph{compatible} with a metric if its right translations are isometries of that metric. The group laws
compatible with $d$ are precisely
\[
 x\star_a y=xa^{-1}y\qquad(a\in G),
\]
and the identity of $\star_a$ is $a$. In particular, for every choice of identity there is exactly one compatible
group law, and $d$ can be chosen bounded, complete, uniformly discrete, and with at most five values, including zero.
Conversely, for any group $G$, if a right-invariant metric on $G$ has only one compatible law with identity $e$,
then $G$ is natural.
\end{corollary}

\begin{proof}
Theorem~\ref{thm:classification}, Corollary~\ref{cor:abeliannat} and Theorem~\ref{thm:XiStructure}, whose cases~(3) and~(4) consist of generalized dicyclic groups, show that $R(G)$ is the unique regular subgroup of $\Xi(G)$.

By Lemma~\ref{lem:regular}, the right translations of a law compatible with $d$, with identity $a$, form a regular subgroup of $\Isom(G,d)=\Xi(G)$, hence equal $R(G)$. The right translation by $y$ takes $a$ to $y$, so it is $x\mapsto xa^{-1}y$, and the law is $\star_a$. Conversely, $x\mapsto a^{-1}x$ is an isomorphism from $(G,\star_a)$ to $G$, so $\star_a$ is a group law with identity $a$; its right translations $x\mapsto xa^{-1}y$ form $R(G)\le\Xi(G)$, so it is compatible. Theorem~\ref{thm:finite} supplies a metric $d$ with $\Isom(G,d)=\Xi(G)$ and the stated properties.

For the converse, let $d'$ be a right-invariant metric on $G$ with only one compatible law with identity $e$; by right invariance, it is the original law. If $K$ is a regular subgroup of $\Isom(G,d')$, Lemma~\ref{lem:regular} with $o=e$ gives a compatible law with identity $e$ whose right translations are precisely $K$, so $K=R(G)\cong G$. By Lemma~\ref{lem:regular}, $d'$ is naturalizing, so $G$ is natural.
\end{proof}

\section{Sharpness and exact minima}\label{sec:exact}\label{sec:consequences}

\subsection{Boolean and ternary sharpness}\label{sec:sharpness}

\begin{lemma}[A color-preserving transvection]\label{lem:fano}
Every coloring $C$ of $\F_2^3\setminus\{0\}$ with at most three colors admits
a nonidentity color-preserving linear transvection $\alpha$, that is,
$C(\alpha x)=C(x)$ for all $x\in\F_2^3\setminus\{0\}$.
\end{lemma}

\begin{proof}
Choose a monochromatic triple $T$, which exists because there are seven
points and at most three colors.

First suppose that $T=U\setminus\{0\}$ for a two-dimensional subspace $U$.
The four points outside $U$ contain a monochromatic pair $p,q$. Their sum
$v=p+q$ lies in $T$. The other two points of $T$ form a second monochromatic
pair with sum $v$.

Otherwise $T=\{a,b,c\}$ is a basis. Consider
\[
 B=\{a+b,a+c,b+c\}.
\]
If a point of $B$ has the color of $T$, it and its two summands form a
monochromatic line, reducing to the preceding case. Otherwise $B$ uses at
most two colors and contains a monochromatic pair. Its sum is the remaining
point of $B$, hence is also the sum of a monochromatic pair from $T$.

In either case there are two disjoint monochromatic pairs
\[
 \{r,r+v\},\qquad \{t,t+v\}.
\]
The vectors $r,t,v$ form a basis: $r,v$ are independent, and the
distinctness and nonzero conditions on the four points give
$t\notin\{0,r,v,r+v\}$. Define the linear map
\[
 \alpha(r)=r+v,\qquad \alpha(t)=t+v,\qquad \alpha(v)=v.
\]
It swaps the two monochromatic pairs and fixes each of the other three
nonzero points, namely $v,r+t,r+t+v$. It therefore preserves the coloring.
Equivalently,
\[
 \alpha(x)=x+f(x)v,\qquad f(r)=f(t)=1,\quad f(v)=0,
\]
where $f$ is linear. Thus $\alpha$ is a nonidentity transvection.
\end{proof}

\begin{proposition}[One obstruction proves both Boolean lower bounds]\label{prop:sharp}
Let $V=\F_2^3$, and let $d$ be a right-invariant metric on $V$ with at most
four distinct values, including zero. Then $\Isom(V,d)$ contains a regular
subgroup isomorphic to $\Cyc{4}\times\Cyc{2}$. In particular, $d$ is not
naturalizing and its isometry group is not $T(V)$.
\end{proposition}

\begin{proof}
Coloring each nonzero vector by its distance from $0$ uses at most three
colors, so Lemma~\ref{lem:fano} gives a color-preserving transvection
$\alpha(x)=x+f(x)v$, with $f\neq0$ linear, $v\neq0$ and $f(v)=0$. By linearity and right invariance, $d(\alpha(x),\alpha(y))=d(\alpha(x-y),0)=d(x-y,0)=d(x,y)$, so $\alpha$ is an isometry. Take $r$ with $f(r)=1$ and $t\in r+\ker f$ outside $\{r,r+v\}$; then $t+r$ and $v$ span $\ker f$, so $r,t,v$ is a basis with $\alpha(r)=r+v$, $\alpha(t)=t+v$, $\alpha(v)=v$, whence $\alpha(r+t)=r+t$ and $\alpha^2=\id$.

Write
$\tau_x(y)=y+x$, and put
\[
 b=r+t,\qquad s=\tau_r\alpha.
\]
Since $\alpha(b)=b$ and $\alpha(r)=r+v$,
\[
 s\tau_b=\tau_bs,\qquad s^2=\tau_v.
\]
Thus $s$ has order four, while $\tau_b$ has order two. The only translations
in $\langle s\rangle$ are $\id$ and $\tau_v$, since its odd powers have
nonidentity linear part $\alpha$. As $b\neq0,v$, the two cyclic subgroups
intersect trivially. Consequently,
\[
 K=\langle s,\tau_b\rangle\cong\Cyc{4}\times\Cyc{2}.
\]
Its orbit of zero is
\[
 K(0)=\{0,r,v,r+v\}+\{0,b\}=V,
\]
because $r,v,b$ form a basis. Hence the order-eight group $K$ acts regularly.
It is not isomorphic to $V$, which has exponent two, so
Lemma~\ref{lem:regular} proves the naturalizing obstruction. The nonidentity
isometry $\alpha$ fixes $0$, so it also excludes $\Isom(V,d)=T(V)$.
\end{proof}

\begin{remark}[How the two sharpness arguments combine]\label{rem:sharpcombine}
By Lemma~\ref{lem:affine}, $\Xi(\Cyc{2}^3)=T(\Cyc{2}^3)$, so Proposition~\ref{prop:sharp} gives the exact-realization bound $N(\Cyc{2}^3)\geq5$ together with the naturalizing bound of Corollary~\ref{cor:Mfive}, and Theorem~\ref{thm:finite} gives the matching five-value metric.
\end{remark}

\subsection{Reduction to colorings}\label{sec:reduction}
 For a finite group $G$, the minimum defining $\operatorname{ci}(G)$ may be restricted to generating connection sets, because the complement of a disconnected graph is connected and has the same automorphism group.

\begin{lemma}[Color--metric reduction]\label{lem:exact-reduction}
For every group $G$, one has $N(G)=k(G)+1$.
Moreover, $k(G)\leq2$ if and only if some Cayley graph on $G$ has
automorphism group $\Xi(G)$. For nontrivial $G$, one has $k(G)=1$
if and only if $|G|\leq3$.
\end{lemma}
\begin{proof}
Nonzero level sets of a right-invariant metric with $n$ values form an inverse-invariant coloring with $n-1$ occurring colors and the same preserving permutations, and Lemma~\ref{lem:encoding} turns $k$ occurring colors into $k+1$ values with the same isometries; both minima are finite by Theorem~\ref{thm:finite}, so $N(G)=k(G)+1$. Colorings with at most two occurring colors are exactly the adjacency and nonadjacency classes of Cayley graphs, either class possibly empty, and preserve the same permutations as the graph.

For nontrivial $G$, one occurring color gives $\Sym(G)$, so $k(G)=1$ exactly when $\Xi(G)=\Sym(G)$. This holds when $G\setminus\{e\}$ is one inverse pair, that is, when $|G|\leq3$; otherwise a permutation fixing $e$ can move an element into another inverse pair, and so lies outside $\Xi(G)$.
\end{proof}

\subsection{Three values outside the exceptional set}\label{sec:outside}
Throughout Sections~\ref{sec:outside}--\ref{sec:lower}, assume $|G|>3$. Sections~\ref{sec:reduction} and~\ref{sec:exact-proof} include the smaller groups. For finite groups, define
\[
 \mathcal E=\{G\text{ finite}:\operatorname{ci}(G)>[\Xi(G):R(G)]\},
\]
where the set is taken up to isomorphism.

\begin{lemma}[Two colors outside the exceptional set]\label{lem:two-color-outside}
Let $G$ be a group with $|G|>3$ that is infinite, or finite with $G\notin\mathcal E$. Then some Cayley graph on $G$ has automorphism group $\Xi(G)$.
\end{lemma}

\begin{proof}
For a symmetric $\Sigma\subseteq G\setminus\{e\}$, membership in $\Sigma$ is an inverse-invariant coloring of $G\setminus\{e\}$ whose color-preserving group $\Aut(\operatorname{Cay}(G,\Sigma))$ therefore contains $\Xi(G)$, so
\begin{equation}\label{eq:cayley-containment}
 R(G)\leq\Xi(G)\leq\Aut(\operatorname{Cay}(G,\Sigma)).
\end{equation}
For infinite $G$, \cite[Theorem~C]{SutherlandWatkins} supplies, for every infinite group, a symmetric $\Sigma$ such that $\operatorname{Cay}(G,\Sigma)$, on the set $G$ itself and in the same right Cayley convention ($x\sim y$ if and only if $xy^{-1}\in\Sigma$), has $\Aut(\operatorname{Cay}(G,\Sigma))=\Xi(G)$; this graph has diameter two, so it is connected and $\Sigma$ generates $G$. For finite $G\notin\mathcal E$, choose $\Sigma$ attaining $\operatorname{ci}(G)$ and put $\Gamma=\Aut(\operatorname{Cay}(G,\Sigma))$; then \eqref{eq:cayley-containment} and the definition of $\mathcal E$ give $[\Gamma:R(G)]=\operatorname{ci}(G)=[\Xi(G):R(G)]$, and these indices are finite, so $\Gamma=\Xi(G)$.
\end{proof}

We introduce notation for preservation on a specified set of differences. For a symmetric subset $\mathcal W\subseteq G$, put
\[
 \Xi_{G,\mathcal W}=\{F\in\Sym(G):F(ux)\in[u]F(x)
                         \text{ for all }u\in \mathcal W,\ x\in G\}.
\]
The right-convention Cayley graph has vertex set $G$ and edges $\{x,ux\}$, for $u\in \mathcal W\setminus\{e\}$, of color $[u]$. Its color-preserving automorphism group is therefore $\Xi_{G,\mathcal W}$; see \cite{HujdurovicEtAl} for related Cayley-graph questions.

\begin{lemma}[Transport by inversion]\label{lem:conventions}
Let $G$ be a group, let $J(x)=x^{-1}$ for $x\in G$, and let
$\mathcal W=\mathcal W^{-1}\subseteq G$. Define
$\Xi^{\mathrm L}_{G,\mathcal W}$ to consist of the permutations
$\varphi\in\Sym(G)$ satisfying $\varphi(xu)\in\varphi(x)[u]$ for all
$x\in G$ and $u\in\mathcal W$, and put $\mathcal L_g(x)=gx$.
Then $J\Xi_{G,\mathcal W}J=\Xi^{\mathrm L}_{G,\mathcal W}$ and
$JR_gJ=\mathcal L_{g^{-1}}$ for all $g\in G$.
\end{lemma}

\begin{proof}
Let $F\in\Xi_{G,\mathcal W}$, $u\in\mathcal W$, and $y\in G$. As $u^{-1}\in\mathcal W$, the defining condition at step $u^{-1}$ and base point $y^{-1}$ gives $F(u^{-1}y^{-1})\in[u^{-1}]F(y^{-1})$, so
\[
 (JFJ)(yu)=F(u^{-1}y^{-1})^{-1}\in F(y^{-1})^{-1}[u]=(JFJ)(y)[u].
\]
Hence $J\Xi_{G,\mathcal W}J\leq\Xi^{\mathrm L}_{G,\mathcal W}$; the reverse calculation and $J^2=\id$ give equality. Finally, $(JR_gJ)(y)=(y^{-1}g)^{-1}=g^{-1}y=\mathcal L_{g^{-1}}(y)$ for all $g,y\in G$.
\end{proof}

The published classification \cite[Table~1 and Sections~2--5]{MorrisTymburski}, compared with Theorem~\ref{thm:XiStructure}, gives exactly
\begin{equation}\label{eq:exceptional-set}
\begin{split}
\mathcal E=\{&\Cyc{2}^2,\Cyc{2}^3,\Cyc{2}^4,
 \Cyc{4}\times\Cyc{2},\Cyc{4}\times\Cyc{2}^2,\Cyc{4}^2,
 \Cyc{3}^2,\Cyc{3}^3,\\
 &Q_8,Q_8\times\Cyc{2},\Dic(\Cyc{6},3),\Dic(\Cyc{8},4),
 \Dic(\Cyc{10},5),\\
 &\Dic(\Cyc{4}\times\Cyc{2},(0,1)),D_6,D_8,D_{10},A_4,
 Q_8\times\Cyc{3},Q_8\times\Cyc{4},\\
 &H_1,H_2,H_3,H_4\}.
\end{split}
\end{equation}
Here $D_{2n}$ has order $2n$; the dicyclic notation and the four presentations $H_i$ are specified in Appendix~\ref{app:certificates}. The trivial group and $\Cyc{2},\Cyc{3}$ remain the separate cases of Lemma~\ref{lem:exact-reduction}.

Published Table~1 of \cite{MorrisTymburski} lists $\Dic(\Cyc{10},5)$, with Cayley index $4$. The ancillary sweep provides an independent check through the stated orders: it checks all 143 isomorphism types of groups of orders 2--32 and obtains exactly the 24 exceptions in \eqref{eq:exceptional-set}. Its directly computed indices $[\Xi(G):R(G)]$ agree with the classification in every case. The exclusion of further finite exceptions at higher orders comes from the published Morris--Tymburski classification, not from the finite sweep.

\subsection{Certificate upper bounds on \texorpdfstring{$\mathcal E$}{the exceptional set}}\label{sec:upper}
A \emph{certificate} here is an explicit partition of $G\setminus\{e\}$ into inverse-invariant color classes. The certificates for all 24 groups of \eqref{eq:exceptional-set} are in the ancillary files \path{five_values/certificates.json} and \texttt{expected.json}. Write $\Xi(G)_e=\{F\in\Xi(G):F(e)=e\}$ and $\operatorname{Aut}_C(G)_e$ for the identity stabilizers.

\begin{lemma}[Three-color certificates]\label{lem:certificates}
For each of the 22 groups $G\in\mathcal E\setminus\{\Cyc{2}^3,\Cyc{3}^2\}$, the ancillary certificate of $G$ is an inverse-invariant coloring of $G\setminus\{e\}$ with three colors whose color-preserving group is $\Xi(G)$. Hence $k(G)\leq3$.
\end{lemma}

\begin{proof}
Let $C$ be the certificate coloring of $G$. To compute $\operatorname{Aut}_C(G)_e$, fix an enumeration of $G$ beginning with $e$ and enumerate the injective partial maps
$p:G\to G$ with $p(e)=e$. An extension $p(x)=y$ is retained exactly when $y$ is unused and
\[
 C\bigl(y\,p(u)^{-1}\bigr)=C(xu^{-1})
\]
for every previously assigned $u\neq x$. At a complete assignment every pair constraint has been checked, so the
complete retained maps are exactly the elements of $\operatorname{Aut}_C(G)_e$; conversely, every color-preserving
permutation fixing $e$ survives each extension test. The search is therefore exhaustive over all permutations of
$G$ fixing $e$, not only over automorphisms of $G$. For each of the 22 three-color certificates the recorded output is
\[
 |\operatorname{Aut}_C(G)_e|=|\Xi(G)_e|,
\]
where Theorem~\ref{thm:XiStructure} determines the right side. Since $C$ is inverse-invariant,
$\Xi(G)\leq\operatorname{Aut}_C(G)$; both groups contain $R(G)$, so equality of their identity stabilizers gives
$\operatorname{Aut}_C(G)=\Xi(G)$. As $C$ has three colors, $k(G)\leq3$.
\end{proof}

\needspace{30\baselineskip}
\begin{center}
\small
\end{center}

Each certificate is also checked by nauty via GRAPE on a layered graph and by intersecting the automorphism groups of the individual color graphs; these two checks share nauty as a backend, and the backtracking check of Lemma~\ref{lem:certificates} is separate. The group models are identified against \texttt{SmallGroup} constructions, and the ancillary verification entry point transports the models to those constructions. Software versions, run records and logs, which record the historical runs and not a fresh rerun, are listed in the README of the ancillary files. The two four-color certificates are corroborative; their upper bounds follow from Theorem~\ref{thm:finite}.

\subsection{Lower bounds}\label{sec:lower}

\begin{lemma}[Three colors on the exceptional set]\label{lem:exceptional-three}
Every $G\in\mathcal E$ satisfies $k(G)\geq3$.
\end{lemma}

\begin{proof}
Let $G\in\mathcal E$. If $k(G)\leq2$, the Cayley-graph part of Lemma~\ref{lem:exact-reduction}, applied to this finite group, would give a graph with automorphism group $\Xi(G)$ and therefore $\operatorname{ci}(G)\leq[\Xi(G):R(G)]$. This contradicts the definition of $\mathcal E$, so $k(G)\geq3$.
\end{proof}

\begin{lemma}[Four colors are necessary for the ternary plane]\label{lem:ternary-plane}
Every inverse-invariant coloring of $\Cyc{3}^2$ with at most three
occurring colors has a color-preserving permutation outside
$\Xi(\Cyc{3}^2)$. Hence $k(\Cyc{3}^2)\geq4$.
\end{lemma}

\begin{proof}
Let $C:\F_3^2\setminus\{0\}\to\mathcal C$ be the given inverse-invariant coloring, with at most three occurring colors. The four nonidentity inverse pairs are the nonzero parts of the four one-dimensional subspaces of $\F_3^2$. At most three colors merge two of these lines. Choose nonzero vectors $c,d$ on the other two lines; they form a basis, and the merged lines are $\langle c+d\rangle$ and $\langle c-d\rangle$. Take the linear map $\vartheta $ with $\vartheta (c)=-c$ and $\vartheta (d)=d$. It preserves each of the other two lines and interchanges the merged lines, so $C(\vartheta x)=C(x)$ for all $x\in\F_3^2\setminus\{0\}$, and linearity gives
\[
 C(\vartheta x-\vartheta y)=C(\vartheta (x-y))=C(x-y)\quad\text{for all distinct }x,y\in\F_3^2.
\]
By Lemma~\ref{lem:affine}, the elements of $\Xi(\Cyc{3}^2)$ fixing $0$ are precisely $\id$ and $-\id$. But $\vartheta(c)=-c\ne c$ and $\vartheta(d)=d\ne-d$, since $c,d$ are nonzero in characteristic three. Hence $\vartheta\ne\pm\id$ and $\vartheta\notin\Xi(\Cyc{3}^2)$.
\end{proof}

\subsection{The two minima}\label{sec:two-minima}

For a natural group $G$, let $M(G)$ be the least number of values, counting
zero, of a naturalizing right-invariant metric on $G$. By
Corollary~\ref{cor:naturalXi}, every right-invariant metric on a natural
group $G$ with isometry group $\Xi(G)$ is naturalizing, so $M(G)\leq N(G)$.

\begin{proposition}[A smaller naturalizing metric on $\Cyc{3}^2$]\label{prop:ternarynat}
The minimum number of distinct values, including zero, of a right-invariant
naturalizing metric on $\Cyc{3}^2$ is
\[
 M(\Cyc{3}^2)=3.
\]
The three-valued Hamming metric has exactly one regular subgroup in its
isometry group, namely the translation group.
\end{proposition}

\begin{proof}
On $V=\F_3^2$, use the Hamming metric
\[
 d(x,y)=\mathbf{1}_{x_1\neq y_1}+\mathbf{1}_{x_2\neq y_2}.
\]
This metric is right-invariant. The maximal cliques of its distance-one graph are the three rows and the three columns; rows are pairwise disjoint, as are columns, and every row meets every column, so an isometry preserves or interchanges the two families. As each point is the unique intersection of its row and column, a family-preserving isometry is given by independent permutations of rows and columns, and these and the interchange of coordinates preserve Hamming distance. Thus $\Isom(V,d)\cong(S_3\times S_3)\rtimes\Cyc{2}$, where $\Cyc{2}$ interchanges the two coordinate families.

A regular subgroup $K\leq\Isom(V,d)$ has order nine, so its image in this $\Cyc{2}$ is trivial and $K\leq S_3\times S_3$. Each coordinate projection of $K$ has order dividing nine and six, so it lies in $A_3$, the unique subgroup of order three of $S_3$; hence $K=A_3\times A_3$ by order. On the alphabet $\F_3$, the group $A_3$ consists of the translations, so $K=T(V)$, proving uniqueness and, by Lemma~\ref{lem:regular}, naturality.

Finally, a metric on nine points with two values, counting zero, is equilateral; its isometry group $\Sym(V)$ contains the regular subgroup $\Cyc{9}$ generated by a nine-cycle, which is not isomorphic to $\Cyc{3}^2$, so the metric is not naturalizing. Hence $M(\Cyc{3}^2)\geq3$, and the Hamming metric attains equality.
\end{proof}

\begin{remark}[The two minima differ]\label{rem:twominima}
Proposition~\ref{prop:ternarynat} and Theorem~\ref{thm:exact} give
\[
 M(\Cyc{3}^2)=3<5=N(\Cyc{3}^2),
\]
whereas $M(\Cyc{2}^3)=N(\Cyc{2}^3)=5$ by Theorem~\ref{thm:exact} and Corollary~\ref{cor:Mfive}. Thus
the exact-realization minimum does not by itself determine the naturalizing
minimum. A metric may have additional isometries while all of its regular
subgroups still recover the required abstract group.
\end{remark}

\subsection{Proof of the exact theorem}\label{sec:exact-proof}
\begin{proof}[Proof of Theorem~\ref{thm:exact}]
By Lemma~\ref{lem:exact-reduction}, $N(G)=k(G)+1$, so it suffices to determine $k(G)$. The empty coloring gives $k(\{e\})=0$; by the same lemma, $k(G)=1$ for $G\cong\Cyc{2},\Cyc{3}$ and $k(G)\geq2$ whenever $|G|>3$. If $|G|>3$ and $G$ is infinite, or finite with $G\notin\mathcal E$, Lemma~\ref{lem:two-color-outside} and the graph case of Lemma~\ref{lem:exact-reduction} give $k(G)\leq2$, hence $N(G)=3$; in particular every infinite group has $N(G)=3$. For the 22 groups of $\mathcal E\setminus\{\Cyc{2}^3,\Cyc{3}^2\}$, Lemma~\ref{lem:certificates} gives $k(G)\leq3$ and Lemma~\ref{lem:exceptional-three} gives $k(G)\geq3$, hence $N(G)=4$. For $\Cyc{2}^3$ and $\Cyc{3}^2$, Theorem~\ref{thm:finite} gives $k(G)\leq4$, while Proposition~\ref{prop:sharp}, with $\Xi(\Cyc{2}^3)=T(\Cyc{2}^3)$ from Lemma~\ref{lem:affine}, and Lemma~\ref{lem:ternary-plane} give $k(G)\geq4$; hence $N(G)=5$, and five is sharp.
\end{proof}

\begin{corollary}[The unique natural group with five values]\label{cor:Mfive}
Among natural groups $G$, $M(G)=5$ if and only if $G\cong\Cyc{2}^3$.
\end{corollary}
\begin{proof}
Let $G$ be natural; by Corollary~\ref{cor:naturalXi}, a metric attaining $N(G)$ is naturalizing, so $M(G)\leq N(G)$, so Theorem~\ref{thm:exact} gives $M(G)\leq4$ unless $G\cong\Cyc{2}^3$ or $G\cong\Cyc{3}^2$, and Proposition~\ref{prop:ternarynat} gives $M(\Cyc{3}^2)=3$. The group $\Cyc{2}^3$ is natural by Theorem~\ref{thm:classification}, and Proposition~\ref{prop:sharp} and Theorem~\ref{thm:exact} give $5\leq M(\Cyc{2}^3)\leq N(\Cyc{2}^3)=5$; both implications follow.
\end{proof}

\subsection{Consequences and exact values for vector spaces}\label{sec:comparisons}

The corollary answers the square-root question in~\cite[Sections~7.5--7.6]{KnillNatural}. For these groups, Theorem~\ref{thm:classification} also supplies a common five-value bound for naturalizing metrics.

\begin{proof}[Proof of Corollary~\ref{cor:fields}]
For a vector space $V$ over $\Q$ or a prime field, doubling is zero in characteristic two and invertible otherwise, so $2V=\{0\}$ or $2V=V$ and the abelian part of Theorem~\ref{thm:classification} makes $(V,+)$ natural; every field is a vector space over its prime field.
\end{proof}
\begin{example}\label{ex:abelian}
For every index set $I$, every odd prime $p$, and every integer $n\geq1$,
the abelian part of Theorem~\ref{thm:classification} applies to
\[
 \Z[1/2]^{(I)},\qquad \prod_{i\in I}\Z[1/2],\qquad
 (\Cyc{p^n})^{(I)}\quad(p\text{ odd}),\qquad
 \bigoplus_{i\in I}\Cyc{2^\infty}.
\]
Doubling is onto in each case: halving is coordinatewise in the first two, two is invertible modulo $p^n$ in the third, and in the fourth, which has nonzero $2$-torsion, doubling is onto each Pr\"ufer group and preserves finite support. To distinguish metric realization from naturality, the metric-realization theorems also apply to the nonnatural groups $\Z$ and $\Cyc{4}$: their unavoidable isometry groups contain nonisomorphic regular subgroups.
\end{example}

Specializing Theorem~\ref{thm:exact}, an $\F_p$-vector space $V$ has $N(V)=1$ if $V=0$; for $p=2$, $N(V)=2,4,5,4$ in dimensions $1,2,3,4$; for $p=3$, $N(V)=2,5,4$ in dimensions $1,2,3$; and $N(V)=3$ in every other dimension, finite or infinite. Hence the least bound uniform over all dimensions is five for $p=2,3$ and three for $p\geq5$.

\appendix

\section{Long branches of the four-color construction}\label{app:branches}

This appendix contains the branches of the proof of Theorem~\ref{thm:finite} that would interrupt Section~\ref{sec:finite}; Table~\ref{tab:branchmap} records which branch each result serves.

\subsection{The abelian building blocks}

In the abelian case we use additive notation, with $\Affpm(A)$ as in Lemma~\ref{lem:affine}. A generating set $X$ is inclusion-minimal if $x\notin\langle X\setminus\{x\}\rangle$ for every $x\in X$.

\begin{lemma}[A rigid rooted tree on a prescribed nonempty set]\label{lem:rigidtree}
Every nonempty set can be made the vertex set of a tree with a distinguished root such that the only automorphism fixing the root is the identity.
\end{lemma}

\begin{proof}
For a finite set use a path rooted at one endpoint. For an infinite cardinal $\kappa$, regarded as its initial ordinal, take the finite strictly decreasing sequences of ordinals below $\kappa$, with root the empty sequence and parent obtained by deleting the last term; there are $\kappa$ of them, since there are $\kappa$ one-entry sequences and $\kappa^{<\omega}=\kappa$. Children of a vertex end in distinct ordinals, and a sequence ending in $\alpha$ has exactly one child ending in each $\beta<\alpha$; since $\sup_{\beta<\alpha}(\beta+1)=\alpha$, induction shows that the well-founded rank of a nonempty sequence is its last term, so distinct children of a vertex have distinct ranks. A rooted automorphism preserves ranks, hence fixes the children of each fixed vertex, and induction on the distance from the root fixes every vertex. Transfer the tree to the prescribed set by a bijection.
\end{proof}

\begin{lemma}[Mixed differences have a unique pair of generators]\label{lem:mixeddiff}
Let $X$ be an inclusion-minimal generating set of an abelian group, and put $S=X\cup(-X)$. The inverse pairs $[x]$, $x\in X$, are disjoint. For distinct $x,y\in X$, the following assertions hold.
\begin{enumerate}
\item Every element $\varepsilon x+\delta y$, with $\varepsilon,\delta\in\{1,-1\}$, is nonzero and lies outside $S$.
\item If
\[
 [\varepsilon x+\delta y]=[\varepsilon' u+\delta' v],\qquad u,v\in X,\quad u\neq v,
\]
then $\{x,y\}=\{u,v\}$.
\item One has
\[
 [x-y]=[x+y]\quad\Longleftrightarrow\quad 2x=0\ \text{or}\ 2y=0.
\]
\item No mixed inverse pair $[\varepsilon x+\delta y]$ equals $[2z]$ for a generator $z\in X$ with $2z\neq0$.
\end{enumerate}
\end{lemma}

\begin{proof}
Each equality to be excluded has a generator $w$ with coefficient $\pm1$ on one side that does not occur on the other; solving for $w$ gives $w\in\langle X\setminus\{w\}\rangle$, contrary to minimality. For the disjointness of the pairs ($x=\pm y$) and the first assertion ($\varepsilon x+\delta y=0$ or $\varepsilon x+\delta y=\pm z$ with $z\in X$), take $w=x$, or $w\in\{x,y\}\setminus\{z\}$ when $z$ occurs; for the second, a generator in the symmetric difference of $\{x,y\}$ and $\{u,v\}$; for the fourth ($\varepsilon x+\delta y=\pm2z$), $w\in\{x,y\}\setminus\{z\}$. For the third, $x-y=x+y$ means $2y=0$, and $x-y=-(x+y)$ means $2x=0$.
\end{proof}

\begin{lemma}[One common local sign is enough]\label{lem:commonsign}
Let $X$ generate an abelian group $A$, put $S=X\cup(-X)$, and let $F\colon A\to A$ be a map. Suppose that for each $u\in A$ there is a sign $\sigma(u)\in\{1,-1\}$ such that
\[
 F(u+s)=F(u)+\sigma(u)s\qquad(s\in S).
\]
Then $F(u)=F(0)+u$ for all $u$, or $F(u)=F(0)-u$ for all $u$. If all generators are involutions, the first alternative holds.
\end{lemma}

\begin{proof}
If $2x=0$ for every $x\in X$, each step is $F(u+s)=F(u)+s$, and summing the steps along words proves the assertion. Otherwise fix $t\in X$ with $2t\neq0$; note that a difference of two signs is $0$ or $\pm2$. For $s\in S$ with $2s\neq0$, the step from $u$ to $u+s$ and the reverse step $-s$ from $u+s$ give $(\sigma(u+s)-\sigma(u))s=0$, hence $\sigma(u+s)=\sigma(u)$. For $s\in S$ with $2s=0$, the two paths from $u$ to $u+s+t$ give
\[
 F(u)+s+\sigma(u+s)t=F(u+s+t)=F(u)+\sigma(u)t+s,
\]
so $(\sigma(u+s)-\sigma(u))t=0$ and, as $2t\neq0$, again $\sigma(u+s)=\sigma(u)$. Since $S$ generates $A$, the sign is constant on $A$, and summing the steps along words gives $F(u)=F(0)+\sigma(0)u$.
\end{proof}

\begin{theorem}[Four colors from a minimal abelian generating set]\label{thm:abelianmin}
Let $A$ be an abelian group admitting an inclusion-minimal generating set. There is an inverse-invariant coloring
\[
 C\colon A\setminus\{0\}\longrightarrow\{1,2,3,4\}\quad\text{such that}\quad \operatorname{Aut}_C(A)=\Affpm(A)=\Xi(A).
\]
Consequently $A$ has a bounded, uniformly discrete, complete translation-invariant metric with isometry group $\Xi(A)$ and at most five values, including zero. The proof applies to arbitrary cardinalities.
\end{theorem}

\begin{proof}
The trivial group is immediate, so choose a nonempty inclusion-minimal generating set $X$. Partition it as $X=I\sqcup J$, where
\[
 I=\{x\in X:2x=0\},\qquad J=\{x\in X:2x\neq0\},
\]
and put $S=X\cup(-X)$.
When these sets are nonempty, Lemma~\ref{lem:rigidtree} supplies rigid rooted trees $U$ on $I$, with root $o$, and $T$ on $J$, with root $r$. On $S$, color $1$ marks $I$ when $I$ and $J$ are both nonempty; otherwise it marks the root pair $[r]$ of $T$ if $I=\varnothing$, and the root $o$ of $U$ if $J=\varnothing$. All other elements of $S$ get color $2$.

Let $D_4$ be the following inverse-invariant subset of $A\setminus(\{0\}\cup S)$:
\[
 D_4=\bigcup_{\{x,y\}\in E(T)}[x-y]\ \cup\bigcup_{\{u,v\}\in E(U)}[u-v],
\]
with the union for an absent tree omitted. If both trees are present, also include $[o-r]=[o+r]$ in $D_4$. Give color $4$ to $D_4$, and color $3$ to every remaining element outside $\{0\}\cup S$. By Lemma~\ref{lem:mixeddiff}, $D_4$ is disjoint from $S$ and no assignment is contradictory, so the coloring is inverse-invariant and $C^{-1}(\{1,2\})=S$.

Let $\Gamma$ be the graph on $S$ in which distinct $s,t$ are adjacent when $C(s-t)=4$, with the vertex colors $C(s)\in\{1,2\}$. By Lemma~\ref{lem:mixeddiff}, distinct pairs of generators give distinct mixed inverse pairs, $[x+y]\neq[x-y]$ for $x,y\in J$, and no difference $2x$ within a pair $[x]$ lies in $D_4$. Hence $\Gamma$ consists of $U$ on $I$, two signed copies of $T$ on $J$ and $-J$, with edges $\{x,y\}$ and $\{-x,-y\}$ for each edge $\{x,y\}$ of $T$, and, when $I$ and $J$ are both nonempty, the two edges from $o$ to $r$ and $-r$. The colors locate the roots: the vertices of color $1$ are $r,-r$ if $I=\varnothing$ and $o$ if $J=\varnothing$; if both occur, color $1$ marks $U$, its root $o$ is its unique vertex adjacent to a vertex of color $2$, and $r,-r$ are the vertices of color $2$ adjacent to a vertex of color $1$. Rooted rigidity therefore fixes $U$ pointwise and either fixes both copies of $T$ pointwise or interchanges them by $x\leftrightarrow-x$, so the color-preserving automorphisms of $\Gamma$ are the identity and the inversion $s\mapsto-s$ of $S$.

Now take $F\in\operatorname{Aut}_C(A)$ and $u\in A$. By Lemma~\ref{lem:normalize}, the normalized map $F_u(z)=F(z+u)-F(u)$ lies in $\operatorname{Aut}_C(A)$ and fixes zero, so it preserves $S=C^{-1}(\{1,2\})$, its vertex colors and $\Gamma$, and is the identity or the inversion on $S$: for one sign $\sigma(u)$,
\[
 F(u+s)=F(u)+\sigma(u)s\qquad(s\in S).
\]
Lemma~\ref{lem:commonsign} gives $F\in\Affpm(A)$. Since $\Affpm(A)\leq\Xi(A)\leq\operatorname{Aut}_C(A)$ by the definition of $\Xi(A)$, the three groups coincide. Lemma~\ref{lem:encoding} gives the metric.
\end{proof}

\begin{lemma}[An elementary decomposition at exponent four]\label{lem:exp4decomp}
Every abelian group $A$ of exponent dividing four is a direct sum of cyclic groups of orders two and four. In particular it admits an inclusion-minimal generating set unless it is trivial.
\end{lemma}

\begin{proof}
By a theorem of Pr\"ufer \cite{Prufer1923} and Baer \cite{Baer1935}, an abelian group of bounded exponent is a direct sum of cyclic groups; the elementary proof below is kept for completeness. Both $2A$ and $A[2]=\{x\in A:2x=0\}$ are vector spaces over $\F_2$, with $2A\subseteq A[2]$. Choose a basis $(b_i)$ of $2A$ and lifts $a_i\in A$ satisfying $2a_i=b_i$. The subgroup $L=\langle a_i\rangle$ is the direct sum of the cyclic groups $\langle a_i\rangle\cong\Cyc{4}$. Indeed, in any finite relation $\sum_i n_ia_i=0$, doubling first forces every $n_i$ to be even; independence of the $b_i$ then forces each $n_i$ to be divisible by four.

We have $2L=2A$ and $L\cap A[2]=2L$. Choose a vector-space complement $K$ to $2L$ in $A[2]$. For every $x\in A$ there is $l\in L$ with $2x=2l$, so $x-l\in A[2]$. Hence $A=L+A[2]=L+K$, while $L\cap K=0$. Thus $A=L\oplus K$. A basis of $K$ together with the generators $a_i$ is the required inclusion-minimal generating set.
\end{proof}

\begin{theorem}[Four colors for all abelian groups]\label{thm:abelian4}
Every abelian group $A$ admits an inverse-invariant coloring
\[
 C\colon A\setminus\{0\}\longrightarrow\{1,2,3,4\}\quad\text{with}\quad \operatorname{Aut}_C(A)=\Affpm(A)=\Xi(A).
\]
Consequently it admits a bounded, uniformly discrete, complete translation-invariant metric with isometry group $\Xi(A)$ and at most five values, including zero.
\end{theorem}

\begin{proof}
The trivial group is immediate. If $A$ contains an element of infinite order or of finite order at least six, apply Theorem~\ref{thm:centralanchor}, which gives $\operatorname{Aut}_C(A)=\Xi(A)$; Lemma~\ref{lem:affine} gives the second equality $\Xi(A)=\Affpm(A)$.

Otherwise every nonzero element has order $2$, $3$, $4$ or $5$. Commuting elements of coprime orders have a sum whose order is their product, which would be at least six, so $A$ has exponent $3$, $5$ or dividing four. A basis over $\F_3$ or $\F_5$, or Lemma~\ref{lem:exp4decomp}, gives an inclusion-minimal generating set, and Theorem~\ref{thm:abelianmin} applies.

In both cases Lemma~\ref{lem:encoding} gives the metric.
\end{proof}

\subsection{The generalized dicyclic exception}

\begin{theorem}[The order-four anchor construction]\label{thm:orderfour}
Let $G=H\sqcup Ht$, where $H$ is an abelian subgroup, $t^2=z\in H$, $z^2=e\neq z$, and $tht^{-1}=h^{-1}$ for every $h\in H$.
If some $a\in H$ satisfies $a^2\notin\{e,z\}$, there is an inverse-invariant coloring
$C\colon G\setminus\{e\}\to\{1,2,3,4\}$ with $\operatorname{Aut}_C(G)=\Xi(G)$.
Its rooted color-preserving stabilizer consists exactly of the identity and the map $\psi$ given by $\psi(h)=h$ and $\psi(ht)=ht^{-1}$ for $h\in H$.
Consequently $G$ has a complete, uniformly discrete, right-invariant metric with at most five values, including zero, and isometry group $\Xi(G)$.
\end{theorem}

\begin{proof}
Let
\[
 Y=\{b\in H:b^2\notin\{e,z\}\}.
\]
This set generates $H$. Indeed, fix $a\in Y$. If $h^2\in\{e,z\}$, then $(ah)^2=a^2h^2\notin\{e,z\}$, so $h=a^{-1}(ah)$ is a product of elements of $Y$.

Retain $t$ first, then, along a well-ordering of $Y$, each element not generated by the elements retained before it; write these later generators as $b_\beta$ $(\beta>0)$. The resulting sequence generates $G$. Take $C=C_D$ of Lemma~\ref{lem:protocol} for this sequence, where $D=\bigcup_\beta D_\beta$ and
\[
 D_0=[t],\qquad D_\beta=[b_\beta]\cup[b_\beta t]\quad(\beta>0).
\]
These blocks lie in the required layers: $b_\beta\notin G_\beta$ and $t\in G_\beta$, so $b_\beta^{\pm1}$ and $(b_\beta t)^{\pm1}$ lie in $G_{\beta+1}\setminus G_\beta$.

Put $T=\langle t\rangle$, a cyclic subgroup of order four, and let $F\in\operatorname{Aut}_C(G)$. By Lemma~\ref{lem:protocol}(i),(ii), each normalization $F_x$ fixes every $D_\beta$ setwise, and $F$ and every $F_x$ are automorphisms of the graph with edges $\{y,ty\}$, whose components are the cycles $Ty$. For a later retained generator $b$, its rank block is
\[
 [b]\cup[bt]=\{b,b^{-1},bt,zbt\}.
\]
The cycles $Tb$ and $Tb^{-1}$ are distinct, because their equality would imply $b^2\in T\cap H=\{e,z\}$. The block meets the first cycle only in $b$, and the second in the three points
\[
 b^{-1},\qquad tb^{-1}=bt,\qquad t^{-1}b^{-1}=zbt,
\]
which sit at positions $0,1,3$ of the cycle $b^{-1},tb^{-1},t^2b^{-1},t^3b^{-1}$ and so induce a path with middle vertex $b^{-1}$. Each $F_x$ preserves the block and permutes the cycles, so it fixes $b$, the singleton, and $b^{-1}$, the middle of the path.

Writing out $F_x(b_\beta)=b_\beta$ for every $x$ gives $F(b_\beta x)=b_\beta F(x)$, so $b_\beta\in Y_F$ for $\beta>0$, where $Y_F$ is the subgroup of Lemma~\ref{lem:protocol}. The antipodal point of the anchor cycle gives the $z$-step: $zx$ is the unique vertex opposite $x$ in $Tx$, so $F(zx)=zF(x)$ and $z\in Y_F$. The retained generators of $H$ may require this additional element $z$. Precisely,
\[
 H=\langle z,b_\beta:\beta>0\rangle,
\]
because words in $t$ and the retained $b_\beta$ can be collected using $tb=b^{-1}t$, and an even power of $t$ is a power of $z$. Hence $H\subseteq Y_F$, that is,
\[
 F(hx)=hF(x)\qquad(h\in H,\ x\in G).
\]

If $F$ is rooted, it fixes $H$ pointwise and $F(t)\in[t]=\{t,t^{-1}\}$; since $F(ht)=hF(t)$, $F$ is the identity or $\psi$. Write $\psi(g)=gz^{\delta(g)}$, where $\delta\colon G\to\mathbb{Z}/2$ is the homomorphism with kernel $H$. Since $z$ is central, $\psi$ is an automorphism; since every element of $Ht$ squares to $z$, $\psi(g)=gz=g^{-1}$ on $Ht$. Hence $\psi(ux)\psi(x)^{-1}=\psi(u)\in[u]$ for all $u,x\in G$, and both candidates belong to $\Xi(G)$. Since $\Xi(G)$ preserves every inverse-invariant coloring, restoring right translations proves $\operatorname{Aut}_C(G)=\Xi(G)$. The metric assertion follows from Lemma~\ref{lem:encoding} applied to the four colors.
\end{proof}

\subsection{Cyclic anchors outside their normalizers}

Lemmas~\ref{lem:orientinf} and~\ref{lem:orientfin} share the following coordinates. Let $A=\langle a\rangle$ and $b\in G$, and let $\Lambda=\mathbb Z$ if $A$ is infinite and $\Lambda=\mathbb Z/N\mathbb Z$ if $|A|=N$. The components of the graph with edges $\{x,ax\}$ are the cosets $Ax$, lines or $N$-cycles; on $Ab$ use the coordinate $k\mapsto a^kb$ ($k\in\Lambda$). Put
\[
 u_k=a^kb,\qquad v_k=u_k^{-1}=b^{-1}a^{-k},\qquad D_P=\{u_k,v_k:k\in P\}\quad(P\subseteq\Lambda\text{ finite}),
\]
\[
 M=\{j\in\Lambda:b^{-1}a^jb\in A\},\qquad E=\{k\in\Lambda:v_k\in Ab\}.
\]
Since $v_kv_l^{-1}=b^{-1}a^{l-k}b$, the subgroup $M$ gives the coset test
\begin{equation}\label{eq:anchor-cosettest}
 Av_k=Av_l\quad\Longleftrightarrow\quad k-l\in M,
\end{equation}
so $E$ is empty or one coset $r+M$. On a nonempty $E$, inversion induces the involution $f$ defined by $v_k=u_{f(k)}$: taking inverses gives $v_{f(k)}=u_k$. Each $D_P$ is inverse-invariant, has at most $2|P|$ points, and meets $Ab$ in $\{u_k:k\in P\cup f(P\cap E)\}$; any selected $u_k$, together with $a$, recovers $b$. The aim is a set $P$ for which $Ab$ is distinguished from every other component meeting $D_P$ by its number of selected points or their diameter, and $D_P\cap Ab$ is a triple whose coordinates have trivial stabilizer among the isometries of the line or cycle. A graph automorphism preserving $D_P$ then maps $Ab$ to itself and fixes it pointwise. On a line, a three-point set has trivial stabilizer exactly when it is not an arithmetic progression.

\begin{lemma}[An oriented pattern outside the normalizer]\label{lem:orientinf}
Let $A=\langle a\rangle$ be infinite cyclic in a group $G$, and let $b\notin N_G(A)$. There is a finite
inverse-invariant set
\[
 D(b)\subseteq Ab\cup(Ab)^{-1}
\]
such that $\langle A,D(b)\rangle=\langle A,b\rangle$ and every automorphism of the graph with edges $\{x,ax\}$ that
preserves $D(b)$ setwise fixes the whole component $Ab$ pointwise. The set $D(b)$ can always be chosen to have at
most six elements.
\end{lemma}

\begin{proof}
Use the common coordinates with $\Lambda=\mathbb Z$. If $M\neq\{0\}$, then $M=m\mathbb Z$, where $m>0$ is the least positive exponent in $M$, and we write
\begin{equation}\label{eq:anchor-conjexp}
 b^{-1}a^mb=a^n\qquad(n\in\mathbb Z\setminus\{0\}).
\end{equation}
By \eqref{eq:anchor-cosettest}, the inverse points then group into components according to their indices modulo $m$.

If $E$ is nonempty, then $E=r+m\mathbb Z$; choose
$q\in\mathbb Z$ with $v_r=a^qb$. Equation \eqref{eq:anchor-conjexp} gives
\begin{equation}\label{eq:anchor-inversecoords}
 v_{r+mj}=a^{q-nj}b\qquad(j\in\mathbb Z).
\end{equation}
Hence $f$ is the affine map
\[
 f(k)=q-\frac nm(k-r)
\]
on $E$. Since $E$ is infinite, comparison of the coefficient of $k$ in $f^2(k)=k$ forces
$n^2=m^2$. It follows that a nonempty overlap is possible only when $n=m$ or $n=-m$.

If $n=m$, this partial inversion is the reflection
\begin{equation}\label{eq:anchor-reflection}
 f(k)=d-k\quad(k\in r+m\mathbb Z),\qquad d=q+r.
\end{equation}
If $n=-m$, it is the translation $k\mapsto k+q-r$. A translation of $\mathbb Z$ whose square is the identity has
zero displacement, so in this case
\begin{equation}\label{eq:anchor-fixed}
 f(k)=k\quad(k\in r+m\mathbb Z).
\end{equation}
In particular, all the $u_k$ in this residue class are involutions.

In the last four rows of the following table $M=m\mathbb Z\neq\{0\}$, and in its last three rows $m\geq2$: $m=1$ and $n=\pm1$ would give $b^{-1}ab=a^{\pm1}$, so that $b$ would normalize $A$. Choose $P$ by the table.
\begin{center}
\begin{tabular}{@{}p{0.22\linewidth}p{0.32\linewidth}p{0.36\linewidth}@{}}
case & $P$ & points of $D_P$ outside $Ab$\\
\hline
$M=\{0\}$ & a translate of $\{0,1,3\}$ missing $E$ & one per component\\
$|n|\neq m$ & $\{0,m,3m\}$ & three in $Ab^{-1}$, diameter $3|n|$\\
$|n|=m$, $E=\varnothing$ & $\{0,1,m+1\}$ & one and two, in two components\\
$n=m$, $E\neq\varnothing$ & $\{k,l\}$ & only $v_l$\\
$n=-m$, $E\neq\varnothing$ & $\{r,r+1,r+m+1\}$ & two, in one component
\end{tabular}
\end{center}
In the first row $E$ has at most one element, and the points $v_k$ lie in distinct components by \eqref{eq:anchor-cosettest}. In the second row $E=\varnothing$ by the overlap calculation, so the inverse points lie in $Av_0=Ab^{-1}\neq Ab$, where \eqref{eq:anchor-conjexp} gives them the coordinates $\{0,-n,-3n\}$ in $k\mapsto a^kb^{-1}$; their diameter is $3|n|\neq3m$. In the third row the inverse points have residues $0,1,1$ modulo $m$. In the fourth row choose $k\in E$ with $k'=d-k\neq k$, where $d$ is as in \eqref{eq:anchor-reflection}; then $k'\in E$ and $v_k=u_{k'}$. Choose $l\notin E$ such that $\{k,k',l\}$ is not an arithmetic progression: among the infinitely many integers outside $E$, only $\frac{k+k'}2$, $2k-k'$ and $2k'-k$ are excluded. Thus $D_P\cap Ab=\{u_k,u_{k'},u_l\}$. In the fifth row $v_r=u_r$ by \eqref{eq:anchor-fixed}, and $r+1$, $r+m+1$ lie outside $E$ in one residue class.

Hence $Ab$ is the only component with three selected points, except in the second row, where it is the only one whose three selected points have diameter $3m$. Its selected coordinates are not an arithmetic progression: their successive gaps are $1,2$ in the first row, $m,2m$ in the second, and $1,m$ with $m\geq2$ in the third and fifth, and $l$ was chosen for this in the fourth. A graph automorphism preserving $D_P$ therefore fixes $Ab$ pointwise. The rows are exhaustive: either $M=\{0\}$, or $|n|\neq m$, or $|n|=m$ and $E=\varnothing$, or $E\neq\varnothing$, and then $n=\pm m$ by the overlap calculation. In every row $|P|\leq3$, so $D(b)=D_P$ has at most six points, and $\langle A,D(b)\rangle=\langle A,b\rangle$ because $D(b)$ contains a point $u_k$. \end{proof}

\begin{lemma}[An oriented finite cyclic pattern]\label{lem:orientfin}
Suppose $A=\langle a\rangle$ has finite order $N\geq6$, and $b\notin N_G(A)$. There is a finite inverse-invariant
set $D(b)\subseteq Ab\cup(Ab)^{-1}$ with at most six points such that $\langle A,D(b)\rangle=\langle A,b\rangle$, and
every automorphism of the graph with edges $\{x,ax\}$ that preserves $D(b)$ fixes the whole component $Ab$
pointwise.
\end{lemma}

\begin{proof}
Use the common coordinates with $\Lambda=\mathbb Z/N\mathbb Z$. Then $M=m\mathbb Z/N\mathbb Z$ for a divisor $m$ of $N$, and $m>1$: if $m=1$, conjugation would send $a$ into $A$, and its image would still have order $N$, forcing $b^{-1}Ab=A$. Every coset of $M$ has $L=N/m$ elements. If $E=r+M$ is nonempty, let $O\subseteq E$ be an orbit of $f$ and put $R=r+1+M$; since $m>1$, $R$ is disjoint from $E$. Choose $P$ by the following table.
\begin{center}
\begin{tabular}{@{}p{0.22\linewidth}p{0.32\linewidth}p{0.36\linewidth}@{}}
case & $P$ & points of $D_P$ outside $Ab$\\
\hline
$E=\varnothing$ & $\{0,1,3\}$ & at most two per component\\
$|E|=1$ & a translate of $\{0,1,3\}$ missing $E$ & one per component\\
$|E|=L\geq2$ & $O$ and $3-|O|$ points of $R$ & one or two, in $Av_{r+1}$
\end{tabular}
\end{center}
The rows are exhaustive, since $|E|$ is $0$ or $L$. In the first two rows the cyclic gaps $1,2,N-3$ of $P$ are pairwise distinct for $N\geq6$, so $P$ has trivial dihedral stabilizer. In the first row the indices $0$ and $1$ differ modulo $m$, so by \eqref{eq:anchor-cosettest} no component other than $Ab$ contains three points of $D_P$; in the second, $m=N$ and the inverse points lie in distinct components. In the third row $|R|=L\geq2$, so the added points exist; $P\cap E=O$ is an entire $f$-orbit, so $D_P\cap Ab=\{u_k:k\in P\}$, and the remaining inverse points lie in $Av_{r+1}\neq Ab$. In every row, therefore, $Ab$ is the only component with three selected points, and it remains to choose the added points of the third row so that $P$ has trivial dihedral stabilizer.

The occupied residue classes $r+M$ and $R$ contain one and two points of $P$, in some order. A translation preserving $P$ must therefore preserve each of them and fix the unique point in the singleton class; hence it is the identity. A reflection $k\mapsto d-k$ preserving $P$ must also preserve each occupied class, which requires
\[
 d\equiv2r\equiv2(r+1)\pmod m.
\]
This is impossible when $m\geq3$, so every choice works then. Finally suppose $m=2$. Now $R$ is the parity class opposite to $E$, with $L=N/2\geq3$ elements, and a preserving reflection fixes the point of $P$ of singleton parity; as it fixes at most two points, it interchanges the other two. If $O=\{k,k'\}$, choose $l\in R$ with
\[
 2l\not\equiv k+k'\pmod N.
\]
The excluded congruence has at most two solutions modulo $N$, whereas $R$ has at least three elements; and it is exactly the condition for a reflection fixing $l$ and interchanging $k,k'$. If $O=\{k\}$, choose distinct $l,l'\in R$ with
\[
 l+l'\not\equiv2k\pmod N.
\]
Such a pair exists among three elements of $R$: if two pairs sharing one element both had the excluded sum, their other elements would be equal. The excluded congruence is exactly the condition for a reflection fixing $k$ and interchanging $l,l'$. Every resulting $D_P$ has at most six points, and $\langle A,D(b)\rangle=\langle A,b\rangle$ because $D_P$ contains a point $u_k$.
\end{proof}

\begin{remark}
The bound $N\geq6$ cannot be lowered to any $N\in\{3,4,5\}$. In a cycle of length at most five, a subset or its complement has at most two points, so it is preserved by a nontrivial reflection; that reflection on $Ab$, with the identity on the other components, is a graph automorphism that preserves any set $D(b)$ and does not fix $Ab$ pointwise.
\end{remark}

\begin{theorem}[Nonnormal cyclic anchors]\label{thm:nonnormal}
Suppose $G$ has a nonnormal cyclic subgroup $A=\langle a\rangle$ of infinite order or finite order at least six.
Then it has an inverse-invariant four-color map with $\operatorname{Aut}_C(G)=R(G)=\Xi(G)$.
\end{theorem}

\begin{proof}
The complement of the proper subgroup $N_G(A)$ generates $G$: fixing $b\notin N_G(A)$, every $h\in N_G(A)$ equals
$(hb)b^{-1}$, with both factors outside $N_G(A)$. Begin a retained sequence with $a$ and choose all later retained
generators from $G\setminus N_G(A)$. For a later $b_\beta$, Lemma~\ref{lem:orientinf} in the infinite case and Lemma~\ref{lem:orientfin} in the finite case supply the marker $D(b_\beta)$; there is at least one later generator, since $A\neq G$. Put
\[
 D_0=[a],\qquad D_\beta=D(b_\beta)\ (\beta>0),\qquad D=\bigcup_\beta D_\beta,
\]
and let $C=C_D$ be the four-color map of Lemma~\ref{lem:protocol}. The blocks are nonempty and inverse-invariant, and each lies in its generator's layer: for $\beta>0$, $D(b_\beta)\subseteq Ab_\beta\cup(Ab_\beta)^{-1}$, while $a\in G_\beta$ and $b_\beta\notin G_\beta$. Their union generates $G$, since $a$ and any selected $a^kb_\beta$ recover $b_\beta$. By Lemma~\ref{lem:protocol}, every normalization $F_x$ of $F\in\operatorname{Aut}_C(G)$ preserves each $D_\beta$, and hence, through $D_0=[a]$, the anchor path or cycle graph; the marker $D(b_\beta)$ then makes $F_x$ fix $Ab_\beta$ pointwise. Lemma~\ref{lem:cosettransfer} gives $F=R_{F(e)}$, so $Y_F=G$, and Lemma~\ref{lem:protocol} gives $\operatorname{Aut}_C(G)=R(G)=\Xi(G)$.
\end{proof}

\subsection{Normal finite anchors with other conjugation actions}

\begin{lemma}[A fully affine-rigid cyclic pattern]\label{lem:affinerigid}
For every $N\geq8$, the subset $P=\{0,1,2,4\}$ of $\mathbb Z/N\mathbb Z$ has trivial stabilizer among affine
permutations $k\mapsto uk+d$, with $u$ a unit modulo $N$. For $N=7$, the subset $\{0,1,3\}$ has trivial dihedral
stabilizer.
\end{lemma}

\begin{proof}
An affine permutation preserving $P$ sends $(0,1,2)$ to an ordered three-term arithmetic progression of distinct
elements of $P$. Since $|x+z-2y|<8\leq N$ for $y\in P$ and distinct $x,z\in P$, reduction modulo $N$ adds no
progression, so the image is $(0,1,2)$, $(2,1,0)$, $(0,2,4)$ or $(4,2,0)$, giving $(d,u)=(0,1),(2,-1),(0,2),(4,-2)$. The second sends $4$ to $-2\notin P$. For $N=8$ the last two slopes are
not units; for $N\geq9$, their images of $4$ are $8$ and $-4$, neither in $P$. Only the identity remains. The last
assertion follows from the distinct cyclic gaps $1,2,4$ of $\{0,1,3\}$ in the seven-cycle.
\end{proof}

\begin{lemma}[Invariant patterns for a non-dihedral affine involution]\label{lem:nondihedral}
Let $N\geq6$, and suppose
\[
 f(k)=qk+c\quad\text{on }\mathbb Z/N\mathbb Z,\qquad f^2=\mathrm{id},\qquad q\not\equiv1,-1\pmod N.
\]
There is an $f$-invariant subset whose dihedral stabilizer is trivial.
\end{lemma}

\begin{proof}
If $f$ fixes a point, translate it to zero, so $f(k)=qk$ and $q^2=1$. The triple $\{0,1,q\}$ is $f$-invariant. A
reflection preserving it fixes one point and interchanges the other two, which requires $q=-1$, $q=2$ or $2q=1$
modulo $N$; the first is excluded, and either of the others with $q^2=1$ gives $N\mid3$. A nontrivial rotation
preserving a three-point set makes it equally spaced, hence gives a reflection. Thus the triple has trivial dihedral
stabilizer.

Otherwise put $d=\gcd(q-1,c,N)$ and $L=N/d$. The relations $q^2=1$ and $(q+1)c=0$ modulo $N$, with B\'ezout's
identity, give $L\mid q+1$. The map $f$ preserves every residue class modulo $d$; on the class of $r$, in coordinates
$k=r+dj$, it is the reflection $j\mapsto c_r-j\pmod L$ with $c_r=((q-1)r+c)/d$. None of these reflections has a fixed
point, so $L$ is even and every $c_r$ is odd; hence $c/d$ is odd and $(q-1)/d$ is even. Then $L=2$ would give
$q\equiv1$ and $d=1$ would give $q\equiv-1\pmod N$, so $L\geq4$ and $d\geq2$.

If $d\geq3$, let $D$ consist of the whole residue-zero class and one $f$-orbit of two nonantipodal points in residue
class one. Such an orbit exists: an antipodal pair requires $c_1-2j\equiv L/2\pmod L$, which excludes at most two of
the $L\geq4$ values of $j$, and none when $L=4$ since $c_1$ is odd. The occupied classes have different sizes $L>2$
and $2$, so a dihedral permutation preserving $D$ preserves each of them. A reflection would then need its center
parameter congruent to both $0$ and $2$ modulo $d$, and a translation preserves the nonantipodal pair, so it is the
identity.

If $d=2$, the parity conclusions above give $q\equiv1\pmod4$ and $c\equiv2\pmod4$. Taking $0\leq q<N=2L$,
the divisibility $L\mid q+1$ and the exclusion $q\neq-1$ force $q=L-1$. Thus
\[
 N=4h,\qquad h\geq3\ \text{odd},\qquad q=2h-1.
\]
Here $\gcd(q-1,N)=4$. Conjugating by a translation can therefore replace $c$ by $2h$, since both are congruent to
$2$ modulo $4$. The Chinese remainder coordinates
$\mathbb Z/N\mathbb Z\cong\mathbb Z/4\mathbb Z\times\mathbb Z/h\mathbb Z$ now give
\[
 f(x,y)=(x+2,-y).
\]
Choose the six-point set with the following nonempty fibers:
\[
 D_0=\{0,2\},\qquad D_1=\{0,1\},\qquad D_{-1}=\{2,3\}.
\]
It is $f$-invariant. A dihedral permutation has the same orientation in both Chinese remainder coordinates. The
fiber at $y=0$ is the only fiber whose two points are antipodal in $\mathbb Z/4\mathbb Z$, so that fiber must be
preserved. A translation then preserves $D_1=\{0,1\}$, forcing its first-coordinate displacement to be zero. A
reflection must preserve $\{0,2\}$, requiring an even first-coordinate center parameter. But exchanging the fibers
at $1$ and $-1$ requires that parameter to be $3$. Translating back proves the lemma.
\end{proof}

\begin{theorem}[A normal cyclic anchor with a nonbinary action]\label{thm:nonbinary}
Let $A=\langle a\rangle$ be a normal cyclic subgroup of finite order $N\geq6$ in $G$. Suppose some element of $G$
conjugates $a$ to $a^r$ with $r\not\equiv1,-1\pmod N$. Then there is an inverse-invariant coloring
\[
 C:G\setminus\{e\}\longrightarrow\{1,2,3,4\},\qquad\operatorname{Aut}_C(G)=R(G)=\Xi(G).
\]
Consequently $G$ has a complete, uniformly discrete, right-invariant naturalizing metric with at most five values,
including zero.
\end{theorem}

\begin{proof}
Write $\rho:G\to(\mathbb Z/N\mathbb Z)^\times$ for the conjugation action, so $gag^{-1}=a^{\rho(g)}$. The subgroup $K=\rho^{-1}(\{1,-1\})$ is proper, and its complement generates $G$: for $y\notin K$, every $h\in K$ equals $(hy)y^{-1}$ with both factors outside $K$. Retain $a$ first, then retain generators from $G\setminus K$. For a
later retained generator $b$, put $r=\rho(b)$, so $r\neq\pm1$.

We specify its finite inverse-invariant rank block $D(b)$. If $Ab\neq Ab^{-1}$, let $P=\{0,1,2,4\}$ when $N\geq8$,
and $P=\{0,1,3\}$ when $N=7$. Use
\[
 D(b)=\{a^kb:k\in P\}\cup\{a^{-r^{-1}k}b^{-1}:k\in P\}.
\]
If $Ab=Ab^{-1}$, write $b^{-1}=a^cb$. Inversion on this component has coordinate form
\[
 f(k)=c-r^{-1}k.
\]
It is an affine involution with multiplier different from $1$ and $-1$. Choose the invariant, dihedrally rigid
subset from Lemma~\ref{lem:nondihedral} as the coordinate set of $D(b)$ in this component. This second case cannot
occur for $N=7$: the equality of the two cosets gives $b^2\in A$, hence $r^2=1$, which over $\mathbb Z/7\mathbb Z$
would force $r=\pm1$. There is no case $N=6$, since all its units are $1$ or $-1$.

Every point of $D(b_\beta)$ lies in $Ab_\beta\cup Ab_\beta^{-1}$ and so, as $a$ precedes $b_\beta$, has the entry rank of $b_\beta$. Let $C=C_D$
be the coloring of Lemma~\ref{lem:protocol} for $D_0=[a]$ and $D_\beta=D(b_\beta)$; by parts (i) and (ii) of that lemma, every normalization $F_x$ of a color preserver $F$ preserves these blocks and the anchor-cycle graph.

First suppose $N\geq8$. A self-inverse coset is fixed pointwise by the chosen pattern. For two distinct cosets, an
exchange would be a cycle isometry $k\mapsto\varepsilon k+d$ sending $P$ to $-r^{-1}P$, so $k\mapsto-r\varepsilon k-rd$
would be an affine permutation preserving $P$; Lemma~\ref{lem:affinerigid} would force $-r\varepsilon=1$, contradicting
$r\neq\pm1$. Thus each component is preserved, and the rigid pattern fixes $Ab$ pointwise.

Now let $N=7$. The two components are distinct, and their selected triples are individually dihedrally rigid. If no
isometry exchanges them, the preceding conclusion holds. Otherwise there is a unique such isometry in each direction,
and its image of $b$ is a fixed element $v=a^db^{-1}$; thus $u(x):=F(bx)F(x)^{-1}\in\{b,v\}$ at every base point.
Since $F$ preserves the anchor-cycle graph, $F(a^kx)=a^{k\sigma(x)}F(x)$ with $\sigma(x)\in\{1,-1\}$, and the relation
$ba^j=a^{rj}b$ gives
\[
 u(a^jx)=a^{rj\sigma(bx)}u(x)a^{-j\sigma(x)}.
\]
So $u(a^jx)$ lies in the $A$-coset of $u(x)$; as $b$ and $v$ lie in distinct $A$-cosets, $u(a^jx)=u(x)$ for every
$j$. If $u(x)=v$, then $v$ conjugates $a$ by $r^{-1}$, and the case $j=1$ gives
\[
 \sigma(x)\equiv r^2\sigma(bx)\pmod7,
\]
impossible because $r\notin\{1,-1\}$ modulo $7$ implies $r^2\notin\{1,-1\}$, whereas both signs are $1$ or $-1$.
Hence $u(x)=b$ at every base point, and $F_x$ fixes $Ab$ pointwise in this case too.

Thus every normalization fixes each $Ab_\beta$ pointwise. Since $A\leq K\neq G$, at least one later retained generator exists, so Lemma~\ref{lem:cosettransfer} makes $F$ a right translation and Lemma~\ref{lem:protocol} gives $\operatorname{Aut}_C(G)=R(G)=\Xi(G)$. Lemma~\ref{lem:encoding} gives the metric, and its regular isometry group gives naturality.
\end{proof}

\subsection{Normal cyclic anchors with signed action}

\begin{lemma}[A witness after the cyclic orientation is fixed]\label{lem:orientwitness}
Let $A=\langle a\rangle$ be normal cyclic, of infinite order or finite order at least six, and let
$waw^{-1}=a^{-1}$. There is an inverse-invariant set $W(w)\subseteq Aw\cup Aw^{-1}$ meeting each occupied $A$-coset
in two or four points such that every automorphism $F$ of the graph with edges $\{x,ax\}$ that preserves $W(w)$
and satisfies $F(ax)=aF(x)$ for all $x$ sends $w$ into $[w]$.
\end{lemma}

\begin{proof}
Use cyclic coordinates in $\Lambda=\mathbb Z$ or $\mathbb Z/N\mathbb Z$ and put $Q=\{0,1,2,3\}$. The only $k\in Q$ with $k-1\notin Q$ is $k=0$ (for finite $N$ because $N\geq6$), so a translation preserving $Q$ fixes $0$ and is the identity: $Q$ has trivial translation stabilizer. Since $F(ax)=aF(x)$, the orientation is fixed: $F$ maps each anchor component it meets to another one by a translation of these coordinates.

If $Aw\neq Aw^{-1}$, use
\[
 W(w)=\{a^kw:k\in Q\}\cup\{a^kw^{-1}:k\in Q\}.
\]
The identity $(a^kw)^{-1}=a^kw^{-1}$ proves inverse-invariance. Rigidity of $Q$ forces each component map either to fix $w$ or to send $w$ to $w^{-1}$.

If $Aw=Aw^{-1}$, then $w^2\in A$. It is fixed by conjugation by $w$, which inverts $A$, so $(w^2)^2=e$. When
$w^2=e$, use the single copy $\{a^kw:k\in Q\}$; the same argument fixes $w$. Otherwise $A$ is finite of even order
$N$ and $w^2=a^{N/2}$. Use $W(w)=[w]$. This set has two points, as $w^2=a^{N/2}\neq e$, and since $F$ preserves it, $F(w)\in W(w)=[w]$.
\end{proof}

\begin{theorem}[Normal cyclic anchors with signed conjugation action]\label{thm:signed}
Suppose $G$ has a normal cyclic subgroup $A=\langle a\rangle$ of infinite order or finite order at least six, and
every element of $G$ conjugates $a$ to $a$ or to $a^{-1}$. Then $G$ admits an inverse-invariant four-color map with
$\operatorname{Aut}_C(G)=\Xi(G)$. If $a$ is not central and the elements outside $C_G(a)$ do not all have the same nonidentity square, the map can be chosen with $\operatorname{Aut}_C(G)=R(G)=\Xi(G)$.
\end{theorem}

\begin{proof}
Put $H=C_G(a)$. If $H=G$, apply Theorem~\ref{thm:centralanchor}. Otherwise $H$ has index two and every element
outside $H$ inverts $A$, while $a$ is central in $H$.

If $A$ has a nonidentity involution $z_A$ and every $t\notin H$ satisfies $t^2=z_A$, fix such a $t$. For $h\in H$,
the equation $(ht)^2=t^2$ gives $tht^{-1}=h^{-1}$. Inversion is then an automorphism of $H$, so $H$ is abelian. The
group is generalized dicyclic. Since the order of $a$ is at least six, $a^2\notin\{e,z_A\}$, and
Theorem~\ref{thm:orderfour} applies.

In every other case choose $t\notin H$ with $t^2\neq z_A$ if $z_A$ exists. If $t^2\in A$, it is both fixed and
inverted by conjugation by $t$, so our choice ensures
\begin{equation}\label{eq:anchor-invertercond}
 t^2\notin A\quad\text{or}\quad t^2=e.
\end{equation}
Choose a retained sequence starting with $a$, continuing until all of $H$ is generated, and ending with $t$. For each
retained $b\in H\setminus A$ use the central-coset marker $D(b)$ of Lemma~\ref{lem:cyclicpatterns}, applied within
$H$. Its statement is local to its cosets, so applies equally to an anchor-graph automorphism of $G$. For the final
block use
\[
 D_t=\{a^kt:k\in P\}\cup\{a^kt^{-1}:k\in P\},\qquad P=\{0,1,3\}\subseteq\Lambda.
\]
Condition \eqref{eq:anchor-invertercond} makes this either two disjoint three-point patterns or one three-point
pattern. The triple $P$ has trivial dihedral stabilizer on a cycle of order at least six and trivial isometry
stabilizer on the line. Thus an anchor-graph automorphism preserving $D_t$ has positive orientation on every
component it meets and sends $t$ into $[t]$.

Let $D$ be the union of $[a]$, all $D(b)$, and $D_t$, each as the rank block of its retained generator, and give $g\neq e$ the color
\[
 C_0(g)=1+\pi(g)+2\mathbf 1_D(g).
\]
Each block is nonempty and inverse-invariant ($[a]=\{a,a^{-1}\}$ by definition, each $D(b)$ by Lemma~\ref{lem:cyclicpatterns}, and $D_t$ because $(a^kt)^{-1}=a^kt^{-1}$ when $t$ inverts $A$), and it lies in $G_{\alpha+1}\setminus G_\alpha$, $\alpha$ the rank of its retained generator, because it lies in the $A$-cosets of that generator and its inverse, where $A\subseteq G_\alpha$ for $\alpha\geq1$ and the generator lies outside $G_\alpha$. Lemma~\ref{lem:protocol} therefore applies: for every $F\in\operatorname{Aut}_{C_0}(G)$ and $x\in G$, the normalization $F_x(u)=F(ux)F(x)^{-1}$ preserves every block, and $F$ is an automorphism of the anchor graph; the final marker gives $F(atx)=aF(tx)$ for every $x$, hence
\[
 F(ay)=aF(y)\qquad(y\in G).
\]
This positive orientation also holds for each $F_x$. It excludes inversion on each commuting-generator marker
$D(b)$, where inversion has negative slope. Therefore all retained $H$-steps are fixed, giving
\begin{equation}\label{eq:anchor-Hsteps}
 F(hx)=hF(x)\qquad(h\in H,\ x\in G).
\end{equation}

If $F$ is rooted, the last marker gives $F(t)\in[t]$. Thus the only possible rooted maps are the identity and
\[
 \psi(h)=h,\qquad\psi(ht)=ht^{-1}\qquad(h\in H).
\]
Suppose the latter is nonidentity and preserves the coloring. The retained $t$-step at base point $t$ gives
$t^3\in\{t,t^{-1}\}$, so $z=t^2$ is a nonidentity involution. The same step at $h\in H$ gives
\[
 tht^{-2}h^{-1}\in\{t,t^{-1}\}.
\]
The first alternative would give $t^2=e$. The second forces $z$ to commute with $h$. Hence $z\in Z(G)$, and
$\psi(g)=gz^{\delta(g)}$, where $\delta:G\to\Cyc{2}$ has kernel $H$. This formula makes $\psi$ an involutive
automorphism. It belongs to $\Xi(G)$ exactly when $w^2=z$ for every $w\notin H$: it fixes $H$ and sends $w$ to $wz$,
which lies in $[w]$ exactly under that square condition.

If this square condition holds, both possible rooted maps are already in $\Xi(G)$. If it fails, choose $w\notin H$
with $w^2\neq z$ and enlarge the final block to $D_t\cup W(w)$, using Lemma~\ref{lem:orientwitness}; this happens only when $t^2=z$ is a nonidentity central involution. The enlarged block keeps the final entry rank: $W(w)$ is inverse-invariant and lies in $Aw\cup Aw^{-1}\subseteq G\setminus H=G_{\beta+1}\setminus G_\beta$, $\beta$ the rank of $t$, so Lemma~\ref{lem:protocol} recovers it with every earlier block. Its cosets are disjoint from $At\cup At^{-1}$: every element of $At\cup At^{-1}$ squares to $z$, every element of $Aw\cup Aw^{-1}$ to $w^{\pm2}\neq z$. Anchor components meet $D_t$ in three points and $W(w)$ in two or four, and an anchor-graph automorphism preserving $D_t\cup W(w)$ preserves these intersection sizes; so it preserves $D_t$ and $W(w)$ separately, and the enlarged marker still recovers the old one.

With the same four-color formula on the enlarged $D$, the argument through \eqref{eq:anchor-Hsteps} applies unchanged and leaves the same two rooted candidates, while the witness requires $F(w)\in[w]$, whereas $\psi(w)=wz\notin[w]$. In every case each rooted color preserver therefore sends every $u$ into $[u]$, and Lemma~\ref{lem:protocol} gives $\operatorname{Aut}_C(G)=\Xi(G)$. If $a$ is not central and the elements outside $H$ do not all have the same nonidentity square, the case treated by Theorem~\ref{thm:orderfour} does not occur, and $\psi$ survives only if $\psi=\id$: a nonidentity $\psi$ preserving $C_0$ makes $z=t^2$ a nonidentity involution, the square condition then fails, and the enlarged coloring excludes $\psi$. Hence every normalization $F_x$ of a color preserver $F$ is the identity, so $F(u)=uF(e)$ and $\operatorname{Aut}_C(G)=R(G)=\Xi(G)$.
\end{proof}

\begin{corollary}[The large-order case]\label{cor:largeorder}
Every group containing an element of infinite order or finite order at least six has an inverse-invariant
four-color map with $\operatorname{Aut}_C(G)=\Xi(G)$.
\end{corollary}

\begin{proof}
By hypothesis $G$ has an element $a$ of infinite order or of finite order at least six; put $A=\langle a\rangle$. If it is not normal, apply Theorem~\ref{thm:nonnormal}. If it is
normal and conjugation acts on $a$ only by $a\mapsto a^{\pm1}$, use Theorem~\ref{thm:signed}. Otherwise $A$ must be
finite, since the only automorphisms of an infinite cyclic group are the two signed maps.
Theorem~\ref{thm:nonbinary} handles this remaining case.
\end{proof}

\subsection{A noncentral subgroup generated by involutions}

Write $\Omega(G)$ for the subgroup generated by all involutions. It is characteristic, although it need not be abelian.

\begin{lemma}[An order-four marker relative to the involution subgroup]\label{lem:orderfourmarker}
Put $H=\Omega(G)$. Suppose $b\notin H$ has order four, and $t\in H$ is an involution not commuting with $b$. Then
\[
 D(b,t)=[b]\cup[tb]\subseteq Hb
\]
has four elements and is inverse-invariant. Every $H$-equivariant permutation of $G$ preserving this set fixes $b$.
\end{lemma}

\begin{proof}
Here $H$-equivariant means $F(hx)=hF(x)$ for $h\in H$ and $x\in G$. Put $z=b^2\in H$ and $\alpha(h)=bhb^{-1}$; normality of $H$ gives $Hb=(Hb)^{-1}$. Neither $b$ nor $tb$ is an involution, since otherwise $b\in H$, and $t\neq e,z$ because $z$ commutes with $b$; so the two inverse pairs are disjoint and have four elements. In coordinates $h\mapsto hb$ on $Hb$ the selected set is
\[
 A=\{e,z,t,f\},\qquad f=b^{-1}tb^{-1}=\alpha^{-1}(t)z.
\]
An $H$-equivariant set preserver maps $Hb$ to itself and acts there by right multiplication by an element of $K=\{c\in H:Ac=A\}$. Since $K$ acts freely on $A$ and $e\in A$, $|K|$ divides four and $K\subseteq A$.

If $z\in K$, which holds when $|K|=4$ because then $K=A$, then $tz\in A\setminus\{e,z,t\}$ gives $tz=f=\alpha^{-1}(t)z$, so $b$ commutes with $t$. If $|K|=2$ and $z\notin K$, then $K$ is $\langle t\rangle$ or $\langle f\rangle$. For $K=\langle t\rangle$, $zt=f$ gives $\alpha^2(t)=ztz=\alpha^{-1}(t)$, so $\alpha^3(t)=t$, and $\alpha^4=\id$ gives $\alpha(t)=t$. For $K=\langle f\rangle$, $t=zf=z\alpha^{-1}(t)z=\alpha(t)$. Each alternative contradicts $tb\neq bt$, so $K=\{e\}$ and the component map fixes $b$.
\end{proof}

\begin{remark}
The hypothesis that $t$ does not commute with $b$ cannot be dropped. If $tb=bt$, the permutation of $G$ that sends $hb$ to $hb^{-1}$ for every $h\in H$ and fixes every point outside $Hb$ is $H$-equivariant and preserves $D(b,t)$, but it moves $b$ to $b^{-1}$.
\end{remark}

\begin{lemma}[Odd-order steps over a fixed normal subgroup]\label{lem:oddsteps}
Let $H\trianglelefteq G$, and let $b$ have finite odd order and fail to centralize $H$. Suppose a permutation $F$ satisfies
\[
 F(hx)=hF(x)\quad(h\in H,\ x\in G),\qquad F(bx)\in[b]F(x)\quad(x\in G).
\]
Then $F(bx)=bF(x)$ for every $x$.
\end{lemma}

\begin{proof}
Suppose $F(bx)=b^{-1}F(x)$ for some $x$, and let $h\in H$. As $bhb^{-1}\in H$, equivariance gives $F(bhx)=(bhb^{-1})F(bx)=bhb^{-2}F(x)$, while the hypothesis at $hx$ gives $F(bhx)\in\{bhF(x),b^{-1}hF(x)\}$. The first value would give $b^2=e$, impossible because $b\neq e$ has odd order. Hence $bhb^{-1}=b^{-1}hb$ for every $h\in H$, so $b^2$ centralizes $H$. Since $\langle b^2\rangle=\langle b\rangle$, so does $b$, a contradiction.
\end{proof}

\begin{theorem}[Four colors when involutions generate a noncentral subgroup]\label{thm:noncentral}
Suppose every element of $G$ has order in $\{1,2,3,4,5\}$, and let $\Omega$ be the subgroup generated by all involutions of $G$. If $\Omega$ is noncentral, there is an inverse-invariant four-color map $C$ on $G\setminus\{e\}$ with
\[
 \operatorname{Aut}_C(G)=R(G)=\Xi(G).
\]
No cardinality or minimal-generation hypothesis is required.
\end{theorem}

\begin{proof}
The subgroup $\Omega$ is characteristic, hence normal, and $K=C_G(\Omega)$ is proper. The complement $G\setminus K$ generates $G$: for fixed $c\notin K$, every $k\in K$ is $(kc)c^{-1}$ with both factors outside $K$. Retain involutions until $\Omega$ is generated, then elements of $G\setminus K$ until $G$ is generated. Each later retained $b$ lies outside $\Omega$ and fails to centralize it, so it has order three, four or five. The blocks are $\{t\}$ for each retained involution $t$, $D_b=[b]$ for each later $b$ of odd order, and $D_b=[b]\cup[t_bb]$ for each later $b$ of order four, where $t_b$ is an involution not commuting with $b$; it exists because the involutions generate $\Omega$. Since $t_b$ lies in the earlier subgroup $\Omega$, the pair $[t_bb]$ enters at the rank of $b$ (Lemma~\ref{lem:rank}). So there is one block for each rank, and these blocks satisfy the hypotheses of Lemma~\ref{lem:protocol}; let $C=C_D$ be its rank-block coloring, which takes only the values $1,2,3,4$.

Let $F\in\operatorname{Aut}_C(G)$. By Lemma~\ref{lem:protocol}(iii) the retained involutions lie in the subgroup $Y_F$, so $\Omega\subseteq Y_F$; hence every normalization $F_x$ is $\Omega$-equivariant, and by Lemma~\ref{lem:protocol}(i) it preserves every block. For $b$ of odd order, $F_x(b)\in[b]$ for all $x$ means $F(bx)\in[b]F(x)$, and Lemma~\ref{lem:oddsteps} with $H=\Omega$ gives $b\in Y_F$. For $b$ of order four, Lemma~\ref{lem:orderfourmarker} applied to each $F_x$ gives $F_x(b)=b$, that is, $b\in Y_F$. Thus $Y_F=G$, and Lemma~\ref{lem:protocol} gives $\operatorname{Aut}_C(G)=R(G)=\Xi(G)$.
\end{proof}

\subsection{A large central involution subgroup}

\begin{theorem}[Central involution subgroup of size at least eight]\label{thm:centrallarge}
Suppose every element of $G$ has order dividing four, and $H=\Omega(G)$ is central with $|H|\geq8$. Then $G$ admits an inverse-invariant four-color map with $\operatorname{Aut}_C(G)=\Xi(G)$. No cardinality restriction or minimal-generating-set assumption is needed.
\end{theorem}

\begin{proof}
The subgroup $H$ is Boolean, $G/H$ is Boolean because every square lies in $H$, and every element outside $H$ has order four. Retain a basis of $H$ first, with singleton blocks, then lifts $b_\beta$ of a basis of $G/H$. A later rank layer is a union of $H$-cosets, each equal to its inverse. In $Hb_\beta$ select $[b_\beta]$, two points. In any other coset $Hw$ of that layer put $z=w^2\in H\setminus\{e\}$. As $|H|\geq8$, the Boolean group $H/\langle z\rangle$ has at least four elements; choose three of them and let $A_w\subset H$ be their full preimage, so $|A_w|=6$. A nonzero translation of $H/\langle z\rangle$ has two-element orbits and so preserves no three-point set; hence the translation stabilizer of $A_w$ in $H$ is exactly $\langle z\rangle$. Select $A_ww$; it is inverse-invariant because inversion on $Hw$ is multiplication by $z$. The block $D_\beta$ of the layer is the union of its selected sets, and need not be finite. With the singleton blocks, one block for each rank, these blocks satisfy the hypotheses of Lemma~\ref{lem:protocol}; let $C=C_D$ be its rank-block coloring.

Let $F\in\operatorname{Aut}_C(G)$. By Lemma~\ref{lem:protocol}(iii) and~(i), $H\subseteq Y_F$, so $F$ and every normalization $F_x$ permute the $H$-cosets, and every $F_x$ preserves every $D_\beta$, hence the number of points of $D_\beta$ in each coset. As $Hb_\beta$ is the only coset meeting $D_\beta$ in two points (every other coset meets it in six points or none), $F_x(Hb_\beta)=Hb_\beta$, that is, $F(b_\beta x)\in Hb_\beta F(x)$ for all $x$. So the permutation induced by $F$ on $G/H$ commutes with translation by every basis element $b_\beta H$; when $F$ is rooted it is the identity, since these elements generate $G/H$. Thus every rooted color preserver fixes every $H$-coset.

Now let $F$ be rooted and $w\notin H$. Then $F(hw)=hcw$ for $h\in H$ and some $c\in H$, and $c$ stabilizes the selected set of $Hw$ in the $H$-coordinate. Outside $Hb_\beta$ this stabilizer is $\langle w^2\rangle$ by the calculation above; in $Hb_\beta$ it is $\langle b_\beta^2\rangle=\langle w^2\rangle$, since $H$ is central and Boolean. Hence $F(w)\in\langle w^2\rangle w=[w]$, and the elements of $H$ are fixed. Applied to every normalization $F_x$, this gives $F_x(u)\in[u]$ for all $u,x$, and Lemma~\ref{lem:protocol} gives $\operatorname{Aut}_C(G)=\Xi(G)$.
\end{proof}

\subsection{The remaining small central involution subgroups}

\begin{lemma}[A dimension bound for an anisotropic quadratic map]\label{lem:anisotropic}
Let $V$ and $W$ be vector spaces over $\mathbb{F}_2$, with $\dim W=d<\infty$, and let $q:V\to W$ be quadratic, with $q(0)=0$ and $q(v)\neq0$ for every $v\neq0$. Then $\dim V\leq 2d$.
\end{lemma}

\begin{proof}
If $V$ contains a subspace of dimension $2d+1$, choose coordinates on that subspace and write the coordinate functions of $q$ as polynomials $q_1,\dots,q_d$ of degree at most two. The polynomial
\[
 P(x)=\prod_{j=1}^{d}\bigl(1+q_j(x)\bigr)
\]
is the indicator, with values in $\mathbb{F}_2$, of $q(x)=0$. Its degree is at most $2d$, so each of its monomials omits one of the $2d+1$ variables, and summing over that variable shows that the monomial sums to zero over $\mathbb{F}_2^{2d+1}$. Hence, by the Chevalley--Warning argument \cite{Chevalley1935,Warning1935}, $q$ has an even number of zeros on the subspace, contradicting that its only zero is the origin.
\end{proof}

\begin{lemma}[A spanning fiber in a two-dimensional target]\label{lem:spanningfiber}
Let $q:V\to\mathbb{F}_2^2$ be quadratic and anisotropic, and suppose $\dim V\geq3$. Some nonzero fiber $q^{-1}(z)$ spans $V$.
\end{lemma}

\begin{proof}
Let $B(u,v)=q(u+v)+q(u)+q(v)$ be the bilinear polar map, and suppose no nonzero fiber spans $V$. Place each of the three nonzero fibers $q^{-1}(z_i)$ inside a proper hyperplane $K_i$; these hyperplanes cover $V$, since the nonzero fibers cover $V\setminus\{0\}$. Their defining nonzero linear forms cannot repeat, and cannot be independent, since independent forms take the value one simultaneously somewhere; so the third form is the sum of the first two, and $K_1,K_2,K_3$ are the three hyperplanes containing a common subspace $L$ of codimension two. Each $K_i\setminus L$ lies in no other $K_j$, so $q$ takes the value $z_i$ on it. Choose $a\in K_1\setminus L$ and $b\in K_2\setminus L$; then $a+b\in K_3\setminus L$. For $h\in L$, constancy of $q$ on $a+L$, $b+L$ and $(a+b)+L$ gives $B(a,h)=B(b,h)=B(a+b,h)=q(h)$, whereas bilinearity gives $B(a+b,h)=B(a,h)+B(b,h)=0$. Hence $q$ vanishes on $L$, which is nonzero because $\dim V\geq3$, contradicting anisotropy.
\end{proof}

\begin{lemma}[A same-square basis with a compatible path]\label{lem:samesquare}
Let $q:V\to\mathbb{F}_2^2$ be quadratic and anisotropic, with $3\leq\dim V\leq4$. There are a nonzero $z\in\mathbb{F}_2^2$ and an ordered basis $v_1,\dots,v_n$ of $V$ such that
\[
 q(v_i)=z\quad(1\leq i\leq n),\qquad q(v_i+v_{i+1})\neq z\quad(1\leq i<n).
\]
\end{lemma}

\begin{proof}
For a basis inside a spanning fiber $q^{-1}(z)$, supplied by Lemma~\ref{lem:spanningfiber}, join two members $u,v$ when $q(u+v)\neq z$, equivalently $B(u,v)\neq z$; for distinct members anisotropy also gives $B(u,v)=q(u+v)\neq0$. This graph has no independent triple: such a triple $a,b,c$ would have $B(a,b)=B(a,c)=B(b,c)=z$ and therefore
\[
 q(a+b+c)=3z+3z=0,
\]
contradicting linear independence and anisotropy. A connected graph on three or four vertices with no independent triple has a spanning path, and such a path gives the required ordering. A disconnected one has exactly two components, both cliques, since three components, or a nonadjacent pair in one component and a vertex in another, would give an independent triple; the possibilities are $K_1\sqcup K_2$ for $n=3$, and $K_1\sqcup K_3$ or $K_2\sqcup K_2$ for $n=4$. A basis change inside the fiber makes the graph connected, as follows.

If $a$ is isolated, choose $b$ in the other component and replace $a$ by $a+b$; then $q(a+b)=z$, and for every other vertex $c\neq b$ of that clique
\[
 B(a+b,c)=z+B(b,c)\neq z,
\]
since $B(b,c)\neq0$, so the new vertex is joined to the rest of that clique. In the case $K_2\sqcup K_2$, with edges $\{a,b\}$ and $\{c,d\}$, the values $B(a,b)$ and $B(c,d)$ are nonzero and different from $z$, and
\[
 q(a+b+c+d)=B(a,b)+B(c,d)\neq0.
\]
They are therefore the two different nonzero values other than $z$, so $e'=a+b+c+d$ has $q(e')=z$; replacing $a$ by $e'$ gives a vertex adjacent to all three others.
\end{proof}

\begin{lemma}[A four-color pattern from a common-square basis]\label{lem:commonsquare}
Suppose $H\leq Z(G)$ is Boolean and $G/H$ is a Boolean group with a finite ordered $\mathbb{F}_2$-basis $a_1H,\dots,a_nH$, where $n\geq2$, and the lifts $a_1,\dots,a_n$ generate $G$. Suppose there is a nonidentity $z\in H$ such that
\[
 a_i^2=z\quad(1\leq i\leq n),\qquad (a_ia_{i+1})^2\neq z\quad(1\leq i<n).
\]
Then $G$ has an inverse-invariant four-color map with $\operatorname{Aut}_C(G)=\Xi(G)$.
\end{lemma}

\begin{proof}
Put $S=\{a_i,a_i^{-1}:1\leq i\leq n\}$, and use word length relative to $S$ throughout this proof. Projection onto the basis of $G/H$ shows that the generator pairs are disjoint, every mixed two-generator label lies outside $S\cup\{e\}$, and mixed labels from different two-element supports are distinct. Because $z$ is central and $a_i^{-1}=za_i$, every pair $a,b$ of these generators satisfies
\[
 [ab^{-1}]=[a^{-1}b]=[zab],\qquad [ab]=[a^{-1}b^{-1}].
\]
These two labels coincide exactly when $(ab)^2=z$, since $zab\neq ab$ and $zab=(ab)^{-1}$ is this square condition; so they are distinct on every consecutive support.

Give $[a_1]$ color $1$, the other members of $S$ color $2$, the labels $[a_ia_{i+1}^{-1}]$ ($1\leq i<n$) color $4$, and every other nonidentity element outside $S$ color $3$. The support facts make this inverse-invariant coloring $C_0$ well-defined. On $S$, the graph joining $u,v$ when $C_0(uv^{-1})=4$ consists of two copies of the path $a_1,\dots,a_n$, with their endpoints $a_1,a_1^{-1}$ distinguished by vertex color $1$. Its color-preserving automorphisms are exactly the identity and simultaneous interchange of the two copies.

Applied to every rooted normalization of a complete color preserver $F$, this gives one local sign $\sigma(x)\in\{1,-1\}$ with
\[
 F(a_i^{\pm1}x)=a_i^{\pm\sigma(x)}F(x)\qquad(1\leq i\leq n).
\]
Comparing an edge with its reverse yields $a_i^{\sigma(x)-\sigma(a_ix)}=e$; since $a_i$ has order four, the two signs agree, and generation makes $\sigma$ constant. Let $\lambda:G\to G/H\cong\mathbb{F}_2^n\to\mathbb{F}_2$ be the sum of the quotient coordinates. The map
\[
 \theta(g)=gz^{\lambda(g)}
\]
is an involutive automorphism: $z$ is central and $\lambda(z)=0$. Moreover $\theta(a_i)=a_i^{-1}$. A rooted preserver is therefore the identity or this explicitly defined $\theta$.

The automorphism $\theta$ preserves every inverse pair of word length at most two. If $\theta$ belongs to $\Xi(G)$, the unavoidable inclusions finish the proof. Otherwise choose $w$ with $[\theta(w)]\neq[w]$. Both labels have word length at least three: $\theta$ preserves word length and the short labels just listed. Change the color of $[w]$ to $4$ and leave $[\theta(w)]$ at color $3$. Every comparison $uv^{-1}$ with $u,v\in S$ has word length at most two, so the generator graph and its vertex colors are unchanged, its only possible rooted preservers are still the identity and $\theta$, and the changed colors exclude $\theta$. In every case the resulting complete color-preserving group is $\Xi(G)$.
\end{proof}

\begin{lemma}[The two quaternion building blocks]\label{lem:quaternionblocks}
The groups $Q_8$ and $Q_8\times\Cyc{2}$ admit inverse-invariant four-color maps with complete color-preserving group $\Xi(G)$.
\end{lemma}

\begin{proof}
For $Q_8$, give its four nonidentity inverse pairs distinct colors. For $Q_8\times\langle t\rangle$, with $t^2=e$ central, write the quaternion generators as $i,j$, put $z=i^2=j^2$, and set $S=\{t,i,i^{-1},j,j^{-1}\}$. Give $t$ color $1$ and the other members of $S$ color $2$. Outside $S\cup\{e\}$, give color $4$ to $[ti]\cup[ij]$ and color $3$ everywhere else.

The color-$4$ graph on $S$ has the two edges from $t$ to $[i]$ and all four edges between $[i]$ and $[j]$, and no others. The vertex $t$ is fixed, its two neighbors recover $[i]$, and the other two vertices recover $[j]$. Every normalized complete color preserver therefore fixes the $t$-step and preserves the $i$- and $j$-steps up to inversion. Reciprocity along $i$-edges fixes the $z$-step, and all four products $i^{\pm1}j^{\pm1}$ lie in $[ij]$. Thus every quaternion inverse pair is preserved at every base point. The exact central $t$-step then gives the same conclusion on its other coset. This is precisely membership in $\Xi(G)$; the opposite inclusion is unavoidable.
\end{proof}

\begin{theorem}[Exponent four with at most four central involutions]\label{thm:centralsmall}
Let $G$ have exponent dividing four, and suppose the subgroup $H=\Omega(G)$ generated by all involutions is central and has order at most four. Then $G$ has an inverse-invariant four-color map with complete color-preserving group $\Xi(G)$.
\end{theorem}

\begin{proof}
The group $H$ is Boolean, every square belongs to $H$, and $V=G/H$ is Boolean. In particular commutators are central, and
\[
 q:V\longrightarrow H,\qquad q(gH)=g^2
\]
is well-defined and quadratic. Its polar map is the commutator map, and it is anisotropic: $q(gH)=e$ implies $g\in H$. We use additive notation for $V,H$ when discussing $q$.

Lemma~\ref{lem:anisotropic} gives $n=\dim V\leq2\dim H\leq4$; in particular $n=0$ if $|H|=1$, $n\leq2$ if $|H|=2$, and $|H|=4$ if $n\geq3$. For $n\leq1$ the group $G$ is abelian, since $G/H$ is cyclic and $H$ is central, and Theorem~\ref{thm:abelian4} applies. For $|H|=2$ and $n=2$, all three nonzero values equal the unique nonzero element $z\in H$, so basis lifts $a,b$ satisfy $a^2=b^2=(ab)^2=z$ and generate $Q_8=G$, covered by Lemma~\ref{lem:quaternionblocks}. For $|H|=4$ and $n=2$, three distinct values are the three nonzero elements of $H$, whose sum is zero; hence the polar form on a basis is zero, the basis lifts commute, and $G$ is abelian because $H$ is central, so Theorem~\ref{thm:abelian4} applies. If two of the three values coincide, choose distinct nonzero $u,v$ with $q(u)=q(v)=z$; they form a basis, with lifts $a,b$. If $q(u+v)=z$, then $a,b$ generate a quaternion subgroup $K\cong Q_8$ with $K\cap H=\langle z\rangle$; since $G=KH$ and $H$ is central, a complement to $\langle z\rangle$ in $H$ gives $G\cong Q_8\times\Cyc{2}$, covered by Lemma~\ref{lem:quaternionblocks}. If instead $q(u+v)=w\neq z$, the subgroup generated by $a,b$ contains $z=a^2$ and $w=(ab)^2$, which generate $H$, so $a,b$ generate $G$, and this common-square basis, in either order, satisfies Lemma~\ref{lem:commonsquare}.

For $n=3$ or $4$, Lemma~\ref{lem:samesquare} supplies an ordered basis $v_1,\dots,v_n$ with common nonzero square $z$ and $q(v_i+v_{i+1})\neq z$. Lift it to $a_1,\dots,a_n$ and put $K=\langle a_1,\dots,a_n\rangle$. We have $G=KH$; centrality and exponent two of $H$ give $(kh)^2=k^2$. Every value of $q$ therefore belongs to $K$. Those values span $H$, since otherwise $q$ would be anisotropic with a target of dimension at most one, contradicting Lemma~\ref{lem:anisotropic} and $n\geq3$. Consequently $H\leq K$ and $K=G$. All the hypotheses of Lemma~\ref{lem:commonsquare} now hold, completing the proof.
\end{proof}

\subsection{The residual odd element orders}

\begin{lemma}\label{lem:squareclosed}
Let every element of $G$ have finite odd order. Every subgroup $H$ is square-closed: $x^2\in H$ implies $x\in H$. No nonidentity element is conjugate to its inverse. Consequently one may choose $P\subseteq G\setminus\{e\}$ such that
\[
 G\setminus\{e\}=P\sqcup P^{-1},\qquad gPg^{-1}=P\quad(g\in G).
\]
\end{lemma}

\begin{proof}
An element $x$ of odd order is a power of $x^2$, giving square-closure. If $hxh^{-1}=x^{-1}$, iterating through the odd order of $h$ gives $x=x^{-1}$ and hence $x=e$. Thus inversion pairs distinct nontrivial conjugacy classes, and choose one class from each such pair.
\end{proof}

\begin{lemma}[Exact collision table]\label{lem:collision}
Under the hypotheses and choice of $P$ in Lemma~\ref{lem:squareclosed}, let $H\leq G$, $b\notin H$, and $a,c\in P\cap H$. Define
\[
 P_a=[ab^{-1}],\quad T_a=[a^{-1}b],\quad Q_a=[ab],\quad R_a=[a^{-1}b^{-1}].
\]
Within each of the four types, equality of labels occurs exactly when $a=c$. The only possible equalities between different types are
\[
 P_a=T_c\iff c=bab^{-1},\qquad Q_a=R_c\iff c=b^{-1}ab.
\]
In particular, no label of type $P$ or $T$ equals a label of type $Q$ or $R$.
\end{lemma}

\begin{proof}
First note that $bub\notin H$ for $u\in H$: otherwise $(bu)^2=bubu\in H$, so square-closure would put $bu$, and then $b$, in $H$. The same statement holds with $b^{-1}$ in place of $b$. Compare the representatives either directly or after inversion. For example, direct equality in $P_a=T_c$ would give $b^2=ca\in H$, which is impossible; equality with the inverse representative is $ab^{-1}=b^{-1}c$, equivalent to $c=bab^{-1}$. The second displayed equivalence follows similarly. Direct same-type equality gives $a=c$; its inverted alternative would put an element of the form $b^{\pm1}ub^{\pm1}$ in $H$. All remaining comparisons either give $b^{\pm2}\in H$, put such a same-sign double product in $H$, or conjugate a member of $P$ to the inverse of a member of $P$. Each is excluded by the preceding observation or by conjugacy-invariance of $P$.
\end{proof}

\begin{theorem}[The residual orders three and five]\label{thm:residual}
Suppose every element of $G$ has order in $\{1,3,5\}$. Then $G$ admits an inverse-invariant four-color map with complete color-preserving group $\Xi(G)$, without a cardinality or minimal generation assumption. If $G$ is nonabelian, the map can be chosen with $\operatorname{Aut}_C(G)=R(G)$; in particular $\Xi(G)=R(G)$.
\end{theorem}

\begin{proof}
The abelian case is covered by Theorem~\ref{thm:abelian4}. Suppose $G$ is nonabelian. Choose the conjugacy-invariant inverse-half $P$ from Lemma~\ref{lem:squareclosed}. Choose noncommuting $a_0,a_1\in P$ and extend them to a retained generating sequence $(a_\alpha)_{\alpha<\tau}$ in $P$. Write $S=\{a_\alpha^{\pm1}:\alpha<\tau\}$ and $\epsilon_\alpha=\pi(a_\alpha)$, so that $\epsilon_0=0$, $\epsilon_1=1$, successor bits alternate, and the bits at nonzero limit ordinals are zero.

\emph{Consistency of the color prescriptions.} Mixed labels of different highest entry ranks are disjoint, and they do not meet $S$: a mixed product of an earlier $a$ and a later $b$ lying in $S$ would equal $b^{\pm1}$ by entry rank, giving $a=e$ or $a=b^{\pm2}$, contrary to square-closure of the earlier subgroup.

No mixed label meets the square label $[c^2]$ of an order-five retained generator $c$. As $\langle c^2\rangle=\langle c\rangle$, entry rank forces $c$ to be the later generator $b$, and $a^\eta b^\delta=b^{2\zeta}$ with $\eta,\delta,\zeta\in\{1,-1\}$ would give $a^\eta=b^{2\zeta-\delta}$, whose exponent $\pm1$ or $\pm3$ is prime to five, putting $b$ in the earlier subgroup. Square labels of distinct order-five retained generators are disjoint by entry rank.

Define an initial bit $\kappa$ on inverse pairs outside $S$. For $a=a_\alpha$, $b=a_\beta$ with $\alpha<\beta$, prescribe
\[
 \kappa([ab^{-1}])=\kappa([a^{-1}b])=\epsilon_\beta,\qquad \kappa([ab])=\kappa([a^{-1}b^{-1}])=1-\epsilon_\beta.
\]
The exact collision table of Lemma~\ref{lem:collision} makes these prescriptions consistent: collisions occur only within one of the two displayed sectors and at a common highest rank. Also prescribe, for each order-five retained generator,
\[
 \kappa([a_\alpha^2])=1-\epsilon_\alpha.
\]
Assign zero to remaining inverse pairs outside $S$.

Finally change only
\[
 L=[a_0a_1^{-1}]\quad\text{from bit $1$ to bit $0$.}
\]
This changes exactly one undirected edge between vertices of $S$: the edge $\{a_0,a_1\}$. Indeed its entry rank is one, so a mixed support carrying it can only be $\{a_0,a_1\}$; the noncommutation of $a_0,a_1$ makes all four entries of that support distinct by Lemma~\ref{lem:collision}. Set
\[
 C(g)=
 \begin{cases}
  1+\epsilon_\alpha, & g\in[a_\alpha],\\
  3+\kappa([g]), & g\notin S\cup\{e\}.
 \end{cases}
\]
It is an inverse-invariant map into $\{1,2,3,4\}$. Let $F$ be a rooted complete color-preserving permutation. The colors $1,2$ recover $S$ and its vertex bits.

\emph{Recovering the inverse pairs.} On $S$, join two distinct vertices when their difference has color $1$ or $2$. These edges are exactly $\{a_\alpha,a_\alpha^{-1}\}$ for order-three generators, since mixed differences avoid $S$ and an order-five generator has its square outside $S$; the remaining vertices $U$ are the order-five vertices.

Call distinct $u,v\in U$ of the same vertex color complementary if $uw^{-1}$ and $vw^{-1}$ have different colors for every $w\in S\setminus\{u,v\}$. Before the change every order-five inverse pair is complementary, since changing the sign of an endpoint exchanges the two mixed bits; after it, precisely the order-five pairs among $[a_0],[a_1]$ lose this property.

No pair from distinct retained ranks is complementary, even after the change. With $\kappa_0$ the bit before the change, the cross-rank formula is
\[
 \kappa_0\Bigl(\Bigl[a_\alpha^{(-1)^s}\bigl(a_\beta^{(-1)^t}\bigr)^{-1}\Bigr]\Bigr)=\epsilon_{\max\{\alpha,\beta\}}+s+t\pmod 2\qquad(\alpha\neq\beta).
\]
Now let
\[
 u=a_\alpha^{(-1)^s},\qquad v=a_\beta^{(-1)^t},\qquad \alpha<\beta,\quad \epsilon_\alpha=\epsilon_\beta,\quad s,t\in\mathbb{F}_2.
\]
At $w=u^{-1}$ the two edge bits sum to $s+t$ modulo two, and neither edge is the changed one, whose endpoints have different vertex bits.

Put $\gamma=\alpha+1<\beta$, so $\epsilon_\gamma=1-\epsilon_\alpha$. At both witnesses $w=a_\gamma^{\pm1}$ the sum before the change is $1+s+t$, and since only one undirected edge was changed, one of them keeps this sum. The sums $s+t$ and $1+s+t$ are not both one, so $u,v$ are not complementary.

Consequently the complementary relation recovers every order-five pair except those among $[a_0],[a_1]$. The unpaired vertices are exactly these exceptional pairs, zero or two of each vertex color because $\epsilon_0\neq\epsilon_1$, and pairing the two of each color recovers them. Every inverse pair $[a_\alpha]$ is now an intrinsic block, although its rank is not yet recovered.

\emph{Fixing the marked generators and all orientations.} Between two recovered inverse pairs the matrix of colors $3,4$ is a checkerboard, except between $[a_0]$ and $[a_1]$, where it has three entries of color $3$ and one of color $4$. This unique unbalanced pair of blocks is fixed blockwise, its two blocks having different vertex colors, and its unique color-$4$ entry, with endpoints $a_0^{-1},a_1^{-1}$, fixes these two vertices and hence $a_0,a_1$.

For $2\leq\alpha<\tau$ the positive representative $a_\alpha$ of the block $[a_\alpha]$ is distinguished by
\[
 \kappa([a_0a_\alpha^{-1}])=\epsilon_\alpha.
\]
The negative representative gives the complementary bit. Since $a_0$ is fixed and vertex colors are preserved, $F$ maps the positive tail $X_*=\{a_\alpha:2\leq\alpha<\tau\}$ to itself, preserving its vertex bits.

\emph{Recovering the tail order.} For opposite-bit $u=a_\alpha$, $v=a_\beta$ in $X_*$, the difference $uv^{-1}$ is mixed, does not carry the changed label, and has bit $\kappa([uv^{-1}])=\epsilon_{\max\{\alpha,\beta\}}$, that of the larger rank. Hence the relation
\[
 u\prec v\iff\kappa([uv^{-1}])=\pi(v),
\]
which $F$ preserves, holds exactly when $\alpha<\beta$: every opposite-bit comparison points from the smaller rank to the larger. By Lemma~\ref{lem:sources}, applied to the singleton blocks $\{a_\alpha\}$, $2\leq\alpha<\tau$, whose bits alternate at successors, $F$ fixes every tail generator. Together with the two marked generators, $F$ fixes all of $X=\{a_\alpha:\alpha<\tau\}$ pointwise.

As $F$ permutes the recovered inverse pairs, it also fixes each $a_\alpha^{-1}$, hence $S$ pointwise. Since $C$ recognizes $S$ by the colors $1,2$, Lemma~\ref{lem:normalize} gives $\operatorname{Aut}_C(G)=R(G)=\Xi(G)$ in the nonabelian case.
\end{proof}
\section{Conventions for the exceptional groups}\label{app:certificates}

The groups in \eqref{eq:exceptional-set} are taken in the following models. The ancillary README specifies the serialization and multiplication conventions and the verified identifications of the certificate models with these abstract groups.

\subsection{Group and element conventions}\label{sec:certificate-conventions}
For a finite non-Boolean abelian group $A$ and an involution $y\in A$,
\[
 \Dic(A,y)=\langle A,x:x^2=y,\ xax^{-1}=a^{-1}
                         \text{ for all }a\in A\rangle.
\]
The presentation uses multiplicative notation for $A$, while $y$ is given in the additive coordinates of the cyclic factors of $A$. For $n\in\{3,4,5\}$, we use
\[
 D_{2n}=\langle r,s:r^n=s^2=e,\ srs=r^{-1}\rangle,
\]
so that $D_{2n}$ has order $2n$.

The four remaining groups use the following models. Write $\mathrm i$
for the complex number with $\mathrm i^2=-1$ and positive imaginary part:
\begin{align*}
 H_1&=\langle\mathsf X,\mathsf Z,\mathrm i\mathsf X\mathsf Z\rangle
 \leq\operatorname{GL}_2(\mathbb C),&
 \mathsf X&=\begin{pmatrix}0&1\\1&0\end{pmatrix},\quad
 \mathsf Z=\begin{pmatrix}1&0\\0&-1\end{pmatrix},\\
 H_2&=\langle a,b:a^8=b^2=e,\ bab=a^5\rangle,&&\\
 H_3&=(\Cyc{3}^2)\rtimes_{-1}\langle c\rangle,&&c^2=e,\\
 H_4&=\{(u,v,w):u,v,w\in\F_3\},&&
\end{align*}
where the multiplication in $H_4$ is
\[
 (u,v,w)(u',v',w')=(u+u',v+v',w+w'+uv')
\]
for all $u,v,w,u',v',w'\in\F_3$.
The orders of $H_1,H_2,H_3,H_4$ are $16,16,18,27$;
their \texttt{SmallGroup} identifiers are respectively
$[16,13]$, $[16,6]$, $[18,4]$, $[27,3]$, also recorded in
\texttt{expected.json}.

\section*{Methods: use of AI tools and software}
The author acknowledges substantial assistance from two AI systems,
ChatGPT (OpenAI, GPT-6 Astra) and Claude (Anthropic, Claude Opus 5.5,
with Claude Sonnet 5.5 and Claude Haiku 5.5 for minor tasks), in
formulating arguments, drafting and revising the manuscript, and
developing, reviewing, and executing the verification code.
ChatGPT also produced the integrated proof package of this version (the
four-color construction), which the author checked.
Claude contributed to the certificate search, exceptional-set audit,
GAP verification agent and review passes; for this version, in an
automated revision pipeline under the author's direction, it transcribed
and integrated that package, checked the arguments by automated
line-by-line readings, and prepared the source files.
The coherent reconstruction theorem of Section~\ref{sec:coherent} and its finite checks were developed with the same assistance.
The author is responsible for the mathematical statements, proofs,
computational claims, citations, and final presentation.
No AI system is an author.

\section*{Statements and Declarations}
\subsection*{Funding}
No funding was received for this work.

\subsection*{Competing interests}
The author has no competing interests to declare.

\subsection*{Data and code availability}
All data and code supporting the finite computations are in the
ancillary files of this arXiv version: the certificate and expected-value JSON files
(with all color classes of the 24 certificates, which the typeset paper does not print),
the verification and audit scripts, per-partition results, all 96 worker
files, and the execution records of the historical HPC run of October 3, 2026. Their README provides
the reproduction commands, environment details, and checksums.

The ancillary files also hold the programs and run records of the finite checks of Section~\ref{sec:coherent}, which are supporting evidence only; the proof of Theorem~\ref{thm:coherent} does not use them. The certificates and their verification programs are used for the finite exceptional cases of the exact minima (Theorem~\ref{thm:exact}, Section~6).

\providecommand{\bysame}{\leavevmode\hbox to3em{\hrulefill}\thinspace}

\end{document}